\documentclass{memo-l}

\usepackage{amsmath}       
\usepackage{amsthm}        
\usepackage{amssymb}       
\usepackage{amsfonts}
\usepackage[all]{xy}
\usepackage{xspace}
\usepackage{calc}
\usepackage{comment}
\usepackage{pgfplots}
\usepgflibrary{fpu}
\usepackage{mathtools}
\usepackage{tabularx}

\usepackage{ifthen}		

\usepackage{chngcntr}
\counterwithout{footnote}{chapter}

\newcommand\ignore[1]{}

\DeclareRobustCommand{\SkipTocEntry}[5]{}

\usepackage{tikz}
\usetikzlibrary{matrix}
\usetikzlibrary{arrows,backgrounds,patterns,scopes,external,
    decorations.pathreplacing,
    decorations.pathmorphing
}

\newlength{\fuzzwidth}
\newlength{\arrowlength}
\newlength{\arrowwidth}
\newlength{\pointrad}
\newlength{\linewid}
\newlength{\circlerad}
\newlength{\smcirclerad}
\newlength{\smcircleradt}
\newlength{\minirad}
\newcommand{\fillcolor}{black!60}
\newcommand{\fuzzcolor}{black!25}
\newcommand{\arrowcolor}{black!25}
\newcommand{\covercolor}{black!0}
\newcommand{\graycolor}{black!55}
\newcommand{\graylightcolor}{black!40}
\newcommand{\graytextcolor}{black!65}

\newcommand{\evlength}{.45} 
\newcommand{\evlengthv}{1} 
\newcommand{\loopsize}{.2} 

\newcommand{\coverwidthfuzz}{6pt}
\newcommand{\coverwidth}{3.5pt}
\newcommand{\coverwidththin}{3.25pt}
\newcommand{\coverwidththick}{3.75pt}

\newlength{\linewidthin}
\newlength{\linewidthick}
\newlength{\framebasewidth}
\newlength{\framewidth}
\newcommand{\framelengthshort}{.04}
\newcommand{\framelengthlong}{.1}
\newcommand{\seqstep}{1/6}

\newcommand{\framebasecolor}{black!30}

\tikzset{
	coverline/.style={
	preaction={draw,line width=\coverwidth,\covercolor}}, 
	coverlinethin/.style={
	preaction={draw,line width=\coverwidththin,\covercolor}}, 
	coverlinethick/.style={
	preaction={draw,line width=\coverwidththick,\covercolor}}, 
	coverlineleft/.style={
	preaction={draw,line width=\coverwidthfuzz,\covercolor,decorate,decoration={curveto,amplitude=0,raise=.35*\fuzzwidth}}}, 
coverlinelefttail/.style={
	preaction={draw,line width=\coverwidthfuzz,\covercolor,decorate,decoration={curveto,amplitude=0,raise=.35*\fuzzwidth,pre=moveto,pre length=2pt}}}, 
        fuzzlefttail/.style={
        preaction={draw,line width=\fuzzwidth,\fuzzcolor,decorate,decoration={curveto,pre=moveto,pre length=2pt,amplitude=0,raise=.5*\fuzzwidth}}}, 
        linestylethin/.style={line width=\linewidthin},
        linestylethick/.style={line width=\linewidthick},
        linestylegray/.style={line width=\linewid,\graycolor},
        linestylegraylight/.style={line width=\linewid,\graylightcolor}
}
\tikzset{
        fuzzright/.style={
        preaction={draw,line width=\fuzzwidth,\fuzzcolor,decorate,decoration={curveto,amplitude=0,raise=-.5*\fuzzwidth}}},
        fuzzleft/.style={
        preaction={draw,line width=\fuzzwidth,\fuzzcolor,decorate,decoration={curveto,amplitude=0,raise=.5*\fuzzwidth}}},
        fuzzrightpre/.style={ 
        preaction={draw,line width=2pt,\fuzzcolor,decorate,decoration={curveto,amplitude=0,raise=-1pt,pre=moveto,pre length=12pt}}},
        fuzzleftpre/.style={ 
        preaction={draw,line width=2pt,\fuzzcolor,decorate,decoration={curveto,post=moveto,post length=32pt,amplitude=0,raise=1pt}}},        
        outstyle/.style={\arrowcolor, line width=\arrowwidth},
        linestyle/.style={line width=\linewid},
        emptylinestyle/.style={line width=0},
        fuzzrightwide/.style={
        preaction={draw,line width=1.25*\fuzzwidth,\fuzzcolor,decorate,decoration={curveto,amplitude=0,raise=-.5*\fuzzwidth}}}
}

\pgfarrowsdeclare{scalehead}{scalehead}
{
  \pgfarrowsleftextend{-2\pgflinewidth}
  \pgfarrowsrightextend{\pgflinewidth}
}
{
  \pgfsetlinewidth{0.8\pgflinewidth}
  \pgfsetdash{}{0pt}
  \pgfsetroundcap
  \pgfsetroundjoin
  \pgfpathmoveto{\pgfpoint{-2.25\pgflinewidth}{3\pgflinewidth}}
  \pgfpathcurveto
  {\pgfpoint{-2.0625\pgflinewidth}{1.875\pgflinewidth}}
  {\pgfpoint{0pt}{0.1875\pgflinewidth}}
  {\pgfpoint{0.5625\pgflinewidth}{0pt}}
  \pgfpathcurveto
  {\pgfpoint{0pt}{-0.1875\pgflinewidth}}
  {\pgfpoint{-2.0625\pgflinewidth}{-1.875\pgflinewidth}}
  {\pgfpoint{-2.25\pgflinewidth}{-3\pgflinewidth}}
  \pgfusepathqstroke
}

\newcommand{\cb}{\raisebox{.6ex-.5\height}}

\newcommand{\arxiv}[1]{\href{http://arxiv.org/abs/#1}{\tt arXiv:\nolinkurl{#1}}}

\newcommand{\googlebooks}[1]{(preview at \href{http://books.google.com/books?id=#1}{google books})}

\theoremstyle{plain} 
\newtheorem{maintheorem}{Theorem}
\newtheorem{maincor}[maintheorem]{Corollary}
\newtheorem{mainconj}[maintheorem]{Conjecture}
\newtheorem{mainquestion}[maintheorem]{Question}
\newtheorem{theorem}{Theorem}[section]
\newtheorem{lemma}[theorem]{Lemma}
\newtheorem{corollary}[theorem]{Corollary}          
\newtheorem{proposition}[theorem]{Proposition}

\newtheorem{apptheorem}{Theorem}[chapter]

\theoremstyle{definition} 
\newtheorem{definition}[theorem]{Definition}

\newtheorem{appdefinition}[apptheorem]{Definition}

\theoremstyle{remark}  
\newtheorem{remark}[theorem]{Remark}
\newtheorem{example}[theorem]{Example}
\newtheorem{conjecture}[theorem]{Conjecture}
\newtheorem{warning}[theorem]{{Warning}}
\newtheorem{question}[theorem]{{Question}}
\newtheorem*{guide*}{Guide}
\newtheorem*{outline*}{Outline}
\newtheorem*{remarkohc*}{Remark on higher categories and the cobordism hypothesis}
\newtheorem*{assumpfield*}{Assumptions on the base field}

\newtheorem{appexample}[apptheorem]{Example}

\newtheoremstyle{special_statement} 
	{\topskip}
	{\topskip}
	{\addtolength{\leftskip}{2.5em} \itshape }
	{}
	{\bfseries}
	{:}
	{.5em}
	{}
\theoremstyle{special_statement}

\newcommand{\Mod}[2]  
{
  \ifthenelse{\equal{#1}{}}{  			
		\ifthenelse{\equal{#2}{}}		
			{\mathrm{Mod}}{ 			
				{\mathrm{Mod}\textrm{-}#2}		
			}
	}{									
		\ifthenelse{\equal{#2}{}}		
			{{#1\textrm{-}\mathrm{Mod}}}{		
				{{#1\textrm{-}\mathrm{Mod}\textrm{-}#2}}	
			}
	}
}

\renewcommand{\mod}[2]  
{
  \ifthenelse{\equal{#1}{}}{  			
		\ifthenelse{\equal{#2}{}}		
			{\mathrm{mod}}{ 			
				{\mathrm{mod}\textrm{-}#2}		
			}
	}{									
		\ifthenelse{\equal{#2}{}}		
			{{#1\textrm{-}\mathrm{mod}}}{		
				{{#1\textrm{-}\mathrm{mod}\textrm{-}#2}}	
			}
	}
}

\newcommand{\bimod}[3]
{{}_{#1} {#2}_{#3}}

\newcommand{\nid}{\noindent}
\newcommand{\ra}{\rightarrow}

\newcommand{\xra}{\xrightarrow}

\DeclareMathOperator{\Hom}{Hom}
\DeclareMathOperator{\End}{End}
\DeclareMathOperator{\IHom}{\underline{Hom}}
\DeclareMathOperator{\Fun}{Fun}

\let\minusplus\mp
\renewcommand{\mp}{\mathrm{mp}}
\newcommand{\op}{\mathrm{op}}
\newcommand{\mop}{\mathrm{mop}}
\newcommand{\rev}{\mathrm{rev}}
\newcommand{\id}{\mathrm{id}}
\newcommand{\inc}{\mathrm{inc}}

\newcommand{\btimes}{\boxtimes}

\newcommand{\Set}{\mathrm{Set}}

\newcommand{\Tr}{\mathrm{Tr}}
\newcommand{\tr}{\mathrm{tr}}

\newcommand{\Bord}{\mathrm{Bord}}
\newcommand{\FrBord}{\mathrm{Bord}^\mathrm{fr}}
\newcommand{\OrBord}{\mathrm{Bord}^\mathrm{or}}

\newcommand{\Vect}{\mathrm{Vect}}

\newcommand{\Alg}{\mathrm{Alg}}
\newcommand{\Cat}{\mathrm{Cat}}

\newcommand{\TC}{\mathrm{TC}}
\newcommand{\ev}{\mathrm{ev}}
\newcommand{\coev}{\mathrm{coev}}
\newcommand{\TCsep}{{\TC^{\mathrm{sep}}}}
\newcommand{\TCss}{{\TC^{\mathrm{ss}}}}
\newcommand{\Rep}{\mathrm{Rep}}

\newcommand{\RP}{\RR \mathrm{P}}

\newcommand{\pt}{\mathrm{pt}}

\newcommand{\coZ}{\reflectbox{$\cZ$}}

\def\cA{\mathcal A}\def\cB{\mathcal B}\def\cC{\mathcal C}\def\cD{\mathcal D}
\def\cE{\mathcal E}\def\cF{\mathcal F}\def\cG{\mathcal G}
\def\cL{\mathcal L}
\def\cM{\mathcal M}\def\cN{\mathcal N}\def\cP{\mathcal P}
\def\cR{\mathcal R}\def\cS{\mathcal S}\def\cT{\mathcal T}

\def\cZ{\mathcal Z}

\def\CC{\mathbb C}

\def\RR{\mathbb R}

\def\ZZ{\mathbb Z}

\definecolor{CSPcolor}{rgb}{0.0,0.5,0.75}	
\definecolor{NScolor}{rgb}{0.5,0.0,0.5}		
\definecolor{CDcolor}{rgb}{0.8,0.0,0.2}		

\DeclareFontFamily{U}  {MnSymbolA}{}
\DeclareFontShape{U}{MnSymbolA}{m}{n}{
    <-6>  MnSymbolA5
   <6-7>  MnSymbolA6
   <7-8>  MnSymbolA7
   <8-9>  MnSymbolA8
   <9-10> MnSymbolA9
  <10-12> MnSymbolA10
  <12->   MnSymbolA12}{}

\DeclareFontShape{U}{MnSymbolA}{b}{n}{
    <-6>  MnSymbolA-Bold5
   <6-7>  MnSymbolA-Bold6
   <7-8>  MnSymbolA-Bold7
   <8-9>  MnSymbolA-Bold8
   <9-10> MnSymbolA-Bold9
  <10-12> MnSymbolA-Bold10
  <12->   MnSymbolA-Bold12}{}

\DeclareSymbolFont{MnSyA}         {U}  {MnSymbolA}{m}{n}
\SetSymbolFont{MnSyA}       {bold}{U}  {MnSymbolA}{b}{n}

\DeclareMathSymbol{\leftmapsto}{\mathrel}{MnSyA}{42}
\DeclareMathSymbol{\rightmapsto}{\mathrel}{MnSyA}{40}

\usepackage[colorlinks=true, linkcolor=black, citecolor=black, urlcolor=black, hypertexnames=false
	]{hyperref}

\begin{document}

\frontmatter

\title{Dualizable tensor categories}

\author{Christopher L. Douglas}
\address{Mathematical Institute\\ University of Oxford\\ Oxford OX2 6GG\\ United Kingdom}
\email{cdouglas@maths.ox.ac.uk}
\urladdr{http://www.christopherleedouglas.com}
      	
\author{Christopher Schommer-Pries}
\address{Department of Mathematics\\ Max Planck Institute for Mathematics \\ 53111 Bonn \\ Germany}
\email{schommerpries.chris.math@gmail.com}
\urladdr{http://sites.google.com/site/chrisschommerpriesmath}

\author{Noah Snyder}
\address{Department of Mathematics\\ Indiana University\\ Bloomington, IN 47401\\ USA}
\email{nsnyder@math.columbia.edu}
\urladdr{http://www.math.columbia.edu/\!\raisebox{-1mm}{~}nsnyder/}


\subjclass[2010]{57R56, 18D10, 55U30, 16D90 (Primary), 57M27, 17B37, 18E10 (Secondary)}
\keywords{Tensor category, fusion category, bimodule category, dualizable, topological field theory, local field theory, pivotal, spherical, framing, combing, 3-manifold, Serre automorphism, Radford equivalence}

\maketitle	

\setcounter{tocdepth}{2}
\tableofcontents

\begin{abstract}
We investigate the relationship between the algebra of tensor categories and the topology of framed 3-manifolds.  On the one hand, tensor categories with certain algebraic properties determine topological invariants.  We prove that fusion categories of nonzero global dimension are 3-dualizable, and therefore provide 3-dimensional 3-framed local field theories.  We also show that all finite tensor categories are 2-dualizable, and yield categorified 2-dimensional 3-framed local field theories.  On the other hand, topological properties of 3-framed manifolds determine algebraic equations among functors of tensor categories.  We show that the 1-dimensional loop bordism, which exhibits a single full rotation, acts as the double dual autofunctor of a tensor category.  We prove that the 2-dimensional belt-trick bordism, which unravels a double rotation, operates on any finite tensor category, and therefore supplies a trivialization of the quadruple dual.  This approach produces a quadruple-dual theorem for suitably dualizable objects in any symmetric monoidal 3-category.  There is furthermore a correspondence between algebraic structures on tensor categories and homotopy fixed point structures, which in turn provide structured field theories; we describe the expected connection between pivotal tensor categories and combed fixed point structures, and between spherical tensor categories and oriented fixed point structures.
\end{abstract}

\chapter*{Acknowledgments}

We thank Ben Balsam, Dan Freed, Mike Hopkins, Alexander Kirillov Jr., Scott Morrison, Stephan Stolz, Constantin Teleman, Dylan Thurston, Alexis Virelizier, and Kevin Walker for helpful and clarifying conversations.  We made intensive use of Google Wave in our collaboration on this project, and we would like to thank the Wave development team for their efforts and commiserate with them on Wave's untimely demise.

We are especially grateful to Pavel Etingof for several inspiring and fruitful conversations; to Andr\'e Henriques for general enlightenment and specific ideas concerning the 3-category of tensor categories and descent calculations for local field theories; to Victor Ostrik for a key suggestion concerning separability and fusion categories; and to Peter Teichner for discussions concerning the homotopy type of structure groups.

CD was partially supported by a Miller Research Fellowship and by EPSRC grant EP/K015478/1, CSP was partially supported by NSF fellowship DMS0902808 and by the Max Planck Institute for Mathematics, and NS was partially supported by NSF fellowship DMS-0902981 and by DARPA grant HR0011-11-1-0001.

\mainmatter

\tikzset{external/force remake}


\chapter*{Introduction}

\renewcommand{\thesection}{I.{\arabic{section}}}

\section{Local topological field theory}

Quantum field theories associate to a manifold, thought of as the underlying physical space of a system, a vector space of quantum field states on that manifold.  A given field will evolve in time, and for each time interval the evolution of all fields together provides an operator on the vector space of field states; these operators compose associatively under concatenation of time intervals.  A distinctive property of field theories is that the vector space of states is multiplicative: the space of states on a disjoint union of two manifolds is the tensor product of the spaces of states on the individual manifolds.  Atiyah and Segal abstracted this situation into the formal notion of a topological (quantum) field theory: an $n$-dimensional topological field theory is a symmetric monoidal functor from the category of $(n-1)$-dimensional manifolds and their bordisms to the category of vector spaces~\cite{MR1001453,Segal}.  Such a topological field theory assigns a vector space to each $(n-1)$-manifold, and a linear operator to each bordism between such manifolds, such that gluing of bordisms corresponds to composition of operators.  Examples of topological field theories include, in dimension 3, Turaev--Viro theories associated to spherical fusion categories~\cite{MR1191386, MR1292673} and Reshetikhin--Turaev theories associated to modular (braided) tensor categories~\cite{MR990772,MR1091619}; and in dimension 4, more or less conjecturally, Donaldson Floer theories for principal bundles~\cite{floer,donaldsonym,atiyah88,witten88}, Crane--Frenkel theories associated to Hopf categories~\cite{cranefrenkel}, monopole Floer homology~\cite{kronmrowka}, Heegaard Floer homology~\cite{os04,os06}, and Khovanov homology~\cite{khovanov}.

A topological field theory provides a numerical invariant of closed $n$-manifolds: by the symmetric monoidal assumption, the theory assigns the standard 1-dimen\-sional vector space to the empty manifold, and so to a closed $n$-manifold assigns an endomorphism of that vector space, which is a scalar.  Indeed, topological field theories can be viewed as numerical invariants of closed $n$-manifolds that have a particularly computable structure---the invariant can be determined by cutting the closed manifold into manifolds with boundary, and then composing the operators associated to those smaller manifolds.  Unfortunately, the operators associated to these manifolds with boundary, and even the vector spaces associated to the boundary $(n-1)$-manifolds, may themselves be difficult to compute.  This situation motivates the notion of extended topological field theories.  A once-extended topological field theory also assigns algebraic invariants to $(n-2)$-manifolds, to $(n-1)$-manifolds with boundary, and to $n$-manifolds with codimension-2 corners, in such a way that the values of the original field theory can be reconstructed by composing these invariants when the manifolds are glued together.  Of course difficulties may again arise in computing the invariants of these $(n-2)$-, $(n-1)$-, and $n$-manifolds, and so we are forced to further extend the theory.  Altogether a \emph{local topological field theory} is fully extended in that it assigns invariants to $i$-manifolds with corners of any codimension, for all $i$ between $0$ and $n$, again in a way respecting all possible gluing operations.  More specifically, a local topological field theory is a symmetric monoidal functor from the $n$-category of $0$-, $1$-, $2$-, ..., $n$-manifold bordisms with corners to a (typically algebraic) target $n$-category.

The invariants of local field theories are computable, by construction, but the locality property also provides a means of classifying these theories.  A physical field theory is determined by its local behavior, that is, its operation on arbitrarily small regions of space-time.  Similarly (provided we restrict attention to local topological field theories on framed manifolds) a local topological field theory is determined by its values on discs and in fact these values are encoded in the algebraic invariant the field theory assigns to a single point.  However, not every object of the target $n$-category is an allowable value for this algebraic invariant of a point; rather, there is a restrictive algebraic condition called \emph{full dualizability} (or \emph{$n$-dualizability}) that ensures an object of an $n$-category extends to a consistent system of invariants providing a local field theory.  Altogether this is the content of the \emph{cobordism hypothesis}, a classification result conjectured by Baez--Dolan~\cite{MR1355899} and proven by Hopkins--Lurie~\cite{lurie-ch}: an $n$-dimensional local framed topological field theory is determined by its value on a point, and any fully dualizable object of a symmetric monoidal $n$-category provides a local framed topological field theory whose point-value is that object.

\section{Three-dimensional topology and three-dimensional algebra}

We will concentrate on 3-dimensional local topological field theories, and therefore we need a 3-dimensional algebraic structure that can serve as a target for such field theories.  The simplest nontrivial candidate for such a structure is \emph{monoidal linear categories}; we might hope there is a 3-category whose objects, morphisms, 2-morphisms, and 3-morphisms are respectively monoidal linear categories, bimodule categories, bimodule functors, and bimodule natural transformations.  In order to ensure there really is a 3-category, we restrict attention to finite tensor categories (that is, finite rigid monoidal linear categories), finite bimodule categories, bimodule functors, and bimodule natural transformations---this symmetric monoidal 3-category $\TC$ exists by a construction of Johnson-Freyd--Scheimbauer~\cite{jfs}.  (Tensor categories are by no means the only possible 3-dimensional algebraic target for field theory, just the simplest.  For instance, Ben-Zvi--Nadler~\cite{0904.1247} consider 2-dimensional local field theories with target an $(\infty,3)$-category of monoids in differential graded linear categories; in work in progress, Freed--Teleman~\cite{FT} consider 3-dimensional field theories with a target 3-category of modules over braided tensor categories; in a rather different direction, Bartels--Douglas--Henriques~\cite{bdh} study 3-dimensional field theories with target the 3-category of conformal nets.)

\subsection{From algebra to topology}

Traditionally, local field theory is seen as a means for taking algebraic gadgets and producing topological invariants.  From that perspective, the task is to find fully dualizable objects of the target, in our case the 3-category $\TC$ of finite tensor categories, and then apply the cobordism hypothesis to obtain corresponding local field theories.  If the base field is characteristic zero, we prove that every finite semisimple tensor category is fully dualizable; if the base field is algebraically closed of finite characteristic, we prove that every fusion category of nonzero global dimension is fully dualizable.  (A fusion category is a finite semisimple tensor category with simple unit.)  In either case, we therefore have an associated Turaev--Viro-style 3-dimensional framed local field theory.  This construction provides a plethora of new 3-dimensional field theories (for instance in finite characteristic, over non-algebraically closed fields, for non-spherical categories, and for categories with non-simple unit).  It also establishes a conceptual framework for the relationship between the algebraic structure of the tensor category (for instance its pivotality or sphericality) and the structure group of the manifolds in the associated field theory (for instance combed or oriented)---see the section on future directions below.  

In addition to the full dualizability of semisimple finite tensor categories, we investigate the partial dualizability of non-semisimple finite tensor categories.  We prove that a non-semisimple finite tensor category is never fully, that is 3-, dualizable, but we establish that it is always 2-dualizable.  There is therefore an associated categorified 2-dimensional framed local topological field theory for every finite tensor category.  (This theory is categorified in the sense that it assigns vector spaces, rather than numbers, to closed 2-manifolds.)  In fact, these categorified 2-dimensional theories extend to take values on certain non-closed 3-manifolds---these non-closed 3-manifold invariants themselves deserve further investigation.  In the next subsection, we take the rather different perspective that we can transport our knowledge of the topology of framed 3-manifolds across these field theory invariants to establish features of the algebra of finite tensor categories.

\subsection{From topology to algebra}

Given an $n$-dimensional topological field theory, along with two distinct decompositions of the same $k$-manifold with corners, applying the field theory to these decompositions yields an algebraic equation in the target category.  Said another way, given an equivalence between two $k$-manifolds, the field theory provides an equivalence between the two seemingly distinct algebraic operations corresponding to the manifolds.

We apply this approach with the aforementioned partial 3-dimensional 3-framed local topological field theory $\cF_\cC$ associated to any finite tensor category $\cC$.  (Here ``3-framed" means that the manifolds are equipped with trivializations of the stabilizations of their tangent bundles up to dimension 3.)  The simplest nontrivial 3-framed manifold is the following 1-dimensional interval, called ``the loop bordism":
\begin{center}
\begin{tikzpicture}
\draw[linestyle] 
(.7,0) to [out=180, in=-20] (0,.1)
	to [looseness=1.6, out=160, in=180] (0,.4)
	to [looseness=1.6, out=0, in=20] (0,.1)
	to [out=-160, in=0] (-.7,0);
\end{tikzpicture}
\end{center}
(Comparing a normal framing of this immersion with the blackboard framing of the paper provides the bordism with a 2-framing, and therefore a 3-framing by stabilization.)  We prove that the field theory invariant of this bordism is the $\cC$-$\cC$ bimodule obtained by twisting the left action on the identity bimodule by the right double dual functor.  If we were studying 2-dimensional field theory, there would be little more to say, as the loop bordism has infinite order in the 2-framed bordism category.  But there is a 3-framed surface, called ``the belt bordism" trivializing the square of the loop bordism: 
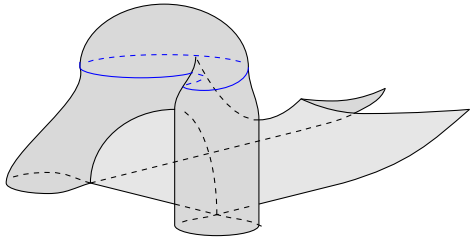
\begin{center}
	\begin{tikzpicture}[
			yscale=0.28, xscale=0.56,
			decoration={border, 
				segment length = 4pt, 	
				amplitude = 2pt, 		
				angle = 0  				 
				}, 
			contour line/.style={thin, blue}
				]
				
			\colorlet{surfacecolor1}{black!15}
			\colorlet{surfacecolor2}{black!10}	
			
			\fill [color = surfacecolor1] (0, 2.5) .. controls (0, 4) and (2,6.5) .. (1.75,8)
				arc (180:0:2cm and 3cm)
				to [out = 270, in = 90] (6, 5.5)
				-- (6,0.5)
				arc (360: 180: 1cm and 0.5cm)
				-- (4, 6)
				arc (90:180:2cm and 3.5cm)
				to [out = 210, in = 0] (1,2)
				arc (270: 180: 1cm and 0.5cm);
			
			\fill [color = surfacecolor1] (7, 6.5) parabola bend (8.5, 5.8) (8.25, 5.8) 
				to [out = 45, in = 260] (9, 7)
				parabola bend (7.75, 6.4) (7, 6.5);
			
			\fill [color = surfacecolor2] (4, 6) arc (90:180:2cm and 3.5cm)
				-- (4, 1.5) -- (4,6);
			
			\fill [color = surfacecolor2] (6, 5.5) parabola (7, 6.5) parabola bend (8.5, 5.8) (11, 6)
				to [out = 240, in = 30] (8, 2.5) -- (6, 1.5) -- (6, 5.5);

			\draw (0, 2.5) .. controls (0, 4) and (2,6.5) .. (1.75,8)
				arc (180:0:2cm and 3cm)
				to [out = 270, in = 90] (6, 5.5)
				-- (6,0.5)
				arc (360: 180: 1cm and 0.5cm)
				-- (4, 6)
				arc (90:180:2cm and 3.5cm)
				to [out = 210, in = 0] (1,2)
				arc (270: 180: 1cm and 0.5cm);
			
			\draw (2,2.5) -- (4,1.5) 
				decorate {-- (5,1) to [out = -30, in = 135] (6,0.5)
					  (4,0.5) to [out = 45, in = 210] (5,1) -- (6,1.5)} 
				-- (8, 2.5) to [out = 30, in = 240] (11, 6);
			\draw decorate { (5, 1) arc (0:90:1cm and 5cm)};
			\draw decorate {(2,2.5) -- (7,5) to [out = 30, in = 225] (8.25, 5.8)} to [out = 45, in = 260] (9, 7) ;
			\draw decorate {(2,2.5) to [out = 150, in = 0] (1,3) arc (90:180:1cm and 0.5cm)};
			\draw (4, 6) to [out = 90, in = 270] (4.5, 8.5);
			\draw (6, 5.5) decorate { parabola (4.5, 8.5)} (6, 5.5) parabola (7, 6.5)
				parabola bend (7.75, 6.4) (9, 7);
			\draw (7, 6.5) parabola bend (8.5, 5.8) (11, 6);
			
		\draw [contour line] (1.75,8) arc (180:320:1.45cm and 0.5cm)
			decorate {to [out = 30,in = 135](4.65, 7.6) to [out = -45, in = 45] (4.2, 7.1) }
			arc (180:270:0.5cm and 0.2cm) to [out = 0, in = -90] (5.75,8)
		;
		\draw [contour line] decorate {(1.75,8) arc (180:0:2cm and 0.6cm)};
				
	\end{tikzpicture}
\end{center}	
(The 3-framing is determined by the comparison of a normal framing of the surface with the blackboard framing of the ambient 3-dimensional Euclidean space.)  Applying the field theory to this belt bordism provides a trivialization of the bimodule associated to the right quadruple dual functor; this recovers the Etingof--Nikshych--Ostrik generalization to finite tensor categories of Radford's quadruple antipode formula for finite-dimensional Hopf algebras~\cite{MR2097289, MR2183279, MR0407069}.  When the tensor category is semisimple, the bimodule trivialization reduces to a monoidal trivialization of the quadruple dual functor itself.  Altogether then, the Dirac belt trick provides an elegant, transparent topological explanation and proof of the a priori rather opaque quadruple dual theorem for finite tensor categories.  In fact, our topological-field-theoretic approach provides a generalization of the quadruple dual theorem for finite tensor categories to any sufficiently dualizable object of any symmetric monoidal 3-category; we give a precise definition of the relevant dualizability condition and call the resulting objects ``Radford objects".

We anticipate that this general method of transporting topological equivalences into algebraic equivalences will be especially effective in dimension 4, where nontrivial relationships among 4-framed 2- and 3-manifolds will reveal novel properties of braided tensor categories and more generally of monoidal 2-categories.

\section{Results}

\subsection{On 3-dualizability}

Dualizability is a strong finiteness condition.  A vector space $V \in \Vect$ is 1-dualizable if it is finite-dimensional.  All algebras $A \in \Alg$ are 1-dualizable; an algebra $A$ is 2-dualizable if it is finite-dimensional and $A$ is projective as an $A$--$A$-bimodule.  This projectivity condition is called separability.  In order to identify 3-dualizable objects in the 3-category $\TC$ of finite tensor categories, we need to impose separability-type conditions on bimodule categories and therefore also on tensor categories themselves.  We say that a finite semisimple tensor category $\cC$ over a perfect field is \emph{separable} if the identity $\cC$--$\cC$-bimodule category $\cC$ can be expressed as the category of modules for a separable algebra object within the tensor category $\cC \boxtimes \cC^\mp$.  (Here, $\cC^\mp$ is the category $\cC$ with the opposite monoidal structure, and $- \boxtimes -$ denotes the Deligne tensor product, that is the linear category corepresenting bilinear functors.)  This separability condition is satisfied by any finite semisimple tensor category over a field of characteristic zero, and also by any fusion category of nonzero global dimension over any algebraically closed field.

Equipped with this notion of separability, we can state our primary result:
\begin{maintheorem} \label{thm1}
Separable tensor categories are fully dualizable.
\end{maintheorem}
\nid This is established by Theorem~\ref{thm:TC-dualizable} and Corollary~\ref{cor:septcisdualizable} in the text.  This result provides a profusion of field theories:
\begin{maincor} \label{cor2}
For any separable tensor category, there is a 3-dimensional 3-framed local topological field theory whose value on a point is that tensor category.  In particular there is such a field theory for any finite semisimple tensor category over a field of characteristic zero, and such a field theory for any fusion category of nonzero global dimension over an algebraically closed field of finite characteristic.
\end{maincor}
\nid This is recorded as Corollaries~\ref{cor:3dtft}, \ref{cor:charzerotft}, and \ref{cor:fusiontft}.  Though these field theories are certainly Turaev--Viro-type theories, they are not directly comparable to existing theories for two reciprocal reasons: our theories take values on 3-framed, not oriented, manifolds, and they do not depend on a choice or even the existence of a spherical structure.  Turaev--Viro invariants for spherical categories were originally constructed by Turaev--Viro~\cite{MR1191386, MR1292673}, Ocneanu~\cite{MR1317353}, and Barrett--Westbury~\cite{bw-invariants,MR1686423}; a modern treatment of these invariants as a once-extended field theory is given by Balsam--Kirillov~\cite{1004.1533,1010.1222,1012.0560} and as a local field theory by Walker~\cite{kw:tqft} in the hybrid topological-categorical context of disc-like $n$-categories defined by Morrison--Walker~\cite{1009.5025}.  These constructions all provide oriented invariants, but there is also a literature on framed and spin 3-manifold invariants and Reshetikhin--Turaev-style once-extended spin field theories associated to semisimple spin modular (spherical) tensor categories~\cite{MR1117149, MR1171303, MR1387228, MR1880321}.

In fact, separability precisely captures the dualizability condition for finite tensor categories:
\begin{maintheorem} \label{thm3}
Fully dualizable finite tensor categories are separable.
\end{maintheorem}
\nid This is proven as Theorem~\ref{thm:converse}.  We also show, in Corollary~\ref{cor:maxfd}, that separable tensor categories, finite semisimple bimodule categories, bimodule functors, and bimodule transformations form the maximal 3-category of finite tensor categories in which all objects have duals and all morphisms and 2-morphisms have adjoints.  This provides a classification of Turaev--Viro-style local field theories with surface, line, and point defects, which is reminiscent of existing results on defects in 2-dimensional conformal field theory~\cite{ffrs-duality,frs-fusion} and 3-dimensional topological field theory~\cite{kapustinsaulina,kitaevkong,fsv}.

\subsection{On categorified 2-dimensional field theories}

The separability condition on a tensor category is needed only in establishing the very last portion of 3-dualizability.  As all algebras are 1-dualizable, so too all finite tensor categories are certainly 1-dualizable.  (Indeed, if we had available a 3-category of not-necessarily-finite tensor categories, we would still expect all such tensor categories to be 1-dualizable.)  The condition of 2-dualizability of a tensor category is by contrast a substantive restriction.  We prove that all finite bimodule categories between finite tensor categories have left and right adjoints (given for the bimodule ${}_\cC \cM_\cD$ by the functor category linear duals $\Fun_\cC(\cM,\cC)$ and $\Fun_\cD(\cM,\cD)$ respectively) and therefore that finite tensor categories are 2-dualizable.  (Though finiteness is sufficient for 2-dualizability, in a hypothetical 3-category of more general tensor categories, we would find it is more than is necessary.)  Indeed, though not 3-dualizable, finite tensor categories are better than 2-dualizable: some of the bimodule functors witnessing adjunctions between relevant bimodule categories themselves have adjoints.  We formalize this intermediate notion of dualizability as follows: a 2-dualizable object $x$ of a symmetric monoidal 3-category is \emph{Radford} if, for the evaluation map $\ev_x$ witnessing the 1-dualizability of $x$, the unit and counit of the adjunction $\ev_x \dashv \ev_x^R$ themselves both admit right adjoints.  Though it may appear obscure or technical, the Radford condition is in fact a precise algebraic analog of the geometric structure of the surface implementing the Dirac belt trick.

We can now give our main result about finite tensor categories:
\begin{maintheorem} \label{thm4}
Finite tensor categories are 2-dualizable.  In fact, they are Radford objects of the 3-category of tensor categories.
\end{maintheorem}
\nid This is proven in Theorems~\ref{thm:TC_is_2Dualizable} and~\ref{thm:TCisRadford}.  This result provides 2-dimensional field theories that are categorified in the sense that they assign vector spaces to closed surfaces:
\begin{maincor} \label{cor5}
Associated to every finite tensor category is a categorified 2-dimensional 2-framed local topological field theory.  Moreover, this theory extends to a 2-dimensional 3-framed field theory.
\end{maincor}
\nid This is proven as Corollaries~\ref{cor:2dualtft} and~\ref{cor:3fr2d}.  Here a 3-framed theory is one where the manifolds have their tangent bundles trivialized after stabilization to dimension 3.  Note that a categorified field theory is a much more refined invariant than an ordinary field theory.  For instance, it provides an action of the mapping class group on the vector space associated to any closed surface.  In particular, a categorified 2-dimensional field theory encodes the data of the classical notion of a modular functor~\cite{Segal, MR1159969,MR1797619}.  The above corollary is not completely satisfactory: the fact that a finite tensor category is Radford shows that the corresponding field theory takes values not only on 3-framed 2-manifolds but also on many non-closed 3-framed 3-manifolds---we have not yet endeavored to specify the precise class of manifolds allowed.  The main antecedents regarding field theories associated to finite tensor categories are the Kuperberg 3-framed and the Hennings--Kauffman--Radford oriented manifold invariants associated to a not-necessarily semisimple Hopf algebra~\cite{MR1394749,hennings,kauffrad}; Kerler and Lyubashenko's investigation of once-extended partial field theories associated to not-necessarily semisimple, modular tensor categories~\cite{MR1862634}; and Fuchs and Schweigert's work on logarithmic conformal field theory~\cite{fuchsschweigert}.

\subsection{On quadruple duals}

The loop bordism depicted above can be decomposed into two half-loops, which may be identified respectively as an evaluation map (witnessing the 1-dualizability of a point in the bordism category) and an adjoint to that evaluation map.  (This decomposition is depicted in Figure~\ref{fig:serrebordism} in Section~\ref{sec:Serre}.)  Because any finite tensor category is 2-dualizable, each of these half-loops has an image under the associated field theory, and those images may be computed algebraically in the target category.  The result of that computation yields the fundamental connection between the topology of framed manifolds and the algebra of tensor categories:
\begin{maintheorem} \label{thm6}
The field theory $\cF_\cC$ of a finite tensor category $\cC$ takes the loop bordism to the bimodule associated to the right double dual functor $(-)^{**}: \cC \ra \cC$.
\end{maintheorem}
\nid This appears as Theorem~\ref{thm:serreisdoubledual}.  Here the bimodule associated to a functor is the one obtained from the identity bimodule ${}_\cC \cC_\cC$ by twisting the left $\cC$-action by the functor.

The belt bordism trivializing the square of the loop bordism, also depicted above, can itself be decomposed into two pieces, each of which can be identified as a certain adjoint or witness to an adjunction in the bordism category.  (This decomposition is depicted in Figure~\ref{fig:Radford_bordism}.)  The Radford property of a finite tensor category ensures that the associated field theory takes values on these pieces (and, quite importantly, on the 3-manifolds witnessing that the belt bordism is invertible), and therefore provides an algebraic trivialization of the bimodule associated to the quadruple dual.  That trivialization in turn controls the quadruple dual functor itself:
\begin{maincor} \label{cor7}
In any finite tensor category, there is an invertible object $D$ and a monoidal natural isomorphism between the quadruple dual functor $(-)^{****}: \cC \ra \cC$ and the conjugation functor $D \otimes (-) \otimes D^{-1}: \cC \ra \cC$.  In any semisimple finite tensor category with absolutely simple unit, the quadruple dual functor is monoidally naturally isomorphic to the identity.
\end{maincor}
\nid This is recorded as Theorem~\ref{thm:quaddual} and Corollary~\ref{cor:ssquaddual}.  Here `absolutely simple' means that the unit remains simple after arbitrary base change of the base field.  This result, generalizing Radford's theorem for finite-dimensional Hopf algebras~\cite{MR0407069}, is originally due to Etingof--Nikshych--Ostrik~\cite{MR2097289, MR2183279}; our approach realizes this algebraic result as a direct corollary of the topological fact that the fundamental group of the 3-dimensional orthogonal group is order 2.  The earliest connection between Radford's theorem and the topology of 3-framings occurs in Kuperberg's construction of framed manifold invariants from finite-dimensional Hopf algebras~\cite{MR1394749}.  The intuition that there ought to be a relationship between the belt trick and the quadruple dual theorem is implicit in Hagge--Hong~\cite{MR2559711} and explicit in Bartlett~\cite{0901.3975}.  (Note that, despite our exposition in terms of field theories, our proof does not depend on the cobordism hypothesis---we can and indeed do directly construct the algebraic images of the geometric components of the loop bordism and the belt bordism.)

For any 2-dualizable object $x$ of a symmetric monoidal $n$-category $\cT$, we can compose the evaluation map of $x$ with its right adjoint, according to the geometry of the loop bordism---this procedure is drawn in Figure~\ref{fig:serrebordism}; the resulting composite is an automorphism of $x$, which is called the Serre automorphism.  Similarly, for any Radford object $x \in \cT$, we can compose appropriate adjoints, units, and counits according to the geometry of the belt bordism---this procedure is drawn in Figure~\ref{fig:Radford_bordism}; the resulting composite is a 2-morphism of $\cT$, which is called the Radford map.  Because our approach is topological in essence, it does not rely on constructions specific to any one algebraic context, and so provides a generalization of the quadruple dual theorem to any algebraic 3- (or indeed higher) category:
\begin{maintheorem} \label{thm8}
For any Radford object $x$ in a symmetric monoidal $n$-category, with $n$ at least 3, the Radford map is an equivalence and so provides a trivialization of the square of the Serre automorphism of $x$.
\end{maintheorem}
\nid This is proven as Theorem~\ref{thm:Cat_Radford} and Corollary~\ref{cor:serreinvserre}.

\subsection{On tensor and bimodule categories}

Our work relies extensively on the theory of finite tensor categories developed by Etingof, Gelaki, Nikshych, and Ostrik~\cite{egno-book, EO-ftc, MR2183279, MR2097289,0909.3140}.  In order to have a robust 3-categorical setting for studying the algebra of finite and separable tensor categories, we prove a number of results extending the work of EGNO, some of which are of interest independent from our topological-field-theoretic purposes.

The most important ingredient in defining a 3-category of tensor categories is the composition of bimodule categories:
\begin{maintheorem} \label{thm9}
For finite bimodule categories ${}_\cB \cM_\cC$ and ${}_\cC \cN_\cD$ over finite tensor categories, the relative Deligne tensor product bimodule ${}_\cB \cM \boxtimes_\cC \cN_\cD$ exists.  Provided the base field is perfect, there is a symmetric monoidal 3-category $\TC$ of finite tensor categories, finite bimodule categories, bimodule functors, and bimodule transformations.
\end{maintheorem}
\nid The two parts of this result are quoted in Theorem~\ref{thm:DelignePrdtOverATCExists} and Theorem~\ref{thm:tcexists} and are established respectively in~\cite{BTP} and~\cite{jfs}.  Here the relative Deligne tensor product $- \boxtimes_\cC -$ is the linear category corepresenting $\cC$-balanced right exact bilinear functors.  

When proving that finite tensor categories are Radford, we will need to construct adjoints of certain bimodule functors, and to do that we leverage Etingof and Ostrik's theory of exact module categories~\cite{EO-ftc}.  Exact module categories have the distinguishing feature that module functors out of them are always exact, and that exactness facilitates the construction of adjoints.  We establish that this property is preserved by Deligne tensor:
\begin{maintheorem} \label{thm10}
For finite module categories ${}_\cC \cM$ and ${}_\cD \cN$ over finite tensor categories, the Deligne tensor product ${}_{\cC \boxtimes \cD} (\cM \boxtimes \cN)$ is an exact module category.
\end{maintheorem}
\nid This is proven as Theorem~\ref{thm:tensor-exactness}.

Recall that a finite semisimple tensor category $\cC$ over a perfect field is called separable if the module category ${}_{\cC \boxtimes \cC^\mp} \cC$ can be expressed as modules for a separable algebra object in $\cC \boxtimes \cC^\mp$.  More generally, any finite module category ${}_\cC \cM$ is called separable if it can be expressed as modules for a separable algebra object in $\cC$.  We characterize separable modules and separable tensor categories:
\begin{maintheorem} \label{thm11}
A finite module category ${}_\cC \cM$ over a semisimple tensor category is separable if and only if the functor category $\Fun_\cC(\cM,\cM)$ is semisimple.  A finite tensor category over a perfect field is separable if and only if its Drinfeld center is semisimple.
\end{maintheorem}
\nid This is established as Theorem~\ref{thm:SepModCats} and Corollary~\ref{cor:Sep=semisimplecenter}.  This result provides a generalization to arbitrary perfect fields and non-spherical categories of M\"uger's theorem that the Drinfeld center of a spherical finite semisimple tensor category, of nonzero global dimension, is semisimple~\cite{MR1966525}.

Completely crucial to our proof of the full dualizability of separable tensor categories is the fact that separable bimodules compose:
\begin{maintheorem} \label{thm12}
For separable bimodule categories ${}_\cB \cM_\cC$ and ${}_\cC \cN_\cD$ over finite semisimple tensor categories, the relative Deligne tensor product ${}_\cB \cM \boxtimes_\cC \cN_\cD$ is separable.  Over a perfect field, there is a symmetric monoidal 3-category $\TCss$ of finite semisimple tensor categories, separable bimodule categories, bimodule functors, and bimodule transformations.
\end{maintheorem}
\nid This is proven as Theorem~\ref{thm:compositeOfSep} and Corollary~\ref{cor:septc}.  Because relative Deligne tensor products can always be reexpressed as functor categories, and vice versa, the first part of this result is a generalization to arbitrary fields of Etingof--Nikshych--Ostrik's theorem that a functor category between finite semisimple module categories is semisimple~\cite{MR2183279}.

In the case of a fusion category $\cC$ over an algebraically closed field, we can explicitly express the module category ${}_{\cC \boxtimes \cC^\mp} \cC$ as modules over a Frobenius algebra object, and use this description to characterize the separability condition in this case; in characteristic zero, the global dimension condition disappears:
\begin{maintheorem} \label{thm13}
A fusion category over an algebraically closed field is separable if and only if its global dimension is nonzero.  Any finite semisimple tensor category over an arbitrary field of characteristic zero is separable.
\end{maintheorem}
\nid This is proven in Theorem~\ref{thm:NonzeroDimension} and Corollary~\ref{cor:charzerosep}.  In light of the fact that tensor categories with semisimple center are separable, the first part of this result sharpens the theorem of Brugui\`eres--Virelizier that pivotal fusion categories with semisimple center have nonzero global dimension~\cite{MR3079759}.

\section{Outlook}

Thus far, we have only considered the framed field theories associated to dualizable tensor categories.  The cobordism hypothesis also classifies field theories with structure groups weaker than a framing, in terms of dualizable objects equipped with the structure of homotopy fixed points.  The existence of these homotopy fixed point structures corresponds to the existence of certain algebraic structures on tensor categories, most notably pivotal and spherical structures.  Here and in our final Section~\ref{sec:spherical} we preview this correspondence, but we defer a detailed treatment to future work.

The cobordism hypothesis provides an $SO(k)$ action on the space of $k$-dualizable objects of any symmetric monoidal $n$-category.  (Here, the `space of $k$-dualizable objects' means the maximal sub-$n$-groupoid on those objects.)  This action arises as follows: the space of such objects is homotopy equivalent to the space of $k$-dimensional $k$-framed local field theories, and there is an $SO(k)$ action on those field theories by rotating the framing.  In particular, we obtain an $SO(2)$ action on the space of finite tensor categories, and an $SO(3)$ action on the space of separable tensor categories.  We can therefore look for finite tensor categories with $SO(2)$ homotopy fixed point structures, which would provide oriented categorified 2-dimensional field theories, or for separable tensor categories with $SO(3)$ homotopy fixed point structures, which would provide oriented 3-dimensional field theories.

The $SO(2)$ action on finite tensor categories is simple to describe on objects: the 1-cell of $SO(2)$ takes a tensor category $\cC \in \TC$ to the Serre automorphism of $\cC$.  Recall that the Serre automorphism is the bimodule associated to the double dual.  By definition, a \emph{pivotal} structure on a finite tensor category is a monoidal trivialization of the double dual.  A priori, that trivialization does not quite provide the tensor category with the structure of an $SO(2)$ homotopy fixed point, but it does provide the weaker structure of an $\Omega S^2$ homotopy fixed point, where $\Omega S^2$ acts via the map $\Omega S^2 \ra \Omega BSO(2) \simeq SO(2)$.  Nevertheless:
\begin{mainconj}
Every pivotal fusion category in characteristic zero admits the structure of an $SO(2)$ homotopy fixed point, and therefore provides the structure of a combed 3-dimensional local field theory.
\end{mainconj}
\nid (A combed 3-manifold is one equipped with a nonvanishing vector field.)

A pivotal tensor category has a chosen trivialization of the double dual; by squaring that trivialization, there is a chosen trivialization of the quadruple dual, and therefore of the bimodule associated to the quadruple dual.  But any finite tensor category has a canonical trivialization of that quadruple dual bimodule, namely the Radford equivalence.  A \emph{spherical} structure on a finite tensor category (with absolutely simple unit) is a pivotal structure whose square is the Radford equivalence.  When the tensor category is semisimple, this notion agrees with the classical notion of sphericality (a condition on the quantum trace of the pivotal structure), but this notion of sphericality as a square root of the Radford equivalence is the correct generalization to non-semisimple finite tensor categories.  

The fact that all finite tensor categories have a Radford equivalence trivializing the square of the Serre automorphism provides an extension of the $\Omega S^2$ action on $\TC$ to an action of $\Omega \Sigma \RP^2$.  A spherical structure trivializes the Serre in a manner compatible with the Radford, and so provides the tensor category with the structure of an $\Omega \Sigma \RP^2$ homotopy fixed point structure.  Note that the inclusion $\RP^2 \ra \RP^3 \simeq SO(3)$ provides a map $\Omega \Sigma \RP^2 \ra \Omega \Sigma SO(3) \ra \Omega B SO(3) \simeq SO(3)$.  The structure of an $\Omega \Sigma \RP^2$ fixed point together with the structure of an $SO(2)$ fixed point (provided by the above conjecture) is nearly enough to provide a full $SO(3)$ fixed point structure.  Indeed:
\begin{mainconj}
Every spherical fusion category in characteristic zero admits the structure of an $SO(3)$ homotopy fixed point, and therefore provides the structure of an oriented 3-dimensional local field theory.
\end{mainconj}

It is a well-known open problem to determine whether all fusion categories admit pivotal or spherical structures.  In light of the above conjectures, we have the following topological analog of that problem:
\begin{mainquestion}
Does every 3-dimensional 3-framed local field theory with target tensor categories descend to a combed field theory?  Does every 3-dimensional 3-framed local field theory with target tensor categories descend to an oriented field theory?
\end{mainquestion}

\section{Overview}

\begin{guide*}
Experts in tensor categories could safely skim Sections~\ref{sec:tc-lincat} and \ref{sec:tc-exact}.  Experts in field theory could safely skim Sections~\ref{sec:notation}, \ref{sec:framed-duality}, and \ref{sec:df-objects}.  Those readers content to restrict to characteristic zero can skip Sections~\ref{sec:tc-separable} and \ref{sec:tc-fusion}.  Readers content to restrict to semisimple tensor categories could just glance at the statements of Section \ref{sec:radfordftc}.  Most readers will want to merely skim Sections~\ref{sec:tensorexact} and \ref{sec:computeradford}, which are more technical.  The expert geodesic path through the book is therefore Sections~\ref{sec:Serre}, \ref{sec:Radford}, \ref{sec:conventions}, \ref{sec:tc-bimodules}, \ref{sec:df-morphisms}, and \ref{sec:separableisfd}.  The extreme minimalist, no doubt discontinuous, path is simply Sections~\ref{sec:Radford}, \ref{sec:tc-bimodules}, and \ref{sec:separableisfd}.
\end{guide*}

\begin{outline*}
Chapter~\ref{sec:lft} concerns the algebraic structure present in the geometry of composition in the 3-framed bordism category; it begins by recalling the notions of 1-, 2-, and 3-dualizability and the statement of the cobordism hypothesis.  Section~\ref{sec:notation} describes an immersion notation for $n$-framed manifolds with corners, and gives various examples in low dimensions.  Section~\ref{sec:framed-duality} completely describes the duals and adjoints arising in the 2-framed bordism 2-category.  Section~\ref{sec:Serre} discusses the loop bordism (also called the Serre bordism) and provides a decomposition of this bordism into pieces describable in terms of elementary dualization data (that is adjoints of witnesses of dualities, and the like).  Section~\ref{sec:Radford} discusses the belt bordism (also called the Radford bordism), giving a decomposition of this bordism in terms of elementary dualization data, and determining the precise dualization structure needed to produce an algebraic analog of the belt trick in an arbitrary symmetric monoidal 3-category.

Chapter~\ref{sec:tc} investigates tensor categories, bimodule categories, tensor products of bimodule categories, exact module categories, dual module categories, categories of module functors, separable module categories, and separable tensor categories, among other things.  Section~\ref{sec:conventions} fixes consistent conventions, notation, and terminology for duals, adjoints, tensors, composition, and bimodules.  Section~\ref{sec:tc-lincat} describes linear, finite, monoidal, rigid, and tensor categories, along with module categories, functors, and transformations; it then discusses balanced (also known as relative Deligne) tensor products of module categories over finite tensor categories and the 3-category of finite tensor categories.  Section~\ref{sec:tc-exact} reviews the theory of exact module categories over finite tensor categories and proves that exactness is preserved by Deligne tensor.  Section~\ref{sec:tc-bimodules} first discusses the flip, the twist, and the dual of a bimodule category; it then shows that the dual bimodule category is equivalent to a functor dual category; next it proves, crucially, that the relative Deligne tensor product is equivalent to a category of module functors; finally, it explains how to interpret a dual module category as a module category over a double dual algebra.  Section~\ref{sec:tc-separable} introduces the theory of separable bimodules and separable tensor categories---in finite characteristic this theory replaces that of semisimple bimodules and semisimple tensor categories; in particular this section proves that the composite of separable bimodule categories is separable, and therefore there is a 3-category of separable tensor categories with morphisms the separable bimodule categories.  Section~\ref{sec:tc-fusion} proves that fusion categories of nonzero global dimension (so in particular all fusion categories in characteristic zero) are separable.

Chapter~\ref{sec:dualizability} analyzes the dualizability of finite and separable tensor categories.  Section~\ref{sec:df-objects} shows that tensor categories are 1-dualizable and discusses the associated twice categorified 1-dimensional field theory.  Section~\ref{sec:df-morphisms} establishes that the functor dual provides an adjoint bimodule category and therefore finite tensor categories are 2-dualizable; it then describes the associated categorified 2-dimensional field theory, showing in particular that the loop bordism is taken to the double dual bimodule.  Section~\ref{sec:radfordftc} proves that finite tensor categories, though not 3-dualizable, are `Radford objects', that is are dualizable enough to have an algebraic belt trick; the quadruple dual theorem follows immediately.  Section~\ref{sec:separableisfd} proves that separable tensor categories are 3-dualizable and therefore have associated 3-dimensional field theories; it also shows that 3-dualizable finite tensor categories must be separable.  Section~\ref{sec:spherical} previews the relationships between pivotal and spherical structures on tensor categories and the descent properties of the corresponding field theories.  The Appendix describes in more detail the notion of $k$-dualizability and recalls a precise form of both the framed and structured versions of the cobordism hypothesis.
\end{outline*}

\begin{remarkohc*}
For expositional convenience and clarity, we phrase some of our results in terms of symmetric monoidal 3-categories.  Though the theory of higher categories is by now on completely solid ground~\cite{luriehtt,lurieinf2,rezkcart,barkan,unicity,MR2742424,simpson,bergrezk,1308.3574}, it remains unduly technical.  To avoid those technicalities in this book, we do not pick a particular model for higher categories---every statement we make is true in any model.  
Moreover, with the exception of Corollaries~\ref{cor2} and~\ref{cor5}, every one of our main results can be reexpressed purely in terms of 2-categories (the theory of which has been long settled), and so does not depend on any theory of higher categories.  (This reexpression transforms a statement about a 3-category into an interlinked collection of statements about 2-categories of morphisms and the homotopy 2-category whose 2-morphisms are isomorphism classes of invertible 2-morphisms of the 3-category.)

In our presentation, we also freely employ the cobordism hypothesis, but we would like to emphasize that with the exception of the corollaries concerning local field theory, none of our results, about dualizability, about quadruple duals, or about tensor categories, depend on the cobordism hypothesis.  Specifically, Corollaries~\ref{cor2} and \ref{cor5} do depend on the cobordism hypothesis, but Theorems~\ref{thm1}, \ref{thm3}, \ref{thm4}, \ref{thm6}, \ref{thm8}, \ref{thm9}, \ref{thm10}, \ref{thm11}, \ref{thm12}, and \ref{thm13}, and Corollary~\ref{cor7} do not.
\end{remarkohc*}

\begin{assumpfield*}
We have endeavored to limit the restrictions on the base field, in particular to avoid the assumption of characteristic zero or algebraic closure.  The existence of the symmetric monoidal structure on the 3-category of tensor categories depends on the base field being perfect, and so perfection is assumed for all the results about dualizability of tensor categories.  (Recall that all fields of characteristic zero, all algebraically closed fields, and all finite fields are perfect, so we view it as a comparatively mild restriction.)  More specifically, Theorems~\ref{thm1} and \ref{thm3}, and the second halves of Theorems~\ref{thm4}, \ref{thm9}, \ref{thm11}, and \ref{thm12}, along with Corollaries~\ref{cor2}, \ref{cor5}, and \ref{cor7} all assume the field is perfect; as stated Theorem~\ref{thm6} and the first half of Theorem~\ref{thm4} also assume perfection, but for those results it is not essential.  The first half of Theorem~\ref{thm13} assumes algebraic closure, and the second half of Theorem~\ref{thm13} assumes characteristic zero.  That leaves Theorems~\ref{thm8} and \ref{thm10}, and the first halves of Theorems~\ref{thm9}, \ref{thm11}, and \ref{thm12}, which apply over an arbitrary field.  

The standing assumptions in the text are as follows.  No fields appear in the more geometric Chapter~\ref{sec:lft}.  In Sections~\ref{sec:conventions} through \ref{sec:tc-separable}, concerning tensor categories, the base field is arbitrary unless otherwise noted.  (A perfection assumption is made for part (6) of Theorem~\ref{thm:DelignePrdtOverATCExists}, and for Theorem~\ref{thm:tcexists}, Corollary~\ref{cor:septc}, Definition~\ref{def:sepcat}, Corollary~\ref{cor:Sep=semisimplecenter}, Proposition~\ref{prop:SSModuleCatsAreSep}, and Corollary~\ref{cor:tcsepexists}.)  In Section~\ref{sec:tc-fusion} on fusion categories, the base field is algebraically closed except where noted.  (Corollaries~\ref{cor:charzerosep} and~\ref{cor:charzeromodulesep} apply over an arbitrary field of characteristic zero.)  In Chapter~\ref{sec:dualizability} on dualizability, the base field is always perfect. (We furthermore require characteristic zero in Corollary~\ref{cor:charzerotft} and algebraically closed characteristic zero in Corollary~\ref{cor:fusiontft}.)  In the final Section~\ref{sec:descconj} on descent conjectures, we restrict to algebraically closed fields of characteristic zero.
\end{assumpfield*}

\renewcommand{\thesection}{\arabic{chapter}.{\arabic{section}}}

\chapter{The algebra of 3-framed bordisms} \label{sec:lft}

The cobordism hypothesis \cite{MR1355899, lurie-ch} classifies local field theories in terms of sufficiently dualizable objects of target higher categories.  Recall that an object $x$ of a symmetric monoidal $(\infty,3)$-category is called ``1-dualizable" if there is a dual object $\overline{x}$, with the duality between $x$ and $\overline{x}$ witnessed by evaluation and coevaluation maps $\ev: x \otimes \overline{x} \ra 1$ and $\coev: 1 \ra \overline{x} \otimes x$.  (Note well that the zigzag equations relating $\ev$ and $\coev$ are only required to hold up to equivalence, and there is no coherence condition on those equivalences.)  The object is called ``2-dualizable" if moreover the evaluation map is part of an infinite chain of adjunctions 
$$\cdots \dashv \ev^{LL} \dashv \ev^L \dashv \ev \dashv \ev^R \dashv \ev^{RR} \dashv \cdots$$
and similarly the coevaluation map is part of an infinite chain of adjunctions 
$$\cdots \dashv \coev^{LL} \dashv \coev^L \dashv \coev \dashv \coev^R \dashv \coev^{RR} \dashv \cdots.$$  
(Again, note that for each of these adjunctions, the zigzag equations are only required to hold up to isomorphism, and there are no coherence conditions on those isomorphisms.)
The object is called ``3-dualizable" if, moreover, for every adjunction $(F \dashv G, u_{F,G}, v_{F,G})$ in each of the two aforementioned infinite chains, it happens that the unit $u_{F,G}$ is part of an infinite chain of adjunctions 
$$\cdots \dashv u_{F,G}^{LL} \dashv u_{F,G}^L \dashv u_{F,G} \dashv u_{F,G}^R \dashv u_{F,G}^{RR} \dashv \cdots$$ 
and similarly the counit $v_{F,G}$ is part of an infinite chain of adjunctions 
$$\cdots \dashv v_{F,G}^{LL} \dashv v_{F,G}^L \dashv v_{F,G} \dashv v_{F,G}^R \dashv v_{F,G}^{RR} \dashv \cdots.$$  These definitions suffice to make sense of our main theorem, that every separable tensor category is 3-dualizable.  

It will be more convenient for us to use an alternative, equivalent formulation of dualizability focusing on ``$k$-dualizable $3$-categories" rather than ``$k$-dualizable objects".  A symmetric monoidal $3$-category $\cC$ is called ``$1$-dualizable" if every object has a dual---as before, here `dual' means dual in the symmetric monoidal $1$-category of objects of $\cC$ and equivalence classes of morphisms of $\cC$.  The symmetric monoidal $3$-category $\cC$ is called ``$2$-dualizable" if it is 1-dualizable and moreover every $1$-morphism has a left and a right adjoint---as before, here `adjoint' means adjoint in the $2$-category of objects, morphisms, and isomorphism classes of 2-morphisms of $\cC$.  The symmetric monoidal $3$-category $\cC$ is called ``$3$-dualizable" if it is $1$-dualizable and $2$-dualizable and if every $2$-morphism of $\cC$ has a left and a right adjoint.  An object of a symmetric monoidal $3$-category is ``$k$-dualizable" if it is in some $k$-dualizable subcategory of $\cC$.  (There are corresponding notions of $k$-dualizable objects of an $n$-category, for any $n$---see the Appendix.)

Equipped with the notion of $k$-dualizability, we can state the cobordism hypothesis classification of local field theories: \vspace{7pt}

\setlength{\leftskip}{.75cm}

\nid {\bfseries Cobordism hypothesis:} \emph{$n$-dimensional local framed topological field theories with target a symmetric monoidal $(\infty,n)$-category $\cC$ are in one-to-one correspondence with the $n$-dualizable objects of $\cC$; in fact, the space of such field theories is homotopy equivalent to the space of $n$-dualizable objects of $\cC$.} \vspace{7pt}

\setlength{\leftskip}{0cm}

\nid Note that an $n$-dualizable $(\infty,n)$-category is also called a ``fully dualizable" $(\infty,n)$-category, and similarly an $n$-dualizable object of an $(\infty,n)$-category is also called a ``fully dualizable" object.  By definition, an $n$-dimensional local framed topological field theory is a symmetric monoidal functor out of the $(\infty,n)$-category $\FrBord_n$ of $n$-framed bordisms.  (A local topological field theory is also known as a ``fully extended" topological field theory.)  In the next section we describe the notion of $n$-framing in question.  We refer the reader to the Appendix for a more detailed and precise review of the notion of full dualizability and statement of the cobordism hypothesis, and to the exposition~\cite{MR2994995} for a thorough motivation for and discussion of the cobordism hypothesis perspective on local field theory.

\section{$n$-framed manifolds and $n$-framed bordisms} \label{sec:notation}

The cobordism hypothesis classifies field theories whose source is a bordism category of $n$-framed manifolds.  In this section, we describe a convenient notation for describing $n$-framed manifolds using normally framed immersions, and explain how this notation behaves on boundaries and corners and thereby produces bordism categories.  We conclude with various low dimensional examples, many of which will also be essential in our later dualizability calculations.

\subsection{$n$-framings from normally framed immersions}

An $n$-framed $k$-manifold $(M,\tau)$ is, by definition, a $k$-manifold $M$ equipped with a trivialization $\tau$ of $TM \oplus \RR^{n-k}$, the $(n-k)$-fold stabilization of the tangent bundle of $M$.  A convenient way to present an $n$-framing on a $k$-manifold $M$ is to give an immersion $\iota: M \looparrowright \RR^n$ together with a normal framing, that is a trivialization $\phi$ of the normal bundle $\nu(\iota)$ of the immersion.  The normally framed immersion $(\iota, \phi)$ provides an $n$-framing of $M$ by the composite
\[
TM \oplus \RR^{n-k} \cong TM \oplus \nu(\iota) \cong \RR^n.
\]
Here the second isomorphism is provided by the standard, ``blackboard" framing of the sum of the tangent bundle and normal bundle of $M$.

Throughout this book, except in one noted instance, we will draw $n$-framed manifolds using this normally framed immersion notation; we leave completely implicit the induced $n$-framing.  When the immersion is codimension-1, we will usually specify the normal framing by a unidirectional gray corona on the immersed manifold.

\begin{example} \label{eg:frcircles}
The following infinite list includes, up to homotopy, all normally framed immersed circles in $\RR^2$:
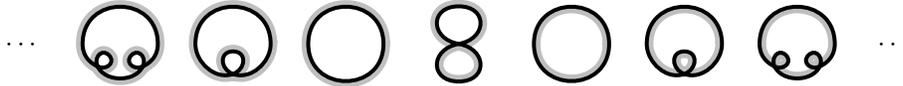
\[\cdots\quad
\cb{
\begin{tikzpicture}
{ [xshift=-3cm]
\draw[linestyle,fuzzleft]
(.5,0) to [out=-90, in=25] (.3,-.3)
	to [looseness=1.6, out=-155, in=225] (.15,-.15)
	to [looseness=1.6, out=45, in=65] (.3,-.3)
	to [out=245, in=-65] (-.3,-.3)
	to [looseness=1.6, out=115, in=135] (-.15,-.15)
	to [looseness=1.6, out=-45, in=-25] (-.3,-.3)
	to [out=155, in=-90] (-.5,0)
	.. controls (-.5,.66) and (.5,.66) .. (.5,0);	
}
{ [xshift=-1.5cm]
\draw[linestyle,fuzzleft]
(.5,0) to [out=-90, in=-20] (0,-.4)
	to [looseness=1.6, out=160, in=180] (0,-.1)
	to [looseness=1.6, out=0, in=20] (0,-.4)
	to [out=-160, in=-90] (-.5,0)
	.. controls (-.5,.66) and (.5,.66) .. (.5,0);
}
{ [xshift=0cm]
\draw[linestyle,fuzzleft]
(.5,0) .. controls (.5,-.66) and (-.5,-.66) .. (-.5,0)
	.. controls (-.5,.66) and (.5,.66) .. (.5,0);
}
{ [xshift=1.5cm]
\draw[linestyle,fuzzleft,looseness=2]
(0,.5) to [out=0, in=10] (0,0)
	to [out=-170, in=180] (0,-.5)
	to [out=0, in=-10] (0,0)
	to [out=170, in=180] (0,.5);
}
{ [xshift=3cm]
\draw[linestyle,fuzzright]
(.5,0) .. controls (.5,-.66) and (-.5,-.66) .. (-.5,0)
	.. controls (-.5,.66) and (.5,.66) .. (.5,0);
}
{ [xshift=4.5cm]
\draw[linestyle,fuzzright]
(.5,0) to [out=-90, in=-20] (0,-.4)
	to [looseness=1.6, out=160, in=180] (0,-.1)
	to [looseness=1.6, out=0, in=20] (0,-.4)
	to [out=-160, in=-90] (-.5,0)
	.. controls (-.5,.66) and (.5,.66) .. (.5,0);
}
{ [xshift=6cm]
\draw[linestyle,fuzzright]
(.5,0) to [out=-90, in=25] (.3,-.3)
	to [looseness=1.6, out=-155, in=225] (.15,-.15)
	to [looseness=1.6, out=45, in=65] (.3,-.3)
	to [out=245, in=-65] (-.3,-.3)
	to [looseness=1.6, out=115, in=135] (-.15,-.15)
	to [looseness=1.6, out=-45, in=-25] (-.3,-.3)
	to [out=155, in=-90] (-.5,0)
	.. controls (-.5,.66) and (.5,.66) .. (.5,0);	
}
\end{tikzpicture}
}
\quad \cdots
\]
As described above, each such immersion specifies a $2$-framed circle.\footnote{The resulting $2$-framed circles can be directly depicted as follows.  Orient each normally framed immersed circle so that the orientation followed by the normal framing is a positive frame of the plane.  Identify the immersed circle with the standard circle $C := [0,1] / 0\!\!\sim\!\!1$, by an orientation-preserving diffeomorphism taking the top point of the immersed circle to $0 \in C$.  A $2$-framing of $C$ is a trivialization of $TC \oplus \RR = C \times \RR^2$; we may draw this trivialization as a sequence of frames in $\RR^2$.  The middle three immersed circles above correspond respectively to the following three sequences of frames.  Here the shorter bar is the first frame vector, and the longer bar is the second frame vector.  Note that the middle sequence of frames is homotopic to the constant framing.
\[
\cb{
\begin{tikzpicture}[line cap=rect,scale=2.25]
{ [xshift=-1.5cm]
\draw[line width=\framebasewidth,\framebasecolor] (0,0) to (1,0);
\draw[line width=\framewidth] (0,0) -- (0:\framelengthshort);
\draw[line width=\framewidth] (0,0) -- (90:\framelengthlong);
\draw[line width=\framewidth] (0,0) ++(\seqstep,0) -- ++(60:\framelengthshort);
\draw[line width=\framewidth] (0,0) ++(\seqstep,0) -- ++(150:\framelengthlong);
\draw[line width=\framewidth] (0,0) ++(2*\seqstep,0) -- ++(120:\framelengthshort);
\draw[line width=\framewidth] (0,0) ++(2*\seqstep,0) -- ++(210:\framelengthlong);
\draw[line width=\framewidth] (0,0) ++(3*\seqstep,0) -- ++(180:\framelengthshort);
\draw[line width=\framewidth] (0,0) ++(3*\seqstep,0) -- ++(270:\framelengthlong);
\draw[line width=\framewidth] (0,0) ++(4*\seqstep,0) -- ++(240:\framelengthshort);
\draw[line width=\framewidth] (0,0) ++(4*\seqstep,0) -- ++(330:\framelengthlong);
\draw[line width=\framewidth] (0,0) ++(5*\seqstep,0) -- ++(300:\framelengthshort);
\draw[line width=\framewidth] (0,0) ++(5*\seqstep,0) -- ++(30:\framelengthlong);
\draw[line width=\framewidth] (0,0) ++(6*\seqstep,0) -- ++(0:\framelengthshort);
\draw[line width=\framewidth] (0,0) ++(6*\seqstep,0) -- ++(90:\framelengthlong);
}
{ [xshift=0cm]
\draw[line width=\framebasewidth,\framebasecolor] (0,0) to (2,0);
\draw[line width=\framewidth] (0,0) -- (0:\framelengthshort);
\draw[line width=\framewidth] (0,0) -- (90:\framelengthlong);
\draw[line width=\framewidth] (0,0) ++(\seqstep,0) -- ++(60:\framelengthshort);
\draw[line width=\framewidth] (0,0) ++(\seqstep,0) -- ++(150:\framelengthlong);
\draw[line width=\framewidth] (0,0) ++(2*\seqstep,0) -- ++(120:\framelengthshort);
\draw[line width=\framewidth] (0,0) ++(2*\seqstep,0) -- ++(210:\framelengthlong);
\draw[line width=\framewidth] (0,0) ++(3*\seqstep,0) -- ++(180:\framelengthshort);
\draw[line width=\framewidth] (0,0) ++(3*\seqstep,0) -- ++(270:\framelengthlong);
\draw[line width=\framewidth] (0,0) ++(4*\seqstep,0) -- ++(120:\framelengthshort);
\draw[line width=\framewidth] (0,0) ++(4*\seqstep,0) -- ++(210:\framelengthlong);
\draw[line width=\framewidth] (0,0) ++(5*\seqstep,0) -- ++(60:\framelengthshort);
\draw[line width=\framewidth] (0,0) ++(5*\seqstep,0) -- ++(150:\framelengthlong);
\draw[line width=\framewidth] (0,0) ++(6*\seqstep,0) -- ++(0:\framelengthshort);
\draw[line width=\framewidth] (0,0) ++(6*\seqstep,0) -- ++(90:\framelengthlong);
\draw[line width=\framewidth] (0,0) ++(7*\seqstep,0) -- ++(-60:\framelengthshort);
\draw[line width=\framewidth] (0,0) ++(7*\seqstep,0) -- ++(30:\framelengthlong);
\draw[line width=\framewidth] (0,0) ++(8*\seqstep,0) -- ++(-120:\framelengthshort);
\draw[line width=\framewidth] (0,0) ++(8*\seqstep,0) -- ++(-30:\framelengthlong);
\draw[line width=\framewidth] (0,0) ++(9*\seqstep,0) -- ++(-180:\framelengthshort);
\draw[line width=\framewidth] (0,0) ++(9*\seqstep,0) -- ++(-90:\framelengthlong);
\draw[line width=\framewidth] (0,0) ++(10*\seqstep,0) -- ++(-120:\framelengthshort);
\draw[line width=\framewidth] (0,0) ++(10*\seqstep,0) -- ++(-30:\framelengthlong);
\draw[line width=\framewidth] (0,0) ++(11*\seqstep,0) -- ++(-60:\framelengthshort);
\draw[line width=\framewidth] (0,0) ++(11*\seqstep,0) -- ++(30:\framelengthlong);
\draw[line width=\framewidth] (0,0) ++(12*\seqstep,0) -- ++(0:\framelengthshort);
\draw[line width=\framewidth] (0,0) ++(12*\seqstep,0) -- ++(90:\framelengthlong);
}
{ [xshift=2.5cm]
\draw[line width=\framebasewidth,\framebasecolor] (0,0) to (1,0);
\draw[line width=\framewidth] (0,0) -- (180:\framelengthshort);
\draw[line width=\framewidth] (0,0) -- (270:\framelengthlong);
\draw[line width=\framewidth] (0,0) ++(\seqstep,0) -- ++(120:\framelengthshort);
\draw[line width=\framewidth] (0,0) ++(\seqstep,0) -- ++(210:\framelengthlong);
\draw[line width=\framewidth] (0,0) ++(2*\seqstep,0) -- ++(60:\framelengthshort);
\draw[line width=\framewidth] (0,0) ++(2*\seqstep,0) -- ++(150:\framelengthlong);
\draw[line width=\framewidth] (0,0) ++(3*\seqstep,0) -- ++(0:\framelengthshort);
\draw[line width=\framewidth] (0,0) ++(3*\seqstep,0) -- ++(90:\framelengthlong);
\draw[line width=\framewidth] (0,0) ++(4*\seqstep,0) -- ++(-60:\framelengthshort);
\draw[line width=\framewidth] (0,0) ++(4*\seqstep,0) -- ++(30:\framelengthlong);
\draw[line width=\framewidth] (0,0) ++(5*\seqstep,0) -- ++(-120:\framelengthshort);
\draw[line width=\framewidth] (0,0) ++(5*\seqstep,0) -- ++(-30:\framelengthlong);
\draw[line width=\framewidth] (0,0) ++(6*\seqstep,0) -- ++(-180:\framelengthshort);
\draw[line width=\framewidth] (0,0) ++(6*\seqstep,0) -- ++(-90:\framelengthlong);
}
\end{tikzpicture}
}
\]
}
In fact all $2$-framed circles are specified by such an immersed circle.  Suppose we have a $2$-framing $\tau: TM \oplus \RR \xra{\cong} \RR^2$ on a closed connected 1-manifold $M$.  Let $\jmath$ denote the orientation of $M$ such that for any point $p \in M$, the pair $(\tau(\jmath,0),\tau(0,1))$ is a positive frame of $\RR^2$.  This orientation provides another $2$-framing of $M$, namely $\sigma: TM \oplus \RR \xra{\jmath \oplus \id} \RR^2$.  The ratio $\sigma/\tau$ is a map $M \ra SO(2)$.  Because $M$ is oriented, by $\jmath$, the set of homotopy classes of maps $M \ra SO(2)$ is canonically identified with the integers.  We therefore have a $\ZZ$-valued invariant of such $2$-framed 1-manifolds, and in fact this procedure produces a bijection from the $2$-framed diffeomorphism classes of $2$-framed closed connected 1-manifolds to the integers---the above picture provides a representative from each diffeomorphism class.
\end{example}

\begin{remark}
Not every $n$-framed $n$-manifold can be specified using the immersion notation.  For example, the circle has a (unique up to diffeomorphism) $1$-framing, but it cannot be immersed in $\RR^1$.  The Hirsch--Smale immersion theorem \cite{MR0119214, MR0105117} ensures, though, that every $n$-framing of an $(n-k)$-manifold can be realized by a normally framed immersion, provided either $k>0$ or each component of the manifold is not closed.
\end{remark}

An $m$-framed $k$-manifold $(M,\tau)$ can be stabilized to an $n$-framed manifold $(M,\tau \oplus \id_{\RR^{n-m}})$.  A normally framed immersion $(\iota: M \looparrowright \RR^m, \phi : \nu(\iota) \simeq \RR^{m-k})$ can similarly be stabilized to the normally framed immersion $(\inc_{\RR^m \ra \RR^n} \circ \iota, \phi \oplus \id_{\RR^{n-m}})$.  The association of an $m$-framed manifold to a normally framed immersion is evidently compatible with these stabilization procedures.

\begin{example}
Though the 1-framing on the circle cannot be represented by a normally framed immersion, its stabilization to a 2-framing can: it is represented by the figure 8 immersion in Example~\ref{eg:frcircles}.
\end{example}

\subsection{$n$-framings with boundary and corners}

The procedure described above for associating an $n$-framed manifold to a normally framed immersion works as stated also for manifolds with boundary or manifolds with corners.  

When the immersed manifold is a bordism, we can fix conventions for how a normal framing (therefore $n$-framing) on the bulk induces a normal framing (therefore $n$-framing) on the source and target boundary.  For a bordism $M$ with boundary but without corners, each boundary component is labelled ``in" or ``out" according to whether it is part of the source or target of the bordism.  Suppose $(\iota, \phi)$ is a normally framed immersion of $M$.  An incoming boundary component $N \subset (\partial M)_{\textrm{in}}$ inherits the structure of a normally framed immersion: the immersion is simply the restriction of the immersion of $M$, and the framing is $(l,s) \subset \nu(N,\RR^n)$, where $l$ is a section of the normal bundle of $N$ pointing into the bulk manifold $M$, and $s \subset \nu(M,\RR^n)$ is the given normal frame of $M$.  Similarly an outgoing boundary component inherits the normal framing $(-l,s)$, that is the first normal frame vector points out of the bulk manifold.  

Given a higher bordism (that is a manifold $M$ with cuspidal corners representing a higher morphism in a bordism $n$-category), equipped with a normally framed immersion, the incoming and outgoing boundaries (which are now codimension-0 submanifolds of the boundary of $M$) inherit normal framings exactly as described for ordinary bordisms.  Furthermore, the boundaries of these boundaries inherit normal framings by the same procedure.  Iterating this process provides consistently defined normal framings to every corner of the higher bordism $M$.  We provide various examples of such manifolds and the associated induced framings in the next section.

In drawing $n$-framed manifolds with boundary and corners, we need to specify not only the immersion and the normal framing, but also which parts of the boundary are incoming and which outgoing.  Outgoing pieces of the boundary will be indicated by small arrows pointing out of the bulk of the manifold; incoming pieces of the boundary will be undecorated---implicitly the arrows would point into the bulk.  When the immersion is codimension zero, the outgoing boundary arrows may be replaced by a gray corona, which serves to directly record the induced normal framing of those parts of the boundary; for incoming boundaries, the implicit gray corona, recording the normal framing of those parts of the boundary, is again covered by the bulk and so cannot be seen.  At corners, a combination of these indications will be used; for instance, an arrow together with a corona on a codimension-2 corner indicates respectively the first and second vectors of the induced framing on the corner.

\subsection{Low-dimensional examples of $n$-framed bordisms}

\begin{example} \label{eg-framenot0}
The following four pictures specify, respectively, a $0$-framed, a $1$-framed, and two $2$-framed 0-manifolds:
\[
\begin{tikzpicture}
\filldraw (0,0) circle (\pointrad);
\filldraw (1,0) circle (\pointrad); 
\begin{pgfonlayer}{background}
\draw[->,outstyle] (1,0) -- +(-\arrowlength,0);
\end{pgfonlayer}
\filldraw (2,0) circle (\pointrad);
\begin{pgfonlayer}{background}
\draw[->,outstyle] (2,0) -- +(45:\arrowlength) node[anchor=south,inner sep=2pt] {\tiny 1};
\draw[->,outstyle] (2,0) -- +(135:\arrowlength) node[anchor=south,inner sep=2pt] {\tiny 2};
\end{pgfonlayer}
\filldraw (3,0) circle (\pointrad); 
\begin{pgfonlayer}{background}
\draw[->,outstyle] (3,0) -- +(45:\arrowlength) node[anchor=south,inner sep=2pt] {\tiny 2};
\draw[->,outstyle] (3,0) -- +(135:\arrowlength) node[anchor=south,inner sep=2pt] {\tiny 1};
\end{pgfonlayer}
\end{tikzpicture}
\]
\end{example}

\begin{example}
Here is a picture of a $1$-framed $1$-manifold bordism from a 1-framed point to a 1-framed point:
\[
\begin{tikzpicture}
\draw[linestyle] (0,0) -- (1.5,0);
\begin{pgfonlayer}{background}
\draw[->,outstyle] (0,0) -- +(180:\arrowlength);
\end{pgfonlayer}
\path
    ([shift={(-5\pgflinewidth,-5\pgflinewidth)}]current bounding box.south west)
    ([shift={( 5\pgflinewidth, 5\pgflinewidth)}]current bounding box.north east);
\end{tikzpicture}
\]
\nid The right point is incoming and the left point is outgoing.  Both boundary points inherit framings isomorphic to the framing of the second picture in Example~\ref{eg-framenot0}.

Next we have a picture of a $2$-framed $1$-manifold bordism from two 2-framed points to the empty set:
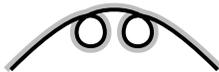
\[
\begin{tikzpicture}
\draw[linestyle,fuzzleft] 
(0,0) .. controls (1.1,1.1) and (1.25,.65) .. (1.25,.5)
	.. controls (1.25,.25) and (.875,.25) .. (.875,.5)
	.. controls (.875,.9) and (1.875,.9) .. (1.875,.5)
	.. controls (1.875,.25) and (1.5,.25) .. (1.5,.5)
	.. controls (1.5,.65) and (1.65,1.1) .. (2.75,0);
\end{tikzpicture}
\]
Both boundary points are incoming, and the induced framings on these points are, left to right respectively, the last two points pictured in Example~\ref{eg-framenot0}.  Note that corona in this example specifies the trivialization of the normal bundle.
\end{example}

\begin{example} \label{ex:disk_bordism_immersed}
The following is a picture of a $2$-framed $2$-manifold bordism from the empty set to a 2-framed circle:
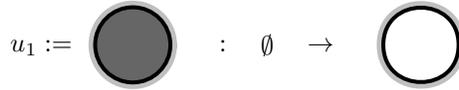
\[
u_1 := \cb{
\begin{tikzpicture}
\filldraw[linestyle,fuzzright,fill=\fillcolor] (0,0) circle (.5);
\end{tikzpicture}
}
\quad
:
\quad
\cb{$\emptyset$}
\quad \ra \quad
\cb{
\begin{tikzpicture}
\draw[linestyle,fuzzleft]
(.5,0) .. controls (.5,-.66) and (-.5,-.66) .. (-.5,0)
	.. controls (-.5,.66) and (.5,.66) .. (.5,0);
\end{tikzpicture}
}
\]
Note that in the picture of the 2-manifold $u_1$, the corona indicates that the boundary is outgoing.  That boundary is the circle with the outward trivialization of its normal bundle, and so the two uses of the corona are consistent, as mentioned previously.
\end{example}

\begin{example} \label{ex:saddle_bordism_immersed}
We now provide an example of a 2-framed 2-manifold with cuspidal corners, representing a 2-morphism in the 2-category of 2-framed 0-, 1-, and 2-manifolds:
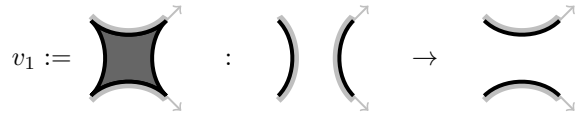
\[
v_1 := \cb{
\begin{tikzpicture}
\filldraw[linestyle,fill=\fillcolor] 
	(0,0) .. controls (.25,.25) and (.75,.25) .. (1,0)
		.. controls (.75,.25) and (.75,.75) .. (1,1)
		.. controls (.75,.75) and (.25,.75) .. (0,1)
		.. controls (.25,.75) and (.25,.25) .. (0,0);
\draw[linestyle,fuzzright]
	(0,0) .. controls (.25,.25) and (.75,.25) .. (1,0);
\draw[linestyle,fuzzleft]
	(0,1) .. controls (.25,.75) and (.75,.75) .. (1,1);
\begin{pgfonlayer}{background}
	\draw[->,outstyle] (1,1) -- +(45:\arrowlength);
	\draw[->,outstyle] (1,0) -- +(-45:\arrowlength);
\end{pgfonlayer}
\end{tikzpicture}
}
\quad
: 
\quad
\cb{
\begin{tikzpicture}
\draw[linestyle,fuzzright]
	(0,0) .. controls (.25,.25) and (.25,.75) .. (0,1);
\draw[linestyle,fuzzleft]
	(1,0) .. controls (.75,.25) and (.75,.75) .. (1,1);
\begin{pgfonlayer}{background}
	\draw[->,outstyle] (1,1) -- +(45:\arrowlength);
	\draw[->,outstyle] (1,0) -- +(-45:\arrowlength);
\end{pgfonlayer}
\end{tikzpicture}
}
\quad
\ra
\quad
\cb{
\begin{tikzpicture}
\draw[linestyle,fuzzright]
	(0,0) .. controls (.25,.25) and (.75,.25) .. (1,0);
\draw[linestyle,fuzzleft]
	(0,1) .. controls (.25,.75) and (.75,.75) .. (1,1);
\begin{pgfonlayer}{background}
	\draw[->,outstyle] (1,1) -- +(45:\arrowlength);
	\draw[->,outstyle] (1,0) -- +(-45:\arrowlength);
\end{pgfonlayer}
\end{tikzpicture}
}
\]
The source and target of the bordism are the pairs of intervals indicated.  Notice that the source of the source (a pair of 2-framed points) is indeed the source of the target, and similarly the target of the source is the target of the target.
\end{example}

\begin{example} \label{eg:radford}
The picture of an immersed surface in the introduction can be equipped with a normal framing (pointing out of the page on the more lightly shaded regions of the bordism), with a corona indicating the outgoing boundary (pointing downward on the lower boundary), and with arrows indicating the outgoing boundary of the boundary (pointing rightward at the two cusps), so that it represents a 3-framed 2-manifold bordism.  The source of this bordism is the stabilization of the 2-framed 1-manifold,
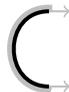
\[
\cb{
\begin{tikzpicture}
	\draw[linestyle,fuzzright] (4,7) to [looseness=1.6,out = 180, in = 180] (4,6);
	\begin{pgfonlayer}{background}
		\draw[->,outstyle] (4,7) -- +(0:\arrowlength);
		\draw[->,outstyle] (4,6) -- +(0:\arrowlength);
	\end{pgfonlayer}
\end{tikzpicture}
}
\]
and target of the bordism is the stabilization of the 2-framed 1-manifold
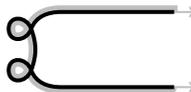
\[
\setlength{\linewid}{15pt}
\setlength{\fuzzwidth}{25pt}
\setlength{\arrowlength}{80pt}
\setlength{\arrowwidth}{7.5pt}
\cb{
\scalebox{.1}{
\begin{tikzpicture}
\draw[linestyle,fuzzright]
(15,5) to (0,5) to [out=180, in=70] (-4,2.75)
	to [looseness=1.6, out=-110, in=-90] (-7,2.75)
	to [looseness=1.6, out=90, in=110] (-4,2.75)
	to [out=-70, in=70] (-4,-2.75)
	to [looseness=1.6, out=-110, in=-90] (-7,-2.75)
	to [looseness=1.6, out=90, in=110] (-4,-2.75)
	to [out=-70, in=180] (0,-5) to (15,-5);
\begin{pgfonlayer}{background}
	\draw[-scalehead,outstyle] (15,5) -- +(0:\arrowlength);
	\draw[-scalehead,outstyle] (15,-5) -- +(0:\arrowlength);
\end{pgfonlayer}
\end{tikzpicture}
}
}
\setlength{\linewid}{1.5pt}
\setlength{\fuzzwidth}{2.5pt}
\setlength{\arrowlength}{8pt}
\setlength{\arrowwidth}{.75pt}
\]
\end{example}

\section{Duality in the 2-framed bordism category}\label{sec:framed-duality}

The 2-framed bordism 2-category is (like all reasonable bordism categories) fully dualizable.  In this section we explicitly illustrate all the duals and adjunctions in this 2-category.  For an analogous but much more involved analysis of the oriented bordism 2-category, and of the resulting classification of oriented 2-dimensional local field theories, see the thesis~\cite{schommer-pries-thesis}.

There are two 2-framed diffeomorphism classes of 2-framed points.  We focus on the following representatives of these classes:
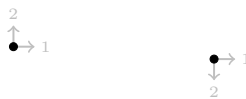
\begin{center}
\raisebox{2.2mm}{
\begin{tikzpicture}
	\filldraw (0,0) circle (\pointrad); 
	\begin{pgfonlayer}{background}
	\draw[->,outstyle] (0,0) -- +(90:\arrowlength) node[anchor=south,inner sep=2pt] {\tiny 2};
	\draw[->,outstyle] (0,0) -- +(0:\arrowlength) node[anchor=west,inner sep=2pt] {\tiny 1};
	\draw[->,outstyle,white] (0,0) -- +(-90:\arrowlength) node[anchor=south,inner sep=2pt] {\tiny 2};
	\end{pgfonlayer}
\end{tikzpicture}
}
\hspace{1.5cm}
\raisebox{-2.2mm}{
\begin{tikzpicture}
	\filldraw (0,0) circle (\pointrad); 
	\begin{pgfonlayer}{background}
	\draw[->,outstyle] (0,0) -- +(270:\arrowlength) node[anchor=north,inner sep=2pt] {\tiny 2};
	\draw[->,outstyle] (0,0) -- +(0:\arrowlength) node[anchor=west,inner sep=2pt] {\tiny 1};
	\draw[->,outstyle,white] (0,0) -- +(-270:\arrowlength) node[anchor=north,inner sep=2pt] {\tiny 2};
	\end{pgfonlayer}
\end{tikzpicture}
}
\end{center}
We refer to the first of these points as the positively framed point, $\pt_+$, and to the second as the negatively framed point, $\pt_-$.  These two 2-framed points are dual: there is an evaluation bordism $\pt_+ \sqcup \pt_- \ra \emptyset$ and a coevaluation bordism $\emptyset \ra \pt_- \sqcup \pt_+$ satisfying the usual zigzag equations (cf. Definition~\ref{def:adjoints_in_bicat}).  We pick the following two bordisms as evaluation and coevaluation bordisms witnessing this duality:
\begin{center}
	$\ev :=$\cb{
	\begin{tikzpicture}
	\draw[linestyle,fuzzright] (0,0) arc (-90:90:\smcirclerad);
	\end{tikzpicture}
	}
	\hspace{1.5cm}
	$\coev :=$ \cb{
	\begin{tikzpicture}
	\draw[emptylinestyle, white] (0,.1) -- (0,-2*\smcirclerad) -- +(0,-.1);
	\draw[linestyle,fuzzleft] (0,0) arc (90:270:\smcirclerad);
	\begin{pgfonlayer}{background}
		\draw[->,outstyle] (0,0) -- +(0:\arrowlength);
		\draw[->,outstyle] (0,-2*\smcirclerad) -- +(0:\arrowlength);
	\end{pgfonlayer}
	\end{tikzpicture}}
\end{center}
Because the bordism category is symmetric monoidal, a right dual is also a left dual; the existence of the above duality between $\pt_+$ and $\pt_-$ therefore shows that the 2-framed bordism 2-category $\FrBord_2$ is 1-dualizable.  (Stabilizing the framings of all the manifolds in question similarly shows that the $n$-framed bordism $n$-category $\FrBord_n$ is 1-dualizable; an analogous discussion shows that the 1-framed bordism category $\FrBord_1$ is 1-, that is fully, dualizable.)

Next we illustrate the adjoints of 1-morphisms in the 2-framed bordism 2-category.  The first morphisms that need adjoints are the evaluation and coevaluation bordisms arising above in the duality between the positively and negatively framed points.

\begin{example} \label{eg:evlevadj}
A left adjoint to the evaluation bordism is provided by the following 2-framed bordism:
\begin{center}
	$\ev^L := $ \cb{
	\begin{tikzpicture}
	\draw[linestyle,fuzzright] (0,0) arc (90:270:\smcirclerad);
	\begin{pgfonlayer}{background}
		\draw[->,outstyle] (0,0) -- +(0:\arrowlength);
		\draw[->,outstyle] (0,-2*\smcirclerad) -- +(0:\arrowlength);
	\end{pgfonlayer}
    \end{tikzpicture} }
\end{center}
The unit and counit 2-morphisms of the adjunction $\ev^L \dashv \ev$ are respectively the bordisms $u_1$ and $v_1$ defined in Examples~\ref{ex:disk_bordism_immersed} and~\ref{ex:saddle_bordism_immersed}.  (Precisely, we use a deformation of $v_1$ such that all four cusps point along the $x$-axis.)
\end{example}

\begin{example} \label{eg:evrevadj}
A right adjoint $\ev^R$ to the evaluation bordism, together with a unit $u_2$ and counit $v_2$ witnessing the adjunction $\ev \dashv \ev^R$, are as follows:
\begin{center}
	$\ev^R = $\cb{
	\begin{tikzpicture}
		\draw[linestyle,fuzzright] (0,0) 
		to [looseness=1.6, out=180, in=190] +(90:1.5*\smcirclerad)
		to [looseness=1.6, out=10, in=-10] +(90:-5*\smcirclerad)
		to [looseness=1.6, out=170, in=180] +(90:1.5*\smcirclerad);
		\begin{pgfonlayer}{background}
			\draw[->,outstyle] (0,0) -- +(0:\arrowlength);
			\draw[->,outstyle] (0,-2*\smcirclerad) -- +(0:\arrowlength);
		\end{pgfonlayer}
	\end{tikzpicture} }
	\hspace{1.5cm}
	$u_2 := $\cb{
	\begin{tikzpicture}
		\draw[linestyle, fill=\fillcolor] (0.75,0) 
			to [looseness=1.6, out=180, in=190] +(90:1.5*\smcirclerad)
			to [looseness=1.6, out=10, in=-10] +(90:-5*\smcirclerad)
			to [looseness=1.6, out=170, in=180] +(90:1.5*\smcirclerad)
			-- +(-0.75,0)
			to [looseness=1.1, out=0, in=170] +(-0.15,-0.70)
			to [looseness=2, out=-10, in=10] (0.6,0.7)
			to [looseness=1.1, out=190, in=0] (0,0)
			-- (0.75,0);
		\draw[linestyle,fuzzright] (0.75,0) 
			to [looseness=1.6, out=180, in=190] +(90:1.5*\smcirclerad)
			to [looseness=1.6, out=10, in=-10] +(90:-5*\smcirclerad)
			to [looseness=1.6, out=170, in=180] +(90:1.5*\smcirclerad)
			+(-0.75,0) 
			to [looseness=1.1, out=0, in=170] +(-0.15,-0.70)
			to [looseness=2, out=-10, in=10] (0.6,0.7)
			to [looseness=1.1, out=190, in=0] (0,0);
		\begin{pgfonlayer}{background}
			\draw[->,outstyle] (0.75,0) -- +(0:\arrowlength);
			\draw[->,outstyle] (0.75,-2*\smcirclerad) -- +(0:\arrowlength);
		\end{pgfonlayer}
	\end{tikzpicture} }
	\hspace{1.5cm}
	$v_2 :=$\cb{
	\begin{tikzpicture}
		\draw[linestyle, fill=\fillcolor] (0,0) 
		to [looseness=1.6, out=180, in=190] +(90:1.5*\smcirclerad)
		to [looseness=1.6, out=10, in=-10] +(90:-5*\smcirclerad)
		to [looseness=1.6, out=170, in=180] +(90:1.5*\smcirclerad)
		to [looseness=1.55, out=0, in=0] +(90:2*\smcirclerad);
	\end{tikzpicture} }
\end{center}
As our notation for $n$-framed manifolds permits immersed rather than embedded manifolds, it is often more convenient to depict $\ev^R$ by an isotopic (rel boundary) immersion:
\[
\cb{
	\begin{tikzpicture}
		\draw[linestyle,fuzzright] (0,0) 
		to [looseness=1.6, out=180, in=190] +(90:1.5*\smcirclerad)
		to [looseness=1.6, out=10, in=-10] +(90:-5*\smcirclerad)
		to [looseness=1.6, out=170, in=180] +(90:1.5*\smcirclerad);
		\begin{pgfonlayer}{background}
			\draw[->,outstyle] (0,0) -- +(0:\arrowlength);
			\draw[->,outstyle] (0,-2*\smcirclerad) -- +(0:\arrowlength);
		\end{pgfonlayer}
	\end{tikzpicture} 
	}
\simeq
\setlength{\linewid}{15pt}
\setlength{\fuzzwidth}{25pt}
\setlength{\arrowlength}{80pt}
\setlength{\arrowwidth}{7.5pt}
\cb{
\scalebox{.1}{
\begin{tikzpicture}
\draw[linestyle,fuzzright]
(0,5) to [out=180, in=70] (-4,2.75)
	to [looseness=1.6, out=-110, in=-90] (-7,2.75)
	to [looseness=1.6, out=90, in=110] (-4,2.75)
	to [out=-70, in=70] (-4,-2.75)
	to [looseness=1.6, out=-110, in=-90] (-7,-2.75)
	to [looseness=1.6, out=90, in=110] (-4,-2.75)
	to [out=-70, in=180] (0,-5);
\begin{pgfonlayer}{background}
	\draw[-scalehead,outstyle] (0,5) -- +(0:\arrowlength);
	\draw[-scalehead,outstyle] (0,-5) -- +(0:\arrowlength);
\end{pgfonlayer}
\end{tikzpicture}
}
}
\setlength{\linewid}{1.5pt}
\setlength{\fuzzwidth}{2.5pt}
\setlength{\arrowlength}{8pt}
\setlength{\arrowwidth}{.75pt}
\]
\end{example}

The 1-morphism $\ev^L$ itself has a left adjoint $\ev^{LL}$, which itself has a left adjoint, and so on, and similarly $\ev^R$ itself has a right adjoint $\ev^{RR}$, which has a right adjoint, and so on.  The 1-morphism $\coev$ also has an infinite chain of left and right adjoints.  These various adjunctions are illustrated in Figure~\ref{fig:adjointchains}.  (The chain of adjoints of the evaluation extends as follows: the 1-morphism $\ev^{LLLL}$ appears as $\ev^{LL}$ except with four loops, while $\ev^{RRR}$ appears as $\ev^R$ except with four loops.  The chain of adjoints of the coevaluation extends similarly.) 
The existence of these two chains of adjunctions shows that the 2-framed bordism 2-category is 2-dualizable.

\begin{figure}[ht]
\[
\cdots
\hspace{.2cm}
\dashv
\hspace{.2cm}
\setlength{\linewid}{1.5pt}
\setlength{\fuzzwidth}{2.5pt}
\setlength{\arrowlength}{8pt}
\setlength{\arrowwidth}{.75pt}
\cb{
\begin{tikzpicture}[rotate=-90]
\draw[linestyle,fuzzleft]
(.5,.2) to [out=-90, in=25] (.3,-.3)
	to [looseness=1.6, out=-155, in=225] (.15,-.15)
	to [looseness=1.6, out=45, in=65] (.3,-.3)
	to [out=245, in=-65] (-.3,-.3)
	to [looseness=1.6, out=115, in=135] (-.15,-.15)
	to [looseness=1.6, out=-45, in=-25] (-.3,-.3)
	to [out=155, in=-90] (-.5,.2);
		\begin{pgfonlayer}{background}
			\draw[->,outstyle] (.5,.2) -- +(90:\arrowlength);
			\draw[->,outstyle] (-.5,.2) -- +(90:\arrowlength);
		\end{pgfonlayer}	
\end{tikzpicture}
}
\hspace{.00cm}
\dashv
\hspace{.2cm}
\setlength{\linewid}{15pt}
\setlength{\fuzzwidth}{25pt}
\setlength{\arrowlength}{80pt}
\setlength{\arrowwidth}{7.5pt}
\cb{
\scalebox{.1}{
\begin{tikzpicture}
\draw[linestyle,fuzzleft]
(0,5) to [out=0, in=110] (4,2.75)
	to [looseness=1.6, out=-70, in=-90] (7,2.75)
	to [looseness=1.6, out=90, in=70] (4,2.75)
	to [out=-110, in=110] (4,-2.75)
	to [looseness=1.6, out=-70, in=-90] (7,-2.75)
	to [looseness=1.6, out=90, in=70] (4,-2.75)
	to [out=-110, in=0] (0,-5);
\end{tikzpicture}
}
}
\setlength{\linewid}{1.5pt}
\setlength{\fuzzwidth}{2.5pt}
\setlength{\arrowlength}{8pt}
\setlength{\arrowwidth}{.75pt}
\hspace{.2cm}
\dashv
\hspace{.1cm}
\cb{
\begin{tikzpicture}
		\draw[linestyle,fuzzright] (4.25,0.5) to [looseness=1.6,out = 180, in = 180] (4.25,-0.5);
		\begin{pgfonlayer}{background}
			\draw[->,outstyle] (4.25,0.5) -- +(0:\arrowlength);
			\draw[->,outstyle] (4.25,-0.5) -- +(0:\arrowlength);
		\end{pgfonlayer}
\end{tikzpicture}
}
\hspace{.05cm}
\dashv
\hspace{.25cm}
\cb{
\begin{tikzpicture}
		\draw[linestyle,fuzzleft] (5.75,0.5) to [looseness=1.6,out = 0, in = 0] (5.75,-0.5);
\end{tikzpicture}
}
\hspace{.1cm}
\dashv
\hspace{.25cm}
\setlength{\linewid}{15pt}
\setlength{\fuzzwidth}{25pt}
\setlength{\arrowlength}{80pt}
\setlength{\arrowwidth}{7.5pt}
\cb{
\scalebox{.1}{
\begin{tikzpicture}
\draw[linestyle,fuzzright]
(0,5) to [out=180, in=70] (-4,2.75)
	to [looseness=1.6, out=-110, in=-90] (-7,2.75)
	to [looseness=1.6, out=90, in=110] (-4,2.75)
	to [out=-70, in=70] (-4,-2.75)
	to [looseness=1.6, out=-110, in=-90] (-7,-2.75)
	to [looseness=1.6, out=90, in=110] (-4,-2.75)
	to [out=-70, in=180] (0,-5);
\begin{pgfonlayer}{background}
	\draw[-scalehead,outstyle] (0,5) -- +(0:\arrowlength);
	\draw[-scalehead,outstyle] (0,-5) -- +(0:\arrowlength);
\end{pgfonlayer}
\end{tikzpicture}
}
}
\hspace{.00cm}
\dashv
\hspace{.2cm}
\setlength{\linewid}{1.5pt}
\setlength{\fuzzwidth}{2.5pt}
\setlength{\arrowlength}{8pt}
\setlength{\arrowwidth}{.75pt}
\cb{
\begin{tikzpicture}[rotate=90]
\draw[linestyle,fuzzleft]
(.5,.2) to [out=-90, in=25] (.3,-.3)
	to [looseness=1.6, out=-155, in=225] (.15,-.15)
	to [looseness=1.6, out=45, in=65] (.3,-.3)
	to [out=245, in=-65] (-.3,-.3)
	to [looseness=1.6, out=115, in=135] (-.15,-.15)
	to [looseness=1.6, out=-45, in=-25] (-.3,-.3)
	to [out=155, in=-90] (-.5,.2);
\end{tikzpicture}
}
\hspace{.1cm}
\dashv
\hspace{.2cm}
\cdots
\]\vspace{1pt}

\hspace*{.1cm}
\begin{tabularx}{.86\textwidth}{*6{>{\centering}X}}
$\ev^{LLL}$
&
$\ev^{LL}$ 
&
$\ev^L$ 
&
$\ev$ 
&
$\ev^R$
&
$\ev^{RR}$
\end{tabularx}

\[
\cdots 
\hspace{.25cm}
\dashv
\hspace{.2cm}
\setlength{\linewid}{1.5pt}
\setlength{\fuzzwidth}{2.5pt}
\setlength{\arrowlength}{8pt}
\setlength{\arrowwidth}{.75pt}
\cb{
\begin{tikzpicture}[rotate=90]
\draw[linestyle,fuzzright]
(.5,.2) to [out=-90, in=25] (.3,-.3)
	to [looseness=1.6, out=-155, in=225] (.15,-.15)
	to [looseness=1.6, out=45, in=65] (.3,-.3)
	to [out=245, in=-65] (-.3,-.3)
	to [looseness=1.6, out=115, in=135] (-.15,-.15)
	to [looseness=1.6, out=-45, in=-25] (-.3,-.3)
	to [out=155, in=-90] (-.5,.2);
\end{tikzpicture}
}
\hspace{.15cm}
\dashv
\hspace{.25cm}
\setlength{\linewid}{15pt}
\setlength{\fuzzwidth}{25pt}
\setlength{\arrowlength}{80pt}
\setlength{\arrowwidth}{7.5pt}
\cb{
\scalebox{.1}{
\begin{tikzpicture}
\draw[emptylinestyle, white] (0,5.8) -- (0,-5.8);
\draw[linestyle,fuzzleft]
(0,5) to [out=180, in=70] (-4,2.75)
	to [looseness=1.6, out=-110, in=-90] (-7,2.75)
	to [looseness=1.6, out=90, in=110] (-4,2.75)
	to [out=-70, in=70] (-4,-2.75)
	to [looseness=1.6, out=-110, in=-90] (-7,-2.75)
	to [looseness=1.6, out=90, in=110] (-4,-2.75)
	to [out=-70, in=180] (0,-5);
\begin{pgfonlayer}{background}
	\draw[-scalehead,outstyle] (0,5) -- +(0:\arrowlength);
	\draw[-scalehead,outstyle] (0,-5) -- +(0:\arrowlength);
\end{pgfonlayer}
\end{tikzpicture}
}
}
\setlength{\linewid}{1.5pt}
\setlength{\fuzzwidth}{2.5pt}
\setlength{\arrowlength}{8pt}
\setlength{\arrowwidth}{.75pt}
\hspace{-.05cm}
\dashv
\hspace{.25cm}
\cb{
\begin{tikzpicture}
		\draw[linestyle,fuzzright] (5.75,0.5) to [looseness=1.6,out = 0, in = 0] (5.75,-0.5);
\end{tikzpicture}
}
\hspace{.1cm}
\dashv
\hspace{.1cm}
\cb{
\begin{tikzpicture}
	\draw[emptylinestyle, white] (4.25,.6) -- (4.25,-.6);
		\draw[linestyle,fuzzleft] (4.25,0.5) to [looseness=1.6,out = 180, in = 180] (4.25,-0.5);
		\begin{pgfonlayer}{background}
			\draw[->,outstyle] (4.25,0.5) -- +(0:\arrowlength);
			\draw[->,outstyle] (4.25,-0.5) -- +(0:\arrowlength);
		\end{pgfonlayer}
\end{tikzpicture}
}
\hspace{.05cm}
\dashv
\hspace{.25cm}
\setlength{\linewid}{15pt}
\setlength{\fuzzwidth}{25pt}
\setlength{\arrowlength}{80pt}
\setlength{\arrowwidth}{7.5pt}
\cb{
\scalebox{.1}{
\begin{tikzpicture}
\draw[linestyle,fuzzright]
(0,5) to [out=0, in=110] (4,2.75)
	to [looseness=1.6, out=-70, in=-90] (7,2.75)
	to [looseness=1.6, out=90, in=70] (4,2.75)
	to [out=-110, in=110] (4,-2.75)
	to [looseness=1.6, out=-70, in=-90] (7,-2.75)
	to [looseness=1.6, out=90, in=70] (4,-2.75)
	to [out=-110, in=0] (0,-5);
\end{tikzpicture}
}
}
\hspace{.2cm}
\dashv
\hspace{.15cm}
\setlength{\linewid}{1.5pt}
\setlength{\fuzzwidth}{2.5pt}
\setlength{\arrowlength}{8pt}
\setlength{\arrowwidth}{.75pt}
\cb{
\begin{tikzpicture}[rotate=-90]
\draw[emptylinestyle, white] (.6,.2) -- (-.6,.2);
\draw[linestyle,fuzzright]
(.5,.2) to [out=-90, in=25] (.3,-.3)
	to [looseness=1.6, out=-155, in=225] (.15,-.15)
	to [looseness=1.6, out=45, in=65] (.3,-.3)
	to [out=245, in=-65] (-.3,-.3)
	to [looseness=1.6, out=115, in=135] (-.15,-.15)
	to [looseness=1.6, out=-45, in=-25] (-.3,-.3)
	to [out=155, in=-90] (-.5,.2);
		\begin{pgfonlayer}{background}
			\draw[->,outstyle] (.5,.2) -- +(90:\arrowlength);
			\draw[->,outstyle] (-.5,.2) -- +(90:\arrowlength);
		\end{pgfonlayer}	
\end{tikzpicture}
}
\setlength{\linewid}{1.5pt}
\setlength{\fuzzwidth}{2.5pt}
\setlength{\arrowlength}{8pt}
\setlength{\arrowwidth}{.75pt}
\hspace{-.05cm}
\dashv
\hspace{.2cm}
\cdots
\]\vspace{1pt}

\hspace*{.2cm}
\begin{tabularx}{.87\textwidth}{*6{>{\centering}X}}
$\coev^{LLL}$
&
$\coev^{LL}$ 
&
$\coev^L$ 
&
$\coev$ 
&
$\coev^R$
&
$\coev^{RR}$
\end{tabularx}
\caption{Two infinite chains of adjunctions of 2-framed bordisms.} \label{fig:adjointchains}
\end{figure}
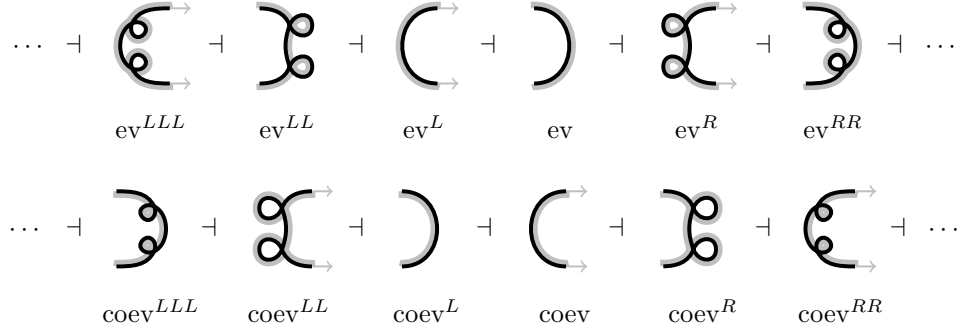

\section{The Serre bordism and the Serre automorphism} \label{sec:Serre}

It is possible to write any 2-framed 1-manifold bordism as a composite of the elementary pieces arising in Figure~\ref{fig:adjointchains}.  We illustrate this in practice with a crucial bordism, namely the one represented by the following normally-oriented immersed 1-manifold:
\begin{center}
\begin{tikzpicture}
\draw[emptylinestyle, white] (.7,.1) -- (.7,-.1);
\draw[linestyle,fuzzright] 
(.7,0) to [out=180, in=-20] (0,.1)
	to [looseness=1.6, out=160, in=180] (0,.4)
	to [looseness=1.6, out=0, in=20] (0,.1)
	to [out=-160, in=0] (-.7,0);
\begin{pgfonlayer}{background}
	\draw[->,outstyle] (.7,0) -- +(0:\arrowlength);
\end{pgfonlayer}
\end{tikzpicture}
\end{center}
By stabilizing the framing, this manifold represents, for any $n \geq 2$, an $n$-framed bordism $\cS$ from the standard positively framed point to itself.  This bordism is called the loop bordism, and we also refer to it as the {\bfseries Serre bordism}, for reasons described in Example~\ref{eg:serrefunctor} below.

\begin{remark}
The interval $[0,1]$ can be viewed as a bordism with 0 as an incoming boundary point and 1 as an outgoing boundary point.  Consider, up to homotopy, the set of $n$-framings of this interval that restrict to the standard positive $n$-framing at both boundary points.  This set is a group under concatenation of intervals, and that group is canonically identified with $\pi_1(SO(n))$.  When $n=2$, this group of framings is therefore identified with $\ZZ$---here we take the counterclockwise rotation as the positive generator of $\pi_1(SO(2))$.  The loop bordism $\cS$ is the framing corresponding to $-1 \in \ZZ$ under this identification.  We will refer to the bordism corresponding to the $+1$ framing as the inverse loop bordism and denote it $\cS^{-1}$.  Evidently, there are isomorphisms $\cS \circ \cS^{-1} \cong \id_{\pt_+} \cong \cS^{-1} \circ \cS$.
\end{remark}

The loop bordism and the inverse loop bordism admit decompositions into more elementary pieces, as shown in Figure~\ref{fig:serrebordism}.
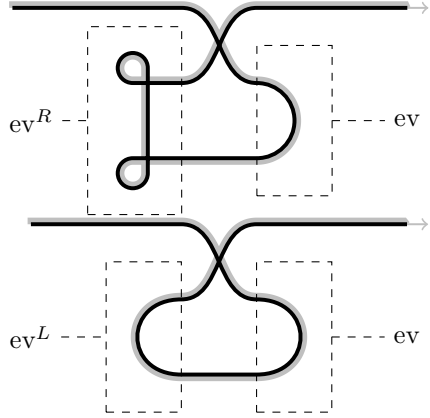
\begin{figure}[htbp]
\vspace{\topskip}	
	\begin{tikzpicture}
		\draw [linestyle,fuzzleft] (-.75, 3) -- (1.5, 3)
			to [looseness=1, out = 0, in = 180] (2.5,2) 
			arc (90:-90:0.5cm) -- (1.5,1) -- +(-\evlength-\loopsize,0)
			arc (-270:0:\loopsize)
			-- +(0,\evlengthv+\loopsize+\loopsize)
			arc (0:270:\loopsize)
			-- +(\evlength+\loopsize,0)
			to [looseness=1, out = 0, in = 180] (2.5, 3) -- (4.5,3);
		\begin{pgfonlayer}{background}
				\draw[->,outstyle] (4.5,3) -- +(0:\arrowlength);
		\end{pgfonlayer}	
		\draw [dashed] (2.5,2.5) rectangle (3.5,0.5);
		\node (B) at (4.5,1.5) {$\ev$};
		\draw [dashed] (3.5, 1.5) -- (B);	
		\draw [dashed] (0.25,2.75) rectangle (1.5,0.25);
		\node (A) at (-.5, 1.5) {$\ev^R$};
		\draw [dashed] (A) -- (0.25, 1.5);
	\end{tikzpicture}
\hspace{1cm}
\raisebox{7pt}{
	\begin{tikzpicture}
		\draw [linestyle,fuzzleft] (-.5, 3) -- (1.5, 3)
			to [looseness=1, out = 0, in = 180] (2.5,2)
			to [looseness=2,out=0,in=0] (2.5,1)
			to (1.5,1)
			to [looseness=2,out=180,in=180] (1.5,2)
			to [looseness=1, out = 0, in = 180] (2.5, 3) -- (4.5,3);
		\begin{pgfonlayer}{background}
				\draw[->,outstyle] (4.5,3) -- +(0:\arrowlength);
		\end{pgfonlayer}	
		\draw [dashed] (2.5,2.5) rectangle (3.5,0.5);
		\node (B) at (4.5,1.5) {$\ev$};
		\draw [dashed] (3.5, 1.5) -- (B);	
		\draw [dashed] (0.5,2.5) rectangle (1.5, 0.5);
		\node (A) at (-.5, 1.5) {$\ev^L$};
		\draw [dashed] (A) -- (0.5, 1.5);
	\end{tikzpicture}
	}
\caption{Decompositions of the Serre and inverse Serre bordisms.} \label{fig:serrebordism}
\end{figure}

\nid Written out algebraically, these two decompositions have the following form:
\begin{align*}
	\cS &\cong (\id_{\pt_+} \sqcup \ev) \circ (\tau_{\pt_+, \pt_+} \sqcup \id_{\pt_-}) \circ (\id_{\pt_+} \sqcup \ev^R) \\
	\cS^{-1} &\cong (\id_{\pt_+} \sqcup \ev) \circ (\tau_{\pt_+, \pt_+} \sqcup \id_{\pt_-}) \circ (\id_{\pt_+} \sqcup \ev^L) 
\end{align*}
Here $\tau$ denotes the symmetric monoidal switch.

Suppose we have a 2-framed 2-dimensional local field theory $\cF: \FrBord_2 \ra \cC$ with target the symmetric monoidal $(\infty,2)$-category $\cC$; denote by $x:=\cF(\pt_+) \in \cC$ the 2-dualizable object that is the image of the standard positive point.  The Serre bordism is an automorphism of the positive point, and so the value of the field theory on the Serre bordism is an automorphism $\cS_x := \cF(\cS)$ of the object $x \in \cC$---we call this automorphism the \emph{Serre automorphism} of $x$.  The terminology is motivated by the following example.

\begin{example}[{\cite[Rem. 4.2.4]{lurie-ch}}] \label{eg:serrefunctor}
Let $\cC$ be the $(\infty,2)$-category of cocomplete differential graded categories over a field.  Let $D$ be the category of quasi-coherent complexes on a smooth variety $X$.  The object $D \in \cC$ is 2-dualizable and so there is an associated field theory $\cF_D$ with $\cF_D(\pt_+) \simeq D$.  The resulting automorphism $\cS_D$ is the Serre functor that appears in Serre duality; it is given by tensoring with the canonical line bundle on $X$ and shifting by the dimension of $X$.
\end{example}

We can use the above decompositions of the loop bordism and inverse loop bordism to calculate the value of the Serre automorphism of any object in any target category.
\begin{proposition} \label{prop:serrecalc}
Let $x \in \cC$ be a 2-dualizable object of a symmetric monoidal $(\infty,2)$-category $\cC$.  Let $\ev_x$ and $\coev_x$ denote evaluation and coevaluation morphisms witnessing a duality between $x$ and an object $\overline{x}$.  Let $\ev_x^R$ and $\ev_x^L$ denote a right and a left adjoint to the evaluation morphism.  The Serre automorphism and inverse Serre automorphism of $x$ are given by the following formulas:
	\begin{align*}
		\cS_x &\simeq (\id_{x} \otimes \ev_{x}) \circ (\tau_{x, x} \otimes \id_{\overline{x}}) \circ (\id_{x} \otimes \ev_{x}^R) \\
		\cS_x^{-1} &\simeq (\id_{x} \otimes \ev_{x}) \circ (\tau_{x, x} \otimes \id_{\overline{x}}) \circ (\id_{x} \otimes \ev_{x}^L)
	\end{align*}
\end{proposition}
\nid Because symmetric monoidal functors take duals to duals and adjoints to adjoints, the previous decomposition of the Serre and inverse Serre bordisms immediately implies that the Serre and inverse Serre automorphisms are equivalent to the listed expressions.  (Recall also that duals and adjoints are uniquely determined---more specifically, the category of triples $(\overline{x},\ev_x,\coev_x)$ witnessing a dual to $x$ is contractible, and similarly for triples witnessing a right or left adjoint to $\ev$---and so it does not matter what choices we make for $\overline{x}$, $\ev_x$, $\coev_x$, $\ev_x^R$, and $\ev_x^L$.)  In fact, given a 2-dualizable object $x \in \cC$, a local field theory may be chosen such that the Serre and inverse Serre automorphisms of $x$ in that choice of field theory are exactly equal to the formulas listed in the proposition; ensuring that equality is somewhat more subtle and is discussed later in Remark~\ref{rem:hep}.

\section{The Radford bordism and the Radford equivalence} \label{sec:Radford}

As mentioned previously, the group, under concatenation, of homotopy classes of $n$-framings of the interval is $\pi_1(SO(n))$.  When $n$ is at least 3, this group is $\ZZ/2$, and we see that the loop bordism is an involution; said differently, the loop bordism and the inverse loop bordism are equivalent.  For our purposes, it is not enough to merely know that there exists an equivalence; we need an explicit realization of the equivalence.  In the last section we described a decomposition of the loop bordism $\cS$, respectively its inverse $\cS^{-1}$, in terms of the more elementary pieces $\ev$ and $\ev^R$, respectively $\ev$ and $\ev^L$.   To produce an equivalence between $\cS$ and $\cS^{-1}$ it therefore suffices to provide an equivalence between the bordisms $\ev^R$ and $\ev^L$.  Such an equivalence is given by the unique (up to 3-framed diffeomorphism rel boundary) 3-framed genus zero bordism $\cR$ from $\ev^L$ to $\ev^R$; this bordism was described in Example~\ref{eg:radford} and illustrated in the introduction.  We call this bordism the belt bordism---if you wrap your belt into the configuration of the outgoing boundary and pull, it will happily (and very rapidly) trace out this bordism for you.  We will also refer to this bordism as the {\bfseries Radford bordism}, because it plays a crucial role in our topological proof of Radford's quadruple dual theorem, in Section~\ref{sec:topquaddual}.

In this section we describe an explicit, Morse-style decomposition of the Radford bordism; we introduce a condition on an object $x$ of a 3-category, weaker than full dualizability, that ensures we can construct an equivalence, guided by the decomposition of the Radford bordism, from the Serre automorphism $\cS_x$ to its inverse $\cS_x^{-1}$.

\subsection{A decomposition of the Radford bordism}

\begin{figure}[bp]
\makebox[\textwidth][c]{
		\begin{tikzpicture}[
			decoration={border, 
				segment length = 4pt, 	
				amplitude = 2pt, 		
				angle = 0  				 
				}, 
			contour line/.style={thin, blue}
				]
				
			\colorlet{surfacecolor1}{black!12}
			\colorlet{surfacecolor2}{black!5}	
			
			
			\begin{scope}[scale=0.75]			
				\draw [fuzzright] (0,2.5) arc (180:270: 1cm and 0.5cm) to [out = 0, in = 210] (2,2.5)
					-- (4.5, 1.25);
				\draw [fuzzleft] (0,2.5) arc (180:90: 1cm and 0.5cm);
				\draw [fuzzright] (4, 0.5) arc (180:360: 1cm and 0.5cm);
				\draw [fuzzleft] (4, 0.5) arc (180:90: 1cm and 0.5cm);
				
				\draw [fuzzright] (6, 1.5) -- (8, 2.5) to [out = 30, in = 240] (11, 6);
				\draw [fuzzright]	 (8.25, 5.8) to [out = 45, in = 260] (9, 7);
				\begin{pgfonlayer}{background}
					\draw[->,outstyle] (11,6) -- +(30:\arrowlength);
					\draw[->,outstyle] (9,7) -- +(45:\arrowlength);
				\end{pgfonlayer}
				
			\fill [color = surfacecolor1] (0, 2.5) .. controls (0, 4) and (2,6.5) .. (1.75,8)
				arc (180:0:2cm and 3cm)
				to [out = 270, in = 90] (6, 5.5)
				-- (6,0.5)
				arc (360: 180: 1cm and 0.5cm)
				-- (4, 6)
				arc (90:180:2cm and 3.5cm)
				to [out = 210, in = 0] (1,2)
				arc (270: 180: 1cm and 0.5cm);
			
			\fill [color = surfacecolor1] (7, 6.5) parabola bend (8.5, 5.8) (8.25, 5.8) 
				to [out = 45, in = 260] (9, 7)
				parabola bend (7.75, 6.4) (7, 6.5);
			
			\fill [color = surfacecolor2] (4, 6) arc (90:180:2cm and 3.5cm)
				-- (4, 1.5) -- (4,6);
			
			\fill [color = surfacecolor2] (6, 5.5) parabola (7, 6.5) parabola bend (8.5, 5.8) (11, 6)
				to [out = 240, in = 30] (8, 2.5) -- (6, 1.5) -- (6, 5.5);

			\draw (0, 2.5) .. controls (0, 4) and (2,6.5) .. (1.75,8)
				arc (180:0:2cm and 3cm)
				to [out = 270, in = 90] (6, 5.5)
				-- (6,0.5)
				arc (360: 180: 1cm and 0.5cm)
				-- (4, 6)
				arc (90:180:2cm and 3.5cm)
				to [out = 210, in = 0] (1,2)
				arc (270: 180: 1cm and 0.5cm);
			
			\draw (2,2.5) -- (4,1.5) 
				decorate {-- (5,1) to [out = -30, in = 135] (6,0.5)
					  (4,0.5) to [out = 45, in = 210] (5,1) -- (6,1.5)} 
				-- (8, 2.5) to [out = 30, in = 240] (11, 6);
			\draw decorate { (5, 1) arc (0:90:1cm and 5cm)};
			\draw decorate {(2,2.5) -- (7,5) to [out = 30, in = 225] (8.25, 5.8)} to [out = 45, in = 260] (9, 7) ;
			\draw decorate {(2,2.5) to [out = 150, in = 0] (1,3) arc (90:180:1cm and 0.5cm)};
			\draw (4, 6) to [out = 90, in = 270] (4.5, 8.5);
			\draw (6, 5.5) decorate { parabola (4.5, 8.5)} (6, 5.5) parabola (7, 6.5)
				parabola bend (7.75, 6.4) (9, 7);
			\draw (7, 6.5) parabola bend (8.5, 5.8) (11, 6);
			
			\draw [contour line] (2.3, 10) arc (180:360:1.45cm and 0.5cm)
				decorate {(2.3, 10) arc (180:0:1.45cm and 0.5cm)};

			\draw [contour line] (1.75,8) arc (180:320:1.45cm and 0.5cm)
				decorate {to [out = 30,in = 135](4.65, 7.6) to [out = -45, in = 45] (4.2, 7.1) }
				arc (180:270:0.5cm and 0.2cm) to [out = 0, in = -90] (5.75,8)
			;
			\draw [contour line] decorate {(1.75,8) arc (180:0:2cm and 0.6cm)};
			
			\draw [contour line] 
				decorate {(1.25, 5.75) arc (180:90: 1cm and 0.5cm) to [out = 0, in = 150] (3.5,5.75)}
				 -- (4, 5.5) 
				decorate {-- (4.5,5.25) to [out = -30, in = 135] (6,5)}
				arc (360: 180: 1cm and 0.5cm)
				decorate { to [out = 45, in = 210] (4.5,5.25)};
			
			\draw [contour line] 
			 (1.25, 5.75) arc (180: 270: 1cm and 0.5cm) to [out = 0, in = 210] (3.25,5.75)
			 	 to [out = 30, in = 150] (5.25, 6.15) to [out = -30,in = 30] (4.5, 5.25);
			
			\end{scope}
					
		\begin{scope}[xshift = 7cm, yshift = 1.75cm]
			
						\draw[linestyle,fuzzright] (4,7) to [looseness=1.6,out = 180, in = 180] (4,6);
						\begin{pgfonlayer}{background}
							\draw[->,outstyle] (4,7) -- +(0:\arrowlength);
							\draw[->,outstyle] (4,6) -- +(0:\arrowlength);
						\end{pgfonlayer}
						
						\begin{scope}[yshift=-.05cm]
						\draw [thick, ->] (3, 5.7) -- (3, 4.8);
							\draw[linestyle, fill=\fillcolor, yshift = 0.5*\smcircleradt] (2.3,5.25) 
								to [looseness=1.6, out=180, in=190] +(90:0.75*\smcircleradt)
								to [looseness=1.6, out=10, in=-10] +(90:-2.5*\smcircleradt)
								to [looseness=1.6, out=170, in=180] +(90:.75*\smcircleradt)
								to [looseness=1.55, out=0, in=0] +(90:\smcircleradt);
						\node at (2.75,5.5) {$\scriptstyle R$};
						\end{scope}

						\draw[linestyle,fuzzright] (4,4.4) to [looseness=1.6,out = 180, in = 180] +(0,-1);
							\begin{pgfonlayer}{background}
								\draw[->,outstyle] (4,4.4) -- +(0:\arrowlength);
								\draw[->,outstyle] (4,3.4) -- +(0:\arrowlength);
							\end{pgfonlayer}
						\draw[yshift = \smcirclerad, linestyle, fuzzright] (2,3.9) 
							to [looseness=1.6, out=180, in=190] +(90:1.5*\smcirclerad)
							to [looseness=1.6, out=10, in=-10] +(90:-5*\smcirclerad)
							to [looseness=1.6, out=170, in=180] +(90:1.5*\smcirclerad)
							to [looseness=1.55, out=0, in=0] +(90:2*\smcirclerad);

						\draw [double equal sign distance] (3, 3.00) -- (3, 2.5);
							
						\begin{scope}[yshift=.1cm]
						\draw[linestyle,fuzzright] (4,2) to [looseness=1.6,out = 180, in = 180] +(0,-1);
								\begin{pgfonlayer}{background}
									\draw[->,outstyle] (4,2) -- +(0:\arrowlength);
									\draw[->,outstyle] (4,1) -- +(0:\arrowlength);
								\end{pgfonlayer}
						\draw[linestyle,fuzzright] (2.5, 2) 
							-- +(-\evlength+.1-\loopsize,0)
			arc (-90:-360:\loopsize)
			-- +(0,-\evlengthv-\loopsize-\loopsize)
			arc (0:-270:\loopsize)
			-- +(\evlength-.1+\loopsize,0)
			arc (-90:90:0.5cm);
						\end{scope}
						
						\begin{scope}[yshift=.05cm]
						\draw [thick, ->] (3, 0.7) -- (3, -0.2);
						\begin{scope}[xshift = 3.3cm, yshift = 0.0cm, scale = 0.57]
							\filldraw[linestyle,fill=\fillcolor] 
							(0,0) .. controls (.25,.25) and (.75,.25) .. (1,0)
								.. controls (.75,.25) and (.75,.75) .. (1,1)
								.. controls (.75,.75) and (.25,.75) .. (0,1)
								.. controls (.25,.75) and (.25,.25) .. (0,0);
						\draw[linestyle,fuzzright]
							(0,0) .. controls (.25,.25) and (.75,.25) .. (1,0);
						\draw[linestyle,fuzzleft]
							(0,1) .. controls (.25,.75) and (.75,.75) .. (1,1);
						\begin{pgfonlayer}{background}
							\draw[->,outstyle] (1,1) -- +(45:\arrowlength);
							\draw[->,outstyle] (1,0) -- +(-45:\arrowlength);
						\end{pgfonlayer}
						\end{scope}
						\end{scope}

							\draw[linestyle,fuzzright] (4, -0.5) -- (2.5, -0.5)
							-- +(-\evlength+.1-\loopsize,0)
			arc (-90:-360:\loopsize)
			-- +(0,-\evlengthv-\loopsize-\loopsize)
			arc (0:-270:\loopsize)
			-- +(\evlength-.1+\loopsize,0)
			-- (4, -1.5);
							\begin{pgfonlayer}{background}
								\draw[->,outstyle] (4,-0.5) -- +(0:\arrowlength);
								\draw[->,outstyle] (4,-1.5) -- +(0:\arrowlength);
							\end{pgfonlayer}

		\end{scope}	
			
	\begin{scope}[xshift = 0cm, yshift = -.25cm]

		\node[anchor=mid west, \graytextcolor] (A) at (11.5, 8.5) {$= \ev^L$};
		\node[anchor=mid west, \graytextcolor] (B) at (11.5, 6) {$= \ev^L \circ \ev \circ \ev^R$};		
		\node[anchor=mid west, \graytextcolor] (C) at (11.5, 1) {$= \ev^R$};
		\node[anchor=mid west, \graytextcolor] at (11.5, 7.25) {$= \id_{\ev^L} \circledcirc v_2^R$};
		\node[anchor=mid west, \graytextcolor] at (11.5, 2.25) {$= v_1 \circledcirc \id_{\ev^R}$};
		
		\node at (6, -1) {$\cR = (v_1 \circledcirc \id_{\ev^R})\circ (\id_{\ev^L} \circledcirc v_2^R)$};				
	\end{scope}				
							
		\end{tikzpicture}
}		
	\caption{A Morse decomposition of the Radford bordism.} \label{fig:Radford_bordism}
\end{figure}

Figure~\ref{fig:Radford_bordism} illustrates a (minimal) Morse decomposition of the Radford bordism.  The normal framing is not indicated, but implicitly points out of the page on the more lightly shaded regions and into the page on the more darkly shaded regions of the bordism.  The gray corona indicates that the lower boundary is outgoing, and the arrows indicate that the cusp points are outgoing for both the source and target of the bordism.  As drawn, the Morse decomposition of the bordism has one index-2 and one index-1 critical point; it is crucial to know not only the indices, but to identify those critical points as implementing particular elementary dualizability operations in the 3-framed bordism category.  The index-2 critical point is the right adjoint to the counit $v_2$ of the adjunction $\ev \dashv \ev^R$ described in Example~\ref{eg:evrevadj}.  The index-1 critical point is the counit $v_1$ of the adjunction $\ev^L \dashv \ev$ described in Example~\ref{eg:evlevadj} and illustrated in Example~\ref{ex:saddle_bordism_immersed}.  Written out algebraically, we have the following formula for the Radford bordism:
\begin{equation*}
	\cR = (v_1 \circledcirc \id_{\ev^R})\circ (\id_{\ev^L} \circledcirc v_2^R).
\end{equation*}
Here $\circ$ denotes composition of surfaces in a vertical (2-morphism) direction, and $\circledcirc$ denotes composition of surfaces in a horizontal (1-morphism) direction.  Neither of the factors in this decomposition is invertible, and so a priori it is not clear that this bordism is an equivalence---in the next section we will see that it is, as promised.

\subsection{A categorical formula for the Radford equivalence}

Suppose we have a 3-framed 3-dimen\-sional local field theory $\cF: \FrBord_3 \ra \cC$; let $x:=\cF(\pt_+) \in \cC$ denote the 3-dualizable object that is the image of the standard positive point.  We can certainly evaluate the theory $\cF$ on the Radford bordism, obtaining a 2-morphism $\cR_x := \cF(\cR) : \cF(\ev^L) \ra \cF(\ev^R)$ between the images of the left and right adjoints of the evaluation bordism; we will see shortly that this morphism is an equivalence, and therefore provides an equivalence between the Serre automorphism $\cS_x$ and its inverse $\cS_x^{-1}$.  However, the Radford bordism and its inverse only use certain elementary bordisms and their adjoints, and so we do not need the full strength of 3-dualizability to construct an equivalence between the Serre automorphism and its inverse.  We capture exactly the necessary amount of dualizability in the following definition.
\begin{definition} \label{def:Radford-Object}
Let $\cC$ be a symmetric monoidal $(\infty,n)$-category, for $n \geq 3$, and let $x \in \cC$ be a 2-dualizable object.  Let
	\begin{align*}
		\ev_x: & \; x \otimes \overline{x} \to 1 \\
		\coev_x: & \;  1 \to  \overline{x} \otimes x
	\end{align*}
be evaluation and coevaluation maps witnessing the duality between $x$ and $\overline{x}$.  Let $\ev_x^R$ be the right adjoint of the evaluation $\ev_x$, and let
	\begin{align*}
		u_x: & \;  \id_{x \otimes \overline{x}} \to \ev_x^R \circ \ev_x \\
		v_x: \; & \ev_x \circ \ev_x^R \to \id_1. 
	\end{align*}
be unit and counit maps witnessing that adjunction.  We say that $x$ is a \emph{Radford object} if $u_x$ and $v_x$ both admit right adjoints.
\end{definition}
\begin{definition} \label{def:radfordequiv}
For a Radford object $x \in \cC$, choose a dual $\overline{x}$ with evaluation and coevaluation maps $\ev_x$ and $\coev_x$; choose a right adjoint $\ev_x^R$ witnessed by unit $u_x$ and counit $v_x$, and a left adjoint $\ev_x^L$ witnessed by unit $\tilde{u}_x : \id_1 \ra \ev_x \circ \ev_x^L$ and counit $\tilde{v}_x: \ev_x^L \circ \ev_x \ra \id_{x \otimes \overline{x}}$; and choose a right adjoint $v_x^R$.  The \emph{Radford equivalence} $\cR_x$ is the composite
\[
\cR_x: \ev_x^L \xra{\id_{\ev_x^L} \circledcirc v_x^R} \ev_x^L \circ \ev_x \circ \ev_x^R \xra{\tilde{v}_x \circledcirc \id_{\ev_x^R}} \ev_x^R.
\]
\end{definition} \vspace{8pt}

Note that certainly any 3-dualizable object is Radford, and in that case the Radford equivalence is the image of the Radford bordism under the corresponding local field theory.  For any Radford object $x$, the equivalence class of the Radford equivalence depends only on the object $x$ and not on any of the other choices mentioned.  Definition~\ref{def:radfordequiv} begs the question of whether the given morphism is an equivalence---we address that issue presently.  Note that the apparent asymmetry in the definition could be resolved by considering the Radford equivalence and its inverse simultaneously.

\begin{theorem} \label{thm:Cat_Radford}
For any Radford object $x \in \cC$ in a symmetric monoidal $(\infty,n)$-category, with $n \geq 3$, the Radford map $\cR_x := (\tilde{v}_x \circledcirc \id_{\ev_x^R})\circ (\id_{\ev_x^L} \circledcirc v_x^R) : \ev_x^L \ra \ev_x^R$ is an equivalence.
\end{theorem}

For the proof, we will need a general lemma about adjunctions of 1-morphisms in 3-categories.  By assumption $x$ is 2-dualizable, and so the evaluation 1-morphism $\ev_x$ is guaranteed to have both a left and a right adjoint.  When $x$ is Radford, the Radford equivalence witnesses that the evaluation morphism in fact has an \emph{ambidextrous} adjoint: the left and right adjoints are canonically equivalent.  The following lemma, stated as Remark 3.4.22 in \cite{lurie-ch}, is the essential categorical fact responsible for this ambidexterity.

\begin{lemma} \label{lem-ambiadjoints}
	Let $\cC$ be a 3-category. Let $f: x \to y$ be a 1-morphism in $\cC$, and suppose that $f$ admits a right adjoint $f^R$,  with unit and counit maps $u:\id_x \to f^R \circ f$ and $v:f \circ f^R \to \id_y$. If $u$ and $v$ admit left adjoints $u^L$ and $v^L$, then the 2-morphisms $v^L$ and $u^L$, as unit and counit maps respectively, exhibit $f^R$ as a left adjoint to $f$.\footnote{If desired, this lemma and its proof can be rephrased entirely in terms of 2-categories (namely 2-categories of the form $\Hom_\cC(x,y)$, and the `homotopy 2-category' of $\cC$, whose 2-morphisms are the isomorphism classes of invertible 2-morphisms of $\cC$); in the case where $\cC$ is the 3-category of tensor categories, the relevant 2-categories are definable directly, without reference to the 3-category.  As such this lemma, and thus any results in the book directly referencing it, do not in fact depend on the existence of the 3-category $\TC$.}
\end{lemma}
\begin{proof}
The counit and unit maps of the adjunctions $v^L \dashv v$ and $u^L \dashv u$ are 3-morphisms
\begin{align*}
	\varepsilon_v: v^L \circ v & \Rightarrow \id_{f \circ f^R} \\
	\eta_v: \id_{\id_y} & \Rightarrow v \circ v^L \\
	\varepsilon_u: u^L \circ u & \Rightarrow \id_{\id_x} \\
	\eta_u: \id_{f^R \circ f} & \Rightarrow  u \circ u^L.
\end{align*}
These 3-morphisms satisfy the zigzag identities
\begin{align*}
	 \id_v  &= (\id_v \circ \varepsilon_v) (\eta_v \circ \id_v)  \\
	 \id_{v^L}  &= (\varepsilon_v \circ \id_{v^L} ) (\id_{v^L} \circ \eta_v)  \\
	 \id_u  &= (\id_u \circ \varepsilon_u) (\eta_u \circ \id_u)  \\
	 \id_{u^L}  &= (\varepsilon_u \circ \id_{u^L} ) (\id_{u^L} \circ \eta_u).  
\end{align*}
Here, to avoid a conflict of notation, composition of 3-morphisms in the 3-morphism direction is denoted by juxtaposition, and composition of 3-morphisms in the 2-morphism direction is denoted by $\circ$.

To see that $v^L$ and $u^L$ are the unit and counit of an adjunction $f^R \dashv f$, we need to find isomorphisms ensuring the following zigzag identities:
\begin{align*}
	 (\id_{f} \circledcirc u^L) \circ (v^L \circledcirc \id_{f} ) & \cong \id_{f} \\
 	 (u^L \circledcirc \id_{f^R}) \circ (\id_{f^R} \circledcirc v^L) & \cong \id_{f^R}. 
\end{align*}
As before, here $\circ$ denotes composition of 2-morphisms in the 2-morphism direction, and $\circledcirc$ denotes composition of 2-morphisms in the 1-morphism direction.

Because $u$ and $v$ witness the adjunction $f \dashv f^R$, we may choose isomorphisms
\begin{align*}
	\alpha: \id_f &\stackrel{\cong}{\to} (v \circledcirc \id_f) \circ (\id_f \circledcirc u) \\
	\beta: \id_{f^R} &\stackrel{\cong}{\to} (\id_{f^R} \circledcirc v) \circ (u \circledcirc \id_{f^R} ).
\end{align*}
The desired isomorphisms for the adjunction $f^R \dashv f$ are given by the following composites:
\begin{align*}
	(\id_{f} \circledcirc u^L) \circ (v^L \circledcirc \id_{f} )
		& \cong (\id_{f} \circledcirc u^L) \circ (v^L \circledcirc \id_{f} ) \circ \id_f\\
		& \stackrel{\alpha}{\cong} (\id_{f} \circledcirc u^L) \circ (v^L \circledcirc \id_{f} ) \circ (v \circledcirc \id_f) \circ (\id_f \circledcirc u) \\
		&  \stackrel{\varepsilon_v }{\Rightarrow} (\id_{f} \circledcirc u^L) \circ (\id_f \circledcirc u) \\
		& \stackrel{\varepsilon_u }{\Rightarrow} \id_{f},  \\
	(u^L \circledcirc \id_{f^R}) \circ (\id_{f^R} \circledcirc v^L) 
		& \cong  (u^L \circledcirc \id_{f^R}) \circ (\id_{f^R} \circledcirc v^L) \circ \id_{f^R} \\
		& \stackrel{\beta}{\cong}  (u^L \circledcirc \id_{f^R}) \circ (\id_{f^R} \circledcirc v^L) \circ (\id_{f^R} \circledcirc v) \circ (u \circledcirc \id_{f^R} ) \\
		& \stackrel{\varepsilon_v }{\Rightarrow} (u^L \circledcirc \id_{f^R}) \circ (u \circledcirc \id_{f^R} ) \\
		& \stackrel{\varepsilon_u }{\Rightarrow} \id_{f^R}.
\end{align*}
The inverses to these composites are the composites $\alpha^{-1} \circ \eta_u \circ \eta_v$ and $\beta^{-1} \circ \eta_u \circ \eta_v$, respectively.
\end{proof}

Note that by working in the category $\cC^{\op_1}$, where 1-morphisms are reversed, the role of the right adjoint $f^R$ in this lemma is replaced by the left adjoint $f^L$ and the conclusion of the lemma is that $f^L$ is in fact a right adjoint to $f$.  Similarly, by working in the category $\cC^{\op_2}$, where 2-morphisms are reversed, the role of the left adjoints $u^L$ and $v^L$ may be replaced by the right adjoints $u^R$ and $v^R$.  We may of course reverse both the 1- and 2-morphisms of $\cC$ to obtain a version of the lemma with all instances of left and right exchanged.  In practice, we will use the form of the lemma assuming that the unit and counit maps have right adjoints and concluding that the right adjoint $f^R$ is in fact a left adjoint.

\begin{proof}[Proof of Theorem \ref{thm:Cat_Radford}]
By assumption, we have adjunctions $(\ev_x^L \dashv \ev_x, \allowbreak\tilde{u}_x, \tilde{v}_x)$ and $(\ev_x \dashv \ev_x^R, u_x, v_x)$.  By the 2-morphism opposite version of Lemma~\ref{lem-ambiadjoints}, there is another adjunction $(\ev_x^R \dashv\ev_x,v_x^R,u_x^R)$.  The standard expressions for the canonical equivalence between any two left adjoints provides the inverse equivalences
$(\tilde{v}_x \circledcirc \id_{\ev_x^R})\circ (\id_{\ev_x^L} \circledcirc v_x^R)$
and
$(u_x^R \circledcirc \id_{\ev_x^L}) \circ (\id_{\ev_x^R} \circledcirc \tilde{u}_x)$, and we recognize the first of these as the Radford map.
\end{proof}

\begin{corollary}\label{cor:serreinvserre}
Let $x \in \cC$ be a Radford object of a symmetric monoidal $(\infty,n)$-category, for $n \geq 3$.  There is a canonical equivalence (namely $\widetilde{\cR}_x := (\id_{x} \otimes \ev_{x}) \circ (\tau_{x, x} \otimes \id_{\overline{x}}) \circ (\id_{x} \otimes \cR_x)$) from the inverse Serre automorphism $\cS_x^{-1}$ to the Serre automorphism $\cS_x$.  Similarly there is a canonical equivalence (namely $\cS_x \circ \widetilde{\cR}_x$) from the identity $\id_x$ to the square of the Serre automorphism $\cS_x^2$.
\end{corollary}

Though we may think of the equivalences appearing in Theorem~\ref{thm:Cat_Radford} and Corollary~\ref{cor:serreinvserre} as images of certain bordisms under a local field theory, in fact neither result depends on the cobordism hypothesis.  There are two related lessons here, namely that applications of the cobordism hypothesis often depend not on full dualizability but only on partial dualizability conditions concerning the specific geometry of relevant bordisms, and that those applications can often be proven directly, perhaps with inspiration from but without appeal to the cobordism hypothesis itself.

\chapter{Tensor categories} \label{sec:tc}

We now investigate the target category for our local field theories, namely the 3-category of tensor categories.  We begin, in Section~\ref{sec:conventions}, by establishing systematic, compatible conventions for dualities of objects within tensor categories, and dualities and adjunctions within a higher monoidal category such as the 3-category of tensor categories.  We then, in Section~\ref{sec:tc-lincat}, review the theory of finite tensor categories and bimodule categories, following Etingof, Gelaki, Nikshych, and Ostrik~\cite{MR1976459,MR2183279,MR2097289, 0909.3140,  egno-book}.  We also recall the relative Deligne tensor product~\cite{0909.3140,BTP}, which provides a composition operation on bimodules between tensor categories, and the resulting 3-category of finite tensor categories~\cite{jfs}.  In the dualizability investigation in Chapter~\ref{sec:dualizability}, we will need to construct adjoints of certain functors of bimodule categories, and for that we will need conditions ensuring exactness properties of bimodule functors.  To that end, in Section~\ref{sec:tc-exact}, we summarize Etingof and Ostrik's theory of exact module categories~\cite{EO-ftc} and then prove that the Deligne tensor product of exact module categories is itself exact.  

In Section~\ref{sec:tc-bimodules}, we describe the notion of the dual of a bimodule category.  The dual category is defined as a bimodule structure on the opposite category, but it can also be realized as a category of functors into the base tensor category, thus as a kind of categorical linear dual.  We also prove that relative Deligne tensor products of bimodule categories can be reexpressed as functor categories, which will later enable us to prove that dual bimodule categories provide adjoints of morphisms in our 3-category of tensor categories.  In characteristic zero, that much is more or less sufficient groundwork for the dualizability results of Chapter~\ref{sec:dualizability}.  However, in finite characteristic there are more subtle issues concerning adjoints of bimodule functors.  In Section~\ref{sec:tc-separable}, we introduce the notions of separable tensor categories and separable module categories, based on a suggestion of Ostrik; these separability conditions provide the proper context for the theory of finite tensor categories in arbitrary characteristic.  We prove that the relative Deligne tensor product of separable bimodule categories is separable, ensuring that there is a 3-category whose morphisms are separable bimodule categories.  Finally, in Section~\ref{sec:tc-fusion} we establish a computable criterion for the separability of a fusion category, that is a finite semisimple tensor category with simple unit over an algebraically closed field: a fusion category is separable if and only if it has nonzero global dimension.

The conventions of Section~\ref{sec:conventions} and results of Section~\ref{sec:tc-bimodules} are essential for understanding the dualizability analysis in Chapter~\ref{sec:dualizability}.  Experts could safely skip or skim Sections~\ref{sec:tc-lincat} and~\ref{sec:tc-exact}, perhaps pausing at the statements of Theorem~\ref{thm:tcexists} (the existence of the 3-category of tensor categories) and Theorem~\ref{thm:tensor-exactness} (the tensor of exact module categories is exact).  Readers new to the theory of tensor categories and exact module categories are encouraged to see~\cite{BTP} for a more detailed treatment of these subjects and for proofs of various of the background results used here.  Because they are designed to work in arbitrary characteristic, Sections~\ref{sec:tc-separable} and~\ref{sec:tc-fusion} are more technical; readers exclusively concerned with characteristic zero can largely skip those two sections, consulting Corollaries~\ref{cor:charzerosep} and~\ref{cor:charzeromodulesep} (for characterizations of separable tensor categories and separable bimodules categories in characteristic zero) and noting well the statements of Theorem~\ref{thm:compositeOfSep} and Corollary~\ref{cor:septc} (that separable bimodules compose and therefore provide a sub-3-category of $\TC$).

\section{Conventions for duality} \label{sec:conventions}

In the 3-category $\TC$ of tensor categories, we will encounter various interacting notions of duality and adjunction: there is the dual of a tensor category $\cC$ as an object of the symmetric monoidal 3-category $\TC$, there are duals of objects within the tensor category $\cC$, there is the adjoint of a $\cC$--$\cD$-bimodule category viewed as a 1-morphism of $\TC$, which itself we will discover can be expressed explicitly in terms of duals within the acting categories $\cC$ and $\cD$, there are adjoints of functors of bimodule categories, and so on.  We need a sensible, consistent notation to navigate these dualities---unfortunately the literature has well established conventions that, now brought together in our context, contradict one another.  In this section we describe a suitable compromise among existing conventions and notations.

There are two conventions we take as fundamental, namely one for the names of adjoint functors and one for the notation in rigid monoidal categories.

\begin{definition}[Convention for adjoints] \label{def:Adjoints}
There is an adjunction, denoted $F \dashv G$ between the functors $F: \cB \to \cA$ and $G : \cA \to \cB$ if there are two natural transformations, the unit $\eta: \id_{\cB} \to G \circ F$ and the counit $\varepsilon: F \circ G \to \id_{\cA}$, satisfying the following pair of `zigzag' equations:
	\begin{align*}
		(\id_{G} \circledcirc \varepsilon  ) \circ (  \eta \circledcirc \id_{G}) &= \id_{G} \\
		(\varepsilon \circledcirc \id_{F}) \circ (\id_{F} \circledcirc \eta) &= \id_{F}.
	\end{align*}
(Here $B \circ A$ denotes composition of functors---that is, apply $A$ and then apply $B$---while $\beta \circledcirc \alpha$ denotes horizontal composition of natural transformations---again with $\alpha$ first and then $\beta$.  This ordering ensures, for instance, that the source $s(\beta \circledcirc \alpha)$ of $\beta \circledcirc \alpha$ is the functor $s(\beta) \circ s(\alpha)$.)  We call $F$ the \emph{left adjoint} of $G$, and $G$ the \emph{right adjoint} of $F$.
\end{definition}
\nid The left/right terminology here is guided by the existence of the usual canonical natural isomorphism
\begin{equation*}
	\Hom_\cA(F(x), y) \cong \Hom_\cB(x, G(y)).
\end{equation*}
We apply the same convention for adjoints to 1-morphisms in 2-categories other than $\Cat$; cf. Definition~\ref{def:adjoints_in_bicat} in the Appendix.

\begin{definition}[Convention for rigidity] \label{def:rigid}
	A monoidal category $(\cC, \otimes, 1)$ is {\em rigid} if, for each object $x \in \cC$, (1) there exists an object $x^*$ and morphisms, the {\em coevaluation} $\eta: 1 \to x \otimes x^*$ and the {\em evaluation} $\varepsilon: x^* \otimes x \to 1$, satisfying the pair of zigzag equations
	\begin{align*}
		(id_{x} \otimes \varepsilon  ) \circ (  \eta \otimes id_{x}) &= id_{x} \\
		(\varepsilon \otimes id_{x^*}) \circ (id_{x^*} \otimes \eta) &= id_{x^*};
	\end{align*}
and (2) there exists an object ${}^* x$ and morphisms,
the coevaluation $\eta: 1 \to {}^* x \otimes x$ and the evaluation $\varepsilon: x \otimes {}^* x \to 1$, satisfying the equations
$		(id_{({}^* x)} \otimes \varepsilon  ) \circ (  \eta \otimes id_{({}^* x)}) = id_{({}^* x)}$
and $		(\varepsilon \otimes id_{x}) \circ (id_{x} \otimes \eta) = id_{x}.$
\end{definition}

Although this convention on the meaning of rigidity is shared throughout the literature, authorities are split on whether $x^*$ should be called the right dual of $x$ (following Bakalov--Kirillov \cite{MR1797619} and the Etingof--Gelaki--Nikshych--Ostrik lecture notes \cite{EGNO})  or the left dual of $x$ (following Kassel \cite{MR1321145}, Chari--Pressley \cite{MR1358358}, and the EGNO book \cite{egno-book}).  We believe that this choice should be determined by the following compatibility desideratum: 
\begin{itemize}
\item[]
	\emph{In reasonable settings where something could be called either a dual or an adjoint, the left dual should be the left adjoint.}
\end{itemize}

\nid For example, the category of endofunctors $\Fun(\cA, \cA)$ is a monoidal category in which the dual of a functor is the adjoint functor.  More generally, any monoidal category can be viewed as a $2$-category with one object, and in this setting, the duals of objects in the monoidal category are adjoints of $1$-morphisms in the corresponding $2$-category.  

In order for the above compatibility condition to resolve the left/right terminological ambiguity, we need to decide whether, in the correspondence between monoidal categories and one-object 2-categories, the product $f \otimes g$ corresponds to the composite $f \circ g$ or $g \circ f$.  We view tensor as a geometric operation and therefore adopt the view that if $f$ and $g$ are morphisms or functions or functors or otherwise are composable, then $f \otimes g$ will always mean do $f$ first; in other words, when both expressions make sense, $f \otimes g$ corresponds to $g \circ f$.  More specifically:

\begin{definition}[Convention for $\otimes$ and composition] \label{def-tensorcomp}
If $A$, $B$, and $C$ are $i$-morphisms in an $n$-category $\cC$, and $f: A \ra B$ and $g: B \ra C$ are $(i+1)$-morphisms, then $f \otimes_B g := g \circ f$.  Often the object $B$ is implicit (or trivial) and $f \otimes_B g$ is denoted simply $f \otimes g$.
\end{definition}

This choice determines the following conventions.
\begin{definition}[Convention for duals]
For $x \in \cC$ an object of a rigid monoidal category, the object $x^*$ is the right dual and the object ${}^*x$ is the left dual.
\end{definition}
\begin{definition}[Convention for bimodules]
A bimodule ${}_A M_B$ is a $1$-morph\-ism from $A$ to $B$ in the $2$-category $\Alg$ of algebras, bimodules, and intertwiners.
\end{definition}

As simple examples of how these various conventions interact, observe the following:
\begin{itemize}
\item[1.] Let $\cC$ be a monoidal category, considered as a 2-category with one object.  There is a 2-functor from $\cC$ to $\Cat$ that takes the unique object $\ast \in \cC$ to $\cC \in \Cat$ and takes the morphism $\ast \xra{x} \ast$ of $\cC$ to the functor $\cC \xra{- \otimes x} \cC$.  That is, tensoring on the right provides a 2-functor.  Similarly, tensoring on the left ($x \otimes -$) provides a 2-functor from $\cC$ to $\Cat^{\op_1}$; here `$\op_1$' indicates that the 1-morphisms of $\Cat$ have been reversed.
\item[2.] There is a 2-functor from $\Alg$ to $\Cat$ sending an algebra $A$ to the category $\Mod{}{A}$ of right modules.  Similarly there is a 2-functor from $\Alg$ to $\Cat^{\op_1}$ taking $A$ to the category $\Mod{A}{}$ of left modules.
\end{itemize}

\section[Tensor cats, bimodule cats, and the Deligne tensor product]{Tensor categories, bimodule categories, and the Deligne tensor product} \label{sec:tc-lincat}

We review the theory of linear categories, tensor categories, finite tensor categories, and bimodule categories; an excellent source for further details on this material is the book by Etingof--Gelaki--Nikshych--Ostrik~\cite{egno-book}.  We then recall the notion of the relative Deligne tensor product of module categories over a monoidal category, and quote the main result of~\cite{BTP}, that the relative product of finite module categories over a finite tensor category exists; this provides the most important ingredient for the construction of the 3-category of finite tensor categories, whose existence follows from a construction of Johnson-Freyd--Scheimbauer~\cite{jfs}.

\subsection[Linear categories, finite categories, monoidal categories, rigid categories]{\for{toc}{Linear categories, finite categories, monoidal categories, rigid cats}\except{toc}{Linear categories, finite categories, monoidal categories, rigid categories}}

Let $k$ be an arbitrary field.  Let $\overline{\Vect}_k$ denote the category of (possibly infinite-dimensional) $k$-vector spaces, and let $\Vect_k$ denote the category of finite-dimensional $k$-vector spaces.  A \emph{linear category} is an abelian category with a compatible enrichment over $\overline{\Vect}_k$---compatible in the sense that the enriched addition of morphisms agrees with the given abelian structure.  A \emph{linear functor} is a right exact $\overline{Vect}_k$-enriched functor. 

\begin{warning} \label{warning-rightexact}
Except where explicitly indicated otherwise, \emph{all linear functors in this book are right exact}---a functor postulated as linear is assumed to be right exact, and the claim that a particular functor is linear includes the claim that it is right exact.  We realize this practice deviates from the convention in the literature.  However, the fundamental operation of composition of bimodule categories is only functorial with respect to right exact functors; right exactness is our proper context, and it streamlines the presentation to build it into the notion of linear functor.
\end{warning}

Recall the following standard terminology concerning linear categories:
\begin{itemize}
	\item[-] An object $X$ of a linear category has {\em finite length} if every strictly decreasing chain of subobjects $X = X_0 \supsetneq X_1 \supsetneq X_2 \supsetneq  \cdots$ has finite length. 
	\item[-] A linear category {\em has enough projectives} if for every object $X$, there is a projective object $P$ with a surjection $P \twoheadrightarrow X$. 
	\item[-] An object of a linear category is {\em simple} if it admits no non-trivial subobjects. The endomorphism ring of a simple object is a division algebra over $k$.
	\item[-] A linear category is {\em semisimple} if every object splits as a direct sum of simple objects.  A finite length object of a semisimple category splits as a finite direct sum of simple objects.
\end{itemize}

We can now describe a strong finiteness condition on linear categories:
\begin{definition}
	A linear category $\cC$ is {\em finite} if 
	\begin{enumerate}
		\item[1.] $\cC$ has finite-dimensional spaces of morphisms;
		\item[2.] every object of $\cC$ has finite length;
		\item[3.] $\cC$ has enough projectives; and
		\item[4.] $\cC$ has finitely many isomorphism classes of simple objects.  
	\end{enumerate}
\end{definition}

The following characterization of finite linear categories is well-known; see \cite{BTP} for a proof.
\begin{proposition} \label{prop:catismod}
A linear category is finite if and only if it is equivalent to the category of finite-dimensional modules over a finite-dimensional $k$-algebra $A$.
\end{proposition}

We next consider monoidal structures on linear categories.  A functor $F: \prod_i \cA_i \to \cB$ is \emph{multilinear} if it is linear in each factor $\cA_i$ separately.

\begin{definition}
A \emph{linear monoidal category} is a monoidal category $(\cC, \otimes, 1)$ with a linear structure on $\cC$ such that the multiplication functor $\otimes: \cC \times \cC \ra \cC$ is bilinear.  A \emph{tensor category} is a rigid linear monoidal category.
\end{definition}
\nid When we refer to a `monoidal functor' between linear monoidal categories, we let it be implicit that the functor is linear.  Also, we use the term `tensor functor' to mean a monoidal functor between tensor categories.  Though the rigidity condition imposed in a tensor category may seem incidental to a theory of linear monoidal categories and their bimodules, that is not the case: rigidity is essential to the result that module categories are categories of modules (see Theorem~\ref{thm:EGNO2.11.6} below), and that result is necessary for our approach to constructing the relative Deligne tensor product, thus the composition of bimodule categories (see Theorem~\ref{thm:DelignePrdtOverATCExists} below).

\begin{example}
When $\cM$ is a finite linear category, the category of (right exact) linear endofunctors $\Fun(\cM,\cM)$ is a finite linear monoidal category.  (Recall our convention that the monoidal structure on $\Fun(\cM,\cM)$ is defined by $\cF \otimes \cG := \cG \circ \cF$.)  If $\cM$ is semisimple, then $\Fun(\cM,\cM)$ is rigid---duals are obtained by taking adjoint functors---and hence a tensor category.
\end{example}

For a monoidal category $\cC$, there are three distinct notions of the opposite monoidal category: we can reverse the order of composition, reverse the order of tensor product, or reverse both.  We will denote these respectively by $\cC^\op$ (the opposite category), $\cC^\mp$ (the monoidally-opposite category), and $\cC^\mop$ (the monoidally-opposite, opposite category).  (Note that in the literature `$\op$' sometimes refers to the opposite category and sometimes to the monoidally-opposite category, and some authors have used `$\rev$' to refer to the monoidally-opposite category.)  Observe that when the monoidal category $\cC$ is rigid, taking the left or right dual provides a monoidal equivalence from $\cC^\op$ to $\cC^\mp$ (or vice versa) and similarly provides a monoidal equivalence from $\cC$ to $\cC^\mop$ (or vice versa).

\subsection{Module categories, functors, and transformations}

A module category is a linear category with an action by a linear monoidal category, and a bimodule category is a linear category with two commuting actions by linear monoidal categories:
\begin{definition}
Let $\cC$ and $\cD$ be linear monoidal categories.
A \emph{left $\cC$-module category} is a linear category $\cM$ together with a bilinear functor $\otimes^{\cM}: \cC \times \cM \to \cM$ and natural isomorphisms
\begin{align*}
		\alpha: & \;    \otimes^{\cM} \circ \; (\otimes^{\cC} \times id_{\cM}) \cong  \otimes^{\cM} \circ (id_{\cC} \times \otimes^{\cM}), \\
		\lambda: & \; \otimes^{\cM} (1_{\cC} \times -) \cong id_{\cM},
\end{align*}
satisfying the evident pentagon identity and triangle identities.  We will use the notation $c \otimes m := \otimes^\cM(c \times m)$.  A \emph{right $\cD$-module category} is defined similarly.  A \emph{$\cC$--$\cD$-bimodule category} is a linear category $\cM$ with the structure of a left $\cC$-module category and the structure of a right $\cD$-module category, together with a natural associator isomorphism $(c \otimes m) \otimes d \cong c \otimes (m \otimes d)$ satisfying two additional pentagon axioms and two additional triangle axioms.  
\end{definition}
\nid By a finite module or bimodule category we will mean simply a module or bimodule category whose underlying linear category is finite.  Note well that, in light of our convention that linear functors are right exact, the action functor $\cC \times \cM \ra \cM$ in the above definition is, by assumption, right exact in each variable.  Below, after we restrict the monoidal category $\cC$ to be rigid, we will see that the action functor is forced to be, in fact, exact in each variable.

The notions of functor between module categories and transformation of such functors are as expected:
\begin{definition}
A \emph{left $\cC$-module functor} from the $\cC$-module category $\cM$ to the $\cC$-module category $\cN$ is a (right exact) linear functor $\cF: \cM \ra \cN$ together with a natural isomorphism $f_{c,m}:\cF(c \otimes m) \rightarrow c \otimes \cF(m)$ satisfying the evident pentagon relation and triangle relation.  A right module functor and a bimodule functor are defined similarly.
\end{definition}
\begin{definition}
A \emph{left $\cC$-module transformation} from the $\cC$-module functor $\cF$ to the $\cC$-module functor $\cG$ is a natural transformation $\eta: \cF \ra \cG$ satisfying the condition $(\id_c \otimes \eta_m) \circ f_{c,m} = g_{c,m} \circ \eta_{c \otimes m}$.  Right module transformations and bimodule transformations are defined similarly.
\end{definition}

\nid Observe that for fixed linear monoidal categories $\cC$ and $\cD$, the collection of $\cC$--$\cD$-bimodule categories, bimodule functors, and bimodule transformations forms a strict 2-category.

We include some examples of module and bimodule categories.

\begin{example}
Every linear category can be given the structure of a $\Vect_k$--$\Vect_k$-bimodule category in an essentially unique way: such a structure always exists, and any two such structures are naturally isomorphic by a unique isomorphism.
\end{example}

\begin{example}
Let $A$ be an algebra object in a linear monoidal category $\cC$.  The category $\Mod{}{A}(\cC)$ of right $A$-modules in $\cC$ is a \emph{left} $\cC$-module category.  Similarly, the category $\Mod{A}{}(\cC)$ of left $A$-modules in $\cC$ is a \emph{right} $\cC$-module category.
\end{example}

\begin{example}
Let $\cC$ be a finite linear monoidal category, and let $\cM$ and $\cN$ be finite left $\cC$-module categories.  The categories $\Fun_{\cC}(\cM, \cM)$ and $\Fun_{\cC}(\cN, \cN)$ of (right exact) $\cC$-module endofunctors are finite linear monoidal categories; the linear category $\Fun_{\cC}(\cM, \cN)$ is a $\Fun_{\cC}(\cM, \cM)$--$\Fun_{\cC}(\cN, \cN)$-bimodule category, with the action $\phi \otimes \cF \otimes \psi := \psi \circ \cF \circ \phi$, for $\phi \in \Fun_\cC(\cM,\cM)$, $\psi \in \Fun_\cC(\cN,\cN)$, and $\cF \in \Fun_\cC(\cM,\cN)$.
\end{example}

As mentioned in Proposition~\ref{prop:catismod}, a finite linear category is a category of modules over an algebra.  In this statement we can replace the implicit ambient monoidal category $\Vect_k$ by any finite tensor category: to wit, any finite module category over a finite tensor category is a category of modules over an algebra object of the tensor category.

\begin{theorem}{\cite[Cor. 7.10.5]{egno-book}, \cite[Thm 1]{MR1976459}, \cite{BTP}} \label{thm:EGNO2.11.6}
Let $\cC$ be a finite tensor category, and let $\cM$ be a finite left $\cC$-module category.  There is an algebra object $A \in \cC$ and an equivalence of left $\cC$-module categories $\cM \simeq \Mod{}{A}(\cC)$.
\end{theorem}

\nid This is a crucial result in the structure theory of finite tensor categories and their modules.  The proof relies on Ostrik's observation \cite{MR1976459, EO-ftc} that when $\cC$ is a finite tensor category, any finite $\cC$-module category $\cM$ is canonically enriched over $\cC$.  (See~\cite{BTP} for a proof of Theorem~\ref{thm:EGNO2.11.6} that expressly avoids assumptions about the base field.)  The proof also uses the following lemma, which we will need directly to construct adjoints of bimodule functors arising in our investigation of the dualizability of finite tensor categories.

\begin{lemma} \label{lma:module-adjoint}
Let $\cC$ and $\cD$ be finite tensor categories, and let $\cM$ and $\cN$ be finite $\cC$--$\cD$-bimodule categories.  Let $\cF: \cM \to \cN$ be a not-necessarily-right-exact $\cC$--$\cD$-bimodule functor.  If the linear functor underlying $\cF$ has a right, respectively left, adjoint linear functor (not-necessarily-right-exact), then $\cF$ itself has a right, respectively left, adjoint $\cC$--$\cD$-bimodule functor (not-necessarily-right-exact).
\end{lemma}

\nid This fact appears as a comment in \cite{EO-ftc} and a proof can be found in \cite{BTP}.  

The above theorem can be used to see that the action map of any finite module category is biexact:
\begin{corollary}[\cite{BTP}] \label{cor:biexact_action}
Let $\cC$ be a finite tensor category, and let $\cM$ be a finite $\cC$-module category.  The action map $\cC \times \cM \ra \cM$ is exact in each variable; that is, for each object $c \in \cC$, the functor $c \otimes -$ is exact, and for each object $m \in \cM$, the functor $- \otimes m$ is exact.
\end{corollary}

\subsection[Balanced tensor products and the 3-category of finite tensor categories]{\for{toc}{Balanced tensor products and the 3-category of finite tensor cats}\except{toc}{Balanced tensor products and the 3-category of finite tensor categories}}

In the 2-category of algebras, bimodules, and bimodule maps, the composition of the 1-morphisms $_A M_B$ and $_B N_C$ is given by the balanced tensor product of bimodules, $_A M \otimes_B N_C$.  A categorification of this notion, the balanced tensor product of bimodule categories over monoidal categories (also known as the relative Deligne tensor product), provides the crucial composition of 1-morphisms in the 3-category of finite tensor categories.

\begin{definition}
	Let $\cC$ be a linear monoidal category.  Let $\cM$ be a right $\cC$-module category and let $\cN$ be a left $\cC$-module category. A {\em $\cC$-balanced functor} from $\cM \times \cN$ into a linear category $\cL$ is a bilinear functor $\cF: \cM \times \cN \to \cL$ (right exact in each variable), together with a natural isomorphism $\alpha: \cF \circ (\otimes^{\cM} \times \id_{\cN}) \cong \cF \circ (\id_{\cM} \times \otimes^{\cN})$ satisfying the evident pentagon and triangle identities. A {\em $\cC$-balanced transformation} is a natural transformation $\eta:\cF \to \cG$ of $\cC$-balanced functors commuting with the isomorphism $\alpha$.
\end{definition}

The relative Deligne tensor product $\cM \boxtimes_\cC \cN$ is the initial linear category admitting a $\cC$-balanced bilinear functor from $\cM \times \cN$:

\begin{definition}
	For $\cC$ a linear monoidal category, let $\cM$ be a right $\cC$-module category and let $\cN$ be a left $\cC$-module category.  The {\em relative Deligne tensor product} is a linear category $\cM \boxtimes_{\cC} \cN$ together with a $\cC$-balanced bilinear functor $\boxtimes_{\cC}: \cM \times \cN \to \cM \boxtimes_{\cC} \cN$ that induces, for all linear categories $\cL$, an equivalence between the category of $\cC$-balanced bilinear functors $\cM \times \cN \to \cL$ and the category of (right exact) linear functors $\cM \boxtimes_{\cC} \cN \to \cL$. 
\end{definition}

\nid If it exists, the relative Deligne tensor product $\cM \boxtimes_\cC \cN$ is unique up to equivalence, and that equivalence is in turn unique up to unique natural isomorphism; in other words, the 2-category of linear categories representing the relative Deligne tensor product is either contractible or empty.  Observe that because it is defined by a universal property, when it exists the relative Deligne tensor product $\cM \boxtimes_\cC \cN$ is functorial in the module categories $\cM$ and $\cN$ and in the monoidal category $\cC$.

Provided the monoidal category $\cC$ is finite and rigid, and provided the module categories are finite, the relative Deligne tensor product does indeed exist.

\begin{theorem}[\cite{BTP}] \label{thm:DelignePrdtOverATCExists}
	Let $\cC$ and $\cD$ be finite tensor categories.  Let $\cM_{\cC}$ and ${}_{\cC}\cN$ be finite right and left $\cC$-module categories, respectively, and let ${}_{\cD}\cP$ be a finite left $\cD$-module category. 
	\begin{enumerate}
		\item The relative Deligne tensor product $\cM \boxtimes_{\cC} \cN$ exists and is a finite linear category.
		\item If $\cM = \Mod{A}{}(\cC)$ and $\cN = \Mod{}{B}(\cC)$, for algebra objects $A \in \cC$ and $B \in \cC$, then $\cM \boxtimes_{\cC} \cN \simeq \Mod{A }{B}(\cC)$. Here $\Mod{A}{B}(\cC)$ denotes the category of $A$--$B$-bimodule objects in $\cC$.
		\item If $\cN = \Mod{}{B}(\cC)$ and $\cP = \Mod{}{C}(\cD)$, for algebra objects $B \in \cC$ and $C \in \cD$, then $\cN \boxtimes \cP \simeq \Mod{}{(B \boxtimes C)}(\cC \boxtimes \cD)$ as a left $\cC \boxtimes \cD$-module category.
		\item The functor $\boxtimes_{\cC}: \cM \times \cN \ra \cM \boxtimes_\cC \cN$ is exact in each variable; when $\cC = \Vect$ this functor satisfies 
		\begin{equation*}
			\Hom_{\cM \boxtimes \cN} (m \boxtimes n, m' \boxtimes n') \cong \Hom_{\cM}(m,m') \otimes \Hom_{\cN}(n, n').
		\end{equation*}
		\item If $\cM_\cC \ra \cM'_\cC$ and ${}_\cC \cN \ra {}_\cC \cN'$ are exact $\cC$-module functors, then the corresponding functor $\cM \boxtimes_\cC \cN \ra \cM' \boxtimes_\cC \cN'$ is exact.
		\item Let $\cM$, $\cN$, and $\cL$ now be linear categories and assume the base field is perfect.  If a bilinear functor $\cM \times \cN \ra \cL$ is exact in each variable, then the corresponding functor $\cM \boxtimes \cN \ra \cL$ is exact \cite[Prop 5.13(vi)]{deligne}.
	\end{enumerate} 
\end{theorem}
\nid Tambara constructed an additive (rather than abelian) version of a relative Deligne tensor product for semisimple module categories over semisimple monoidal categories~\cite{tambara}.  Etingof--Nikshych--Ostrik constructed the relative Deligne tensor product for semisimple module categories over a semisimple tensor category with simple unit over an algebraically closed field of characteristic zero~\cite{0909.3140}.  

The relative Deligne tensor product provides a means of composing bimodule categories; this is the most important ingredient in constructing the 3-category of finite tensor categories.  We need not just a 3-category, though, but a symmetric monoidal 3-category.  The monoidal structure is given by the relative Deligne tensor product over the tensor category $\Vect_k$; such a product is denoted $\cM \boxtimes \cN$ and referred to simply as the Deligne tensor product.  For the Deligne tensor to provide a monoidal structure on the objects of the 3-category of finite tensor categories, we need to know that the product $\cC \boxtimes \cD$ is indeed rigid.  To ensure that we make our first and only assumption about the base field $k$, namely that it is perfect; cf. Deligne~\cite[Prop 5.17]{deligne}.\footnote{For example: if $k$ is not perfect, and $l$ is an inseparable field extension of $k$, then the Deligne tensor $(\Mod{l}{l}) \boxtimes (\Mod{l}{l})$ is not rigid.}

\begin{theorem}[\cite{jfs}] \label{thm:tcexists}
Provided the base field $k$ is perfect, there exists a symmetric monoidal $(3,3)$-category $\TC$ whose objects are the finite tensor categories, whose 1-morphisms are finite bimodule categories, whose 2-morphisms are bimodule functors, and whose 3-morphisms are bimodule transformations.  Composition of 1-morphisms is given by the relative Deligne tensor product, and the symmetric monoidal structure is given by the Deligne tensor product.
\end{theorem}

\nid Over any base field, Johnson-Freyd--Scheimbauer's construction of higher categories of higher algebra objects~\cite{jfs} provides an $(\infty,3)$-category $\mathrm{Alg}_1^{\text{strong}}(\mathrm{Rex})$ of algebra objects (with bimodules and strong bimodule functors) in the $(\infty,2)$-category of finitely-cocomplete linear categories (with right exact functors).  When the base field is perfect, the finite tensor categories, finite bimodule categories, bimodule functors, and bimodule transformations form a symmetric monoidal sub-$(3,3)$-category $\TC$ of the larger $(\infty,3)$-category $\mathrm{Alg}_1^{\text{strong}}(\mathrm{Rex})$.
The existence of a symmetric monoidal 3-category of tensor categories was anticipated by Etingof--Nikshych--Ostrik~\cite{0909.3140}, Greenough~\cite{0911.4979}, Douglas--Henriques~\cite{dh-ib}, and Schaumann~\cite{Schaumann-PhD}.

\section{Exact module categories} \label{sec:tc-exact}

When investigating the dualizability properties of a finite tensor category $\cC \in \TC$, we will first construct a dual object, and then build chains of left and right adjoints to the evaluation and coevaluation maps of that duality.  We will then find ourselves wanting to build adjoints to the unit and counit maps of those adjunctions.  By assumption, those unit and counit maps are right exact functors, and so will admit right adjoint functors which however are not obviously right exact, therefore are not obviously adjoints within our 3-category $\TC$.  What we need is a condition on a bimodule category that ensures that all not-necessarily right or left exact bimodule functors out of it are in fact exact.  If we can verify that certain bimodules arising in the duality data for $\cC \in \TC$ satisfy this condition, then we will be able to construct the corresponding adjoints of bimodule functors.

Such a condition, ensuring the exactness of module functors out of a module category, is provided by Etingof and Ostrik's (aptly named) notion of an exact module category.  In this section we summarize the properties of exact module categories, following~\cite{EO-ftc}, and then prove a new result concerning the exactness of Deligne tensor products.

\subsection{Properties of exact module categories over finite tensor categories}

\begin{definition}
Let $\cC$ be a finite tensor category.  A left $\cC$-module category $\cM$ is \emph{exact} 
if the object $p \otimes m \in \cM$ is projective for any projective $p \in \cC$ and any $m \in \cM$.
\end{definition}

\begin{definition}
A $\cC$--$\cD$-bimodule category is \emph{exact} if it is exact when considered as a left $\cC \boxtimes \cD^\mp$-module category.
\end{definition}

The most elementary condition on a module category that ensures functors out of it are exact is of course semisimplicity.  In fact that is the exactness condition in question when the tensor category itself is semisimple:

\begin{example} \label{eg:semiexact}
Let $\cC$ be a semisimple finite tensor category.  A $\cC$-module category is exact if and only it is semisimple.
\end{example}

\nid We will need to know that certain basic bimodules arising in the dualization of $\cC$ are exact:

\begin{example}[{\cite[Ex. 3.3]{EO-ftc} \cite[Ex. 7.5.5]{egno-book}}] \label{ex:exactness}
	A finite tensor category $\cC$ is exact when considered as a $\cC$-, or $\cC^{mp}$-, or $\cC \boxtimes \cC^{mp}$-module category. 
\end{example}

The following omnibus theorem summarizes the properties of exact module categories: 
\begin{theorem}[\cite{EO-ftc, egno-book}] \label{Thm:ExactModCatOmnibus}
	Let $\cC$ be a finite tensor category over an arbitrary field, let $\cM$, $\cM'$, and $\cM''$ be finite exact $\cC$-module categories, and let $\cN$ be any finite $\cC$-module category.
	\begin{enumerate}
		\item Every projective object of $\cM$ is injective, and every injective object of $\cM$ is projective \cite[Cor. 3.6]{EO-ftc} \cite[Cor. 7.6.4]{egno-book}.
		\item Every additive (not a priori right exact) module functor $F:\cM \to \cN$ is exact \cite[Prop. 3.11]{EO-ftc} \cite[Prop. 7.6.9]{egno-book}.		
		\item The category $\Fun_{\cC}(\cM, \cM')$ is finite \cite[Prop. 7.11.6]{egno-book}.
		\item The composition functor $\Fun_{\cC}(\cM, \cM') \times \Fun_{\cC}(\cM', \cM'') \to \Fun_{\cC}(\cM, \cM'')$ is exact in each variable \cite[Lemma 3.20]{EO-ftc} \cite[Lemma 7.11.3]{egno-book}.	
		\item The monoidal category of endofunctors $\Fun_{\cC}(\cM,\cM)$ is rigid, therefore is a finite tensor category.
		\item The category $\cM$ is exact when considered as a $\Fun_{\cC}(\cM,\cM)^\mp$-module category \cite[Lemma 3.25]{EO-ftc}.
	\end{enumerate}
\end{theorem}

\nid The fifth item is seen as follows.  Any functor $\cF \in \Fun_\cC(\cM,\cM)$ is exact, therefore has left and right adjoints as a linear functor.  By Lemma~\ref{lma:module-adjoint}, the functor $\cF$ also has adjoints as a module functor---these adjoints provide the required rigidity.

The second item is proven as \cite[Prop. 3.11]{EO-ftc} and \cite[Prop. 7.6.9]{egno-book} under the assumption that the unit of the tensor category $\cC$ is simple; the property, that functors out of an exact module category are exact, can be seen when $\cC$ is any finite tensor category as follows.  By the finiteness assumption, the tensor category $\cC$ has enough projectives; thus there is a short exact sequence $0 \ra k \ra p \ra 1 \ra 0$, with $p \in \cC$ projective.  Let $0 \ra x \ra y \ra z \ra 0$ be a short exact sequence in $\cM$.  Tensoring each object of the first sequence by the second sequence, and applying the functor $F$ to the resulting grid, produces the following diagram:
	\begin{center}
	\begin{tikzpicture}
			\matrix (m) [matrix of math nodes, row sep = 0.5cm, column sep = 0.75cm, text height=1.5ex, text depth=0.25ex]
			{
				  & 0 & 0 & 0 & \\
				0 & k \otimes F(x) & k \otimes F(y) & k \otimes F(z) & 0 \\
				0 & p \otimes F(x) & p \otimes F(y) & p \otimes F(z) & 0 \\
				0 &  F(x) &  F(y) &  F(z) & 0 \\ 
				  & 0 & 0 & 0 & \\
			};
			\foreach \x in {2,3,4}{ 
				\foreach \y/\z in {1/2,2/3,3/4,4/5}{
					\draw [->] (m-\y-\x) -- (m-\z-\x);
				};		
			};
			\foreach \y in {2,4}{ 
				\foreach \x/\z in {1/2,2/3,3/4,4/5}{
					\draw [->] (m-\y-\x) -- node [above] {} (m-\y-\z);
				};		
			};
			\foreach \x/\z/\text in {1/2,2/3,3/4,4/5}{
				\draw [->] (m-3-\x) -- (m-3-\z);
			};
	\end{tikzpicture}
	\end{center}
Because $\cM$ is exact, the tensor $p \otimes z$ is projective, so the sequence $0 \ra p \otimes x \ra p \otimes y \ra p \otimes z \ra 0$ is split exact; thus, because the functor $F$ is additive, the middle row of this diagram is exact. The columns are all exact because the action map is biexact, by Corollary~\ref{cor:biexact_action}.  A diagram chase shows that $F(y) \ra F(z) \ra 0$ is exact; applying Corollary~\ref{cor:biexact_action}, it follows that $k \otimes F(y) \ra k \otimes F(z) \ra 0$ is exact.  Further diagram chasing then shows that $F(x) \ra F(y) \ra F(z)$ is exact, and so, again by Corollary~\ref{cor:biexact_action}, it follows that $k \otimes F(x) \ra k \otimes F(y) \ra k \otimes F(z)$ is also exact.  A final diagram chase gives that $0 \ra F(x) \ra F(y)$ is exact, completing the argument.

\subsection{The tensor product of exact module categories is exact} \label{sec:tensorexact}

At a crucial moment in our analysis of the dualizability of a finite tensor category $\cC \in \TC$, we will want to construct a right adjoint to a $\cC \boxtimes \cC^\mp$-bimodule map from the identity bimodule to a bimodule $\ev_\cC \boxtimes \ev_\cC^R$, where $\ev_\cC$ is an evaluation for a duality between $\cC$ and $\cC^\mp$, and $\ev_\cC^R$ is a right adjoint of that bimodule.  We would therefore like to know that the bimodule $\ev_\cC \boxtimes \ev_\cC^R$ is exact.  We will know by then that each of the two factors in that bimodule is exact.  We now undertake to prove the general result that any ordinary Deligne tensor product of exact module categories is exact.

We will need the following lemma, ensuring that it suffices to check the exactness condition for a collection of projective generators.

\begin{lemma} \label{lma:Exact_checked_on_proj_gens}
	Let $\cC$ be a finite tensor category and let $\cM$ be a finite left $\cC$-module category. Let $\cP = \{ p_\alpha \}$ be a collection of jointly generating projective objects of  $\cC$. The following conditions are equivalent:
	\begin{enumerate}
		\item for all $m \in \cM$ and all $p_\alpha \in \cP$, the object $p_\alpha \otimes m \in \cM$ is projective;
		\item the $\cC$-module category $\cM$ is exact.
	\end{enumerate}
\end{lemma}

\noindent (Recall that a family $\cP$ is called {\em jointly generating} if the functor $\prod_\alpha \Hom_\cC(p_\alpha, -): \cC \to \Set$ is faithful.)

\begin{proof}
The second condition certainly implies the first; we prove the converse.  By the joint generation assumption, any object $q \in \cC$ is a quotient $\pi: p \twoheadrightarrow q$ of a sum $p := \oplus p_\alpha$ of objects $p_\alpha \in \cP$ in the generating set.  Suppose $q \in \cC$ is projective, and $m \in \cM$ is any object.  The map $\pi: p \ra q$ is split, and so the map $\pi \otimes \id_m: p \otimes m \ra q \otimes m$ is split.  Thus the product $q \otimes m$ is a summand of the projective object $p \otimes m$, and so is projective as required.
\end{proof}

\begin{theorem}\label{thm:tensor-exactness}
Let $\cC$ and $\cD$ be finite tensor categories over a perfect field.  If $\cM$ is an exact finite left $\cC$-module category, and $\cN$ is an exact finite left $\cD$-module category, then $\cM \boxtimes \cN$ is an exact left $\cC \boxtimes \cD$-module category.
\end{theorem}

\begin{proof} 
First observe that the Deligne tensor product preserves projective and injective objects in the following sense: if $m \in \cM$ and $n \in \cN$ are both projective, equivalently injective, objects, then $m \boxtimes n \in \cM \boxtimes \cN$ is also projective and injective.  This closure property follows from parts (4) and (6) of Theorem \ref{thm:DelignePrdtOverATCExists}.  (By part (4), the functor $\Hom_{\cM \boxtimes \cN}(m \boxtimes n,-): \cM \boxtimes \cN \to \Vect$ corresponds to the bilinear functor $\Hom_{\cM}(m,-) \otimes \Hom_{\cN}(n,-): \cM \times \cN \to \Vect$.  The latter functor is exact in each variable because $m$ and $n$ are projective; by part (6) the former functor is therefore exact, and thus $m \boxtimes n$ is projective as required.  Injectivity follows similarly.)

Let $\cP$ be the collection of objects of $\cC \boxtimes \cD$ of the form $p \boxtimes q$ for all projective objects $p \in \cC$ and $q \in \cD$.  This collection is a jointly generating family of projectives, as follows from expressing $\cC \boxtimes \cD$ as $(A \otimes B)\textrm{-}\mathrm{Mod}$ (using Proposition~\ref{prop:catismod} and part (2) of Theorem~\ref{thm:DelignePrdtOverATCExists}).

By Lemma \ref{lma:Exact_checked_on_proj_gens}, it suffices to check that for every object $x \in \cM \boxtimes \cN$, and every pair of projective objects $p \in \cC$ and $q \in \cD$, the product $(p \boxtimes q) \otimes x \in \cM \boxtimes \cN$ is projective.  That is, we need to show that the following functor is exact:
\begin{equation*}
	\Hom((p \boxtimes q) \otimes x, -) \cong \Hom(x, ({}^*p \boxtimes {}^*q) \otimes(-) ): \cM \boxtimes \cN \to \Vect
\end{equation*}
We first observe that this is true when $x$ is a primitive tensor $a \boxtimes b$: because $\cM$ and $\cN$ are exact, the products $p \otimes a$ and $q \otimes b$ are projective, and the product $(p \boxtimes q) \otimes (a \boxtimes b) \cong (p \otimes a) \boxtimes (q \otimes b)$ is projective by the aforementioned closure property.

Now for a general object $x \in \cM \boxtimes \cN$, by part (6) of Theorem~\ref{thm:DelignePrdtOverATCExists} it is enough to show that the functor
\begin{align*}
\cF: \cM \times \cN &\rightarrow \Vect \\
m \times n &\mapsto \Hom_{\cM \boxtimes \cN}((p \boxtimes q) \otimes x, m \boxtimes n) \cong \Hom_{\cM \boxtimes \cN}(x, ({}^* p \otimes m) \boxtimes ({}^* q \otimes n))
\end{align*}
is exact in each of the variables $\cM$ and $\cN$ separately.  We show exactness in the variable $\cM$; exactness in the variable $\cN$ follows by the same argument.

Fix an object $n \in \cN$ and consider an exact sequence in $\cM$,
\begin{align*}
	0 \to m_1 \to m_2 \to m_3 \to 0. 
\end{align*}
Set $z_i := ({}^* p \otimes m_i) \boxtimes ({}^* q \otimes n)$ and consider the resulting sequence
\begin{equation*}
	0 \to z_1 \to z_2 \to z_3 \to 0.
\end{equation*}
This sequence is exact because $(c \otimes -): \cM \ra \cM$ is exact for any object $c \in \cC$ (as tensoring with appropriate duals gives the adjoint functors) and $(- \boxtimes \tilde{n}) : \cM \ra \cM \boxtimes \cN$ is exact for any $\tilde{n} \in \cN$ (by part (4) of Theorem~\ref{thm:DelignePrdtOverATCExists}).  Moreover, the objects $z_i$ are all injective: the dual of a projective object of a finite tensor category is projective \cite[Prop. 2.3]{EO-ftc}, hence the objects ${}^* p \otimes m_i$ and ${}^* q \otimes n$ are projective, hence injective by part (1) of Theorem~\ref{Thm:ExactModCatOmnibus}, and so the product $({}^* p \otimes m_i) \boxtimes ({}^* q \otimes n)$ is injective by the closure property.

We are left to show that for any object $x \in \cM \boxtimes \cN$, the sequence $0 \to \Hom(x,z_1) \to \Hom(x,z_2) \to \Hom(x,z_3) \to 0$ is exact.  Using Theorem~\ref{thm:EGNO2.11.6} and part (2) of Theorem~\ref{thm:DelignePrdtOverATCExists}, note that there is a short exact sequence
\begin{equation*}
	0 \to x'' \to x' \to x \to 0
\end{equation*}
where $x'$ is a primitive tensor, that is an object of the form $a \boxtimes b$, for which we already know the desired exactness property.  Consider, finally, the following diagram:
\[
		\begin{tikzpicture} \matrix (m) [matrix of math nodes, row sep={1cm,between origins}] {
		 &[0.5cm] 0 &[1cm] 0 &[1cm] 0 &[0.5cm]  \\ 
		 0 & \Hom(x, z_1) & \Hom(x, z_2) & \Hom(x, z_3) & \\ 
		 0 & \Hom(x', z_1) & \Hom(x', z_2) & \Hom(x', z_3) & 0 \\
		 0 & \Hom(x'', z_1) & \Hom(x'', z_2) & \Hom(x'', z_3) & \\
		& 0 & 0 & 0 & \\
		};
		\foreach \sourcerow/ \sourcecol / \targetrow / \targetcol in 
			{1/2/2/2, 1/3/2/3, 1/4/2/4, 
			2/2/3/2, 2/3/3/3, 2/4/3/4,  
			3/2/4/2, 3/3/4/3, 3/4/4/4,  
			4/2/5/2, 4/3/5/3, 4/4/5/4,  
			2/1/2/2, 2/2/2/3, 2/3/2/4, 
			3/1/3/2, 3/2/3/3, 3/3/3/4, 3/4/3/5, 
			4/1/4/2, 4/2/4/3, 4/3/4/4} 
			\draw [->] (m-\sourcerow-\sourcecol) -- (m-\targetrow-\targetcol);
		\end{tikzpicture}
\]
The columns are exact because the objects $z_i$ are injective.  The top and bottom rows are exact because the functor $\Hom(\tilde{x},-)$ is left exact for any $\tilde{x} \in \cM \boxtimes \cN$.  The middle row is exact because $x'$ is a primitive tensor.  A diagram chase shows that the map $\Hom(x,z_2) \ra \Hom(x,z_3)$ is surjective, and hence $\Hom((p \boxtimes q) \otimes x,-)$ is exact, as required.
\end{proof}

\begin{remark}
Over an imperfect field, the conclusion of this theorem is false even when $\cC \boxtimes \cD$ happens to be rigid.  For example, when $\ell$ is an inseparable field extension of the base field $k$, the linear category $\Vect_\ell$ is semisimple, therefore exact as a $\Vect_k$-module category, but the product $\Vect_\ell \boxtimes \Vect_\ell \simeq \Mod{(\ell \otimes_k \ell)}{}$ is not semisimple, therefore not exact as a $\Vect_k$-module category.
\end{remark}

\section{Dual and functor bimodule categories} \label{sec:tc-bimodules}

In this section we describe a number of operations on bimodule categories (flipping an action, twisting an action, taking a dual category, taking a functor category) and explain how they are related.  We then prove, in Proposition~\ref{prop:FunctorsAsATensorPdt}, that the relative Deligne tensor product can be expressed as a category of functors.  This fact will be crucial, in Section~\ref{sec:df-modules}, in proving the $2$-dualizability of finite tensor categories.  Finally, we describe how the (right) dual of a (left) module category can be interpreted as the category of modules over the double dual of an algebra object.  We will need this interpretation in giving, in Section~\ref{sec:radfordftc}, a topological proof of the quadruple dual theorem.

\subsection{Flips and twists of bimodule categories} \label{sec:fliptwist}

Given a bimodule ${}_A M_B$ between ordinary algebras $A$ and $B$, we can flip the left $A$-action onto the right, producing a module $M_{A^\op \otimes B}$, or flip the right $B$-action onto the left, producing a module ${}_{A \otimes B^\op} M$, or flip both actions to the other side, producing the bimodule ${}_{B^\op} M_{A^\op}$.  More generally, if we have a bimodule of the form ${}_{A \otimes B} M_C$, we can flip just part of the action to obtain the bimodule ${}_A M_{B^\op \otimes C}$, or, of course, flip a bimodule ${}_A M_{B \otimes C}$ to a bimodule ${}_{A \otimes B^\op} M_C$.  The same operations work perfectly well on bimodule categories.

\begin{definition}
Let $\cC$, $\cD$, and $\cE$ be linear monoidal categories.  Given a bimodule category ${}_{\cC \boxtimes \cD} \cM_\cE$, the corresponding bimodule category ${}_{\cC} \cM_{\cD^\mp \boxtimes \cE}$ is called a \emph{flip} of ${}_{\cC \boxtimes \cD} \cM_\cE$.  Similarly, given a bimodule category ${}_\cC \cM_{\cD \boxtimes \cE}$, the corresponding bimodule category ${}_{\cC \boxtimes \cD^\mp} \cM_\cE$ is a flip of ${}_\cC \cM_{\cD \boxtimes \cE}$.
\end{definition}

\noindent We will be particularly interested in the bimodules obtained by flipping the actions of the identity bimodule category ${}_\cC \cC_\cC$.  If we flip the right action to the left, then we have a module category ${}_{\cC \boxtimes \cC^\mp} \cC$; if we flip the left action to the right, then we have $\cC_{\cC^\mp \boxtimes \cC}$; if we flip both actions, we have a bimodule ${}_{\cC^\mp} \cC_{\cC^\mp}$, which is in fact precisely the identity bimodule of the category $\cC^\mp$.

Given a bimodule ${}_A M_B$ and a homomorphism of algebras $f: A' \rightarrow A$, we can twist, that is precompose, the $A$-action by the map $f$ to obtain an $A'$--$B$-bimodule denoted ${}_{\langle f \rangle} M$.  Similarly, given a homomorphism $g: B' \rightarrow B$, we can twist the right action to obtain an $A$--$B'$-bimodule denoted $M_{\langle g \rangle}$.  The same operations can be applied to bimodule categories.
\begin{definition}
Let $\cC$, $\cC'$, $\cD$, and $\cD'$ be linear monoidal categories, let ${}_\cC \cM_\cD$ be a bimodule category, and let $\cF: \cC' \rightarrow \cC$ and $\cG: \cD' \rightarrow \cD$ be monoidal functors.  The $\cC'$--$\cD$-bimodule obtained by precomposition with $\cF$ is called a \emph{twist} of ${}_\cC \cM_\cD$ and is denoted ${}_{\langle \cF \rangle} \cM$; in particular, the action on objects of ${}_{\langle \cF \rangle} \cM$ is $c' \cdot m \cdot d := \cF(c') \otimes m \otimes d$.  Similarly, a twist of the right action by the monoidal functor $\cG: \cD' \rightarrow \cD$ provides a $\cC$--$\cD'$-bimodule category denoted $\cM_{\langle \cG \rangle}$.
\end{definition}

Because of the relations ${}_{\langle f \rangle} A \otimes_A M \cong {}_{\langle f \rangle} M$ and ${}_{\langle f \rangle \langle g \rangle} M \cong {}_{\langle g \circ f \rangle} M$, there is a functor from the category of algebras and algebra homomorphisms to the 2-category of algebras, bimodules, and intertwiners: the functor is the identity on objects and sends a homomorphism to the corresponding twisted bimodule.  The analogous facts for bimodule categories, that ${}_{\langle \cF \rangle} \cC \boxtimes_\cC \cM \simeq {}_{\langle \cF \rangle} \cM$ and that ${}_{\langle \cF \rangle \langle \cG \rangle} \cM \simeq {}_{\langle \cF \otimes \cG \rangle} \cM$, ensure there is a functor from the 2-category of finite tensor categories, tensor functors, and tensor natural transformations, to the 3-category $\TC$ of finite tensor categories, finite bimodule categories, their functors, and their transformations: again the functor is the identity on objects and takes a functor to the corresponding twisted bimodule category.  Similarly, the fact that $\cM \boxtimes_\cD \cD_{\langle \cF \rangle} \simeq \cM_{\langle \cF \rangle}$ and $\cM_{\langle \cF \rangle \langle \cG \rangle} \simeq \cM_{\langle \cG \otimes \cF \rangle}$ ensures there is a functor from the 2-category of finite tensor categories to $\TC^{\op_1}$, the 3-category of tensor categories with the 1-morphisms reversed.

In Section~\ref{sec:topquaddual}, the Radford bordism will provide an isomorphism between two twisted bimodules, and we will want to understand the significance of this isomorphism for the corresponding twisting functors.

\begin{lemma} \label{lem:BimoduleToFunctor}
Let $\cC$ and $\cD$ be linear monoidal categories and let $\cF: \cC \rightarrow \cD$ and $\cG: \cC \rightarrow \cD$ be monoidal functors.  The category of bimodule equivalences $\alpha: {}_\cC ({}_{\langle \cF \rangle} \cD)_\cD \rightarrow {}_\cC ({}_{\langle \cG \rangle} \cD)_\cD$ is equivalent to the category of pairs $(D, \varphi)$, where $D \in \cD$ is a tensor-invertible object and $\varphi$ is a monoidal natural isomorphism from $\cF$ to the monoidal functor $D^{-1} \otimes \cG(-) \otimes D$.
\end{lemma}
\begin{proof}
Such a bimodule equivalence $\alpha$ is in particular a right $\cD$-module functor from $\cD$ to $\cD$ and is therefore equivalent to the left multiplication functor $d \mapsto \alpha(1) \otimes d$.  The left $\cC$-module structure of $\alpha$ is a natural isomorphism $\alpha(c \cdot -) \cong c \cdot \alpha(-)$, that is $\alpha(\cF(c) \otimes -) \cong \cG(c) \otimes \alpha(-)$, which therefore provides a natural isomorphism $a_{(c,-)}: \alpha(1) \otimes \cF(c) \otimes - \cong \cG(c) \otimes \alpha(1) \otimes -$.

The functor from bimodule equivalences to pairs takes the equivalence $\alpha$ to the pair $(D,\varphi)$ where $D:=\alpha(1)$ and $\varphi(c) := D^{-1} \otimes a_{(c,1)}$.  Note that $D$ is tensor-invertible because $\alpha$ is an equivalence.  Given a pair $(D,\varphi)$, the associated bimodule equivalence is $\alpha(d) := D \otimes d$ with left module structure map $D \otimes \varphi$.
\end{proof}

\subsection{Duals of bimodule categories}

Given a $\cC$--$\cD$-bimodule category, we can form a new bimodule category by twisting the $\cC$ and $\cD$ actions by the (say left) dual functors $\cC \ra \cC^\mop$ and $\cD \ra \cD^\mop$.  The resulting `dual' bimodule categories will play a central role in our dualizability analysis.

\begin{definition} \label{def:Dual_bimodule_notation}
Given a $\cC$--$\cD$-bimodule category $\cM$, the \emph{right dual} of $\cM$ is a $\cD$--$\cC$-bimodule category denoted $\cM^*$ and defined as follows.  The underlying linear category of $\cM^*$ is $\cM^\op$; an object $m \in \cM$ is denoted $m^*$ when viewed as an object of $\cM^*$.  The bimodule category structure on $\cM^*$ is given by
\[
d \cdot m^* \cdot c := ({}^* c \otimes m \otimes {}^* d)^*
\]
Similarly, the \emph{left dual} of $\cM$ is the $\cD$--$\cC$-bimodule category ${}^* \cM$ with underlying category $\cM^\op$ and actions
\[
d \cdot {}^* m \cdot c := {}^* (c^* \otimes m \otimes d^*).
\]
\end{definition}

\noindent The terminology, referring to these bimodule categories as duals, will be justified later in two ways: first, we will see that the bimodule categories are categories of functors into the base tensor categories, therefore are linear duals in the classical sense, and second, we will see that these bimodule categories are adjoint bimodules.  

There is a quite important, potentially confusing, subtlety arising here, namely that the process of taking a dual bimodule category does not commute with the process of flipping an action between the left and the right of a bimodule; this issue is probably responsible for some of the inaccuracies in the literature concerning dual bimodules.  Because flipping an action changes $\cC$ to $\cC^\mp$, and a left dual in $\cC$ is a right dual in $\cC^\mp$, the failure of commutativity is in fact a double dual.
\begin{lemma} \label{lemma:DualTwist}
The $\cD$--$\cC$-bimodule flip of the right dual of the right $(\cC^\mp \boxtimes \cD)$-module flip of a $\cC$--$\cD$-bimodule category $\cM$ is the right twist by the right double dual functor of the right dual of the bimodule category $\cM$; that is,
\[
{}_\cD ((\cM_{\cC^\mp \boxtimes \cD})^*)_\cC \simeq (({}_\cC \cM_\cD)^*)_{\langle \mathfrak{r r} \rangle},
\]
where $\mathfrak{r}$ denotes the right dual functor.  Similarly
\[
{}_\cD (({}_{\cC \boxtimes \cD^\mp}\cM)^*)_\cC \simeq {}_{\langle \mathfrak{r r} \rangle}(({}_\cC \cM_\cD)^*).
\]

\end{lemma}
\noindent The proof is simply chasing carefully through the definitions.  Analogous formulas hold for left dual categories.

We will need to know that the duals of exact module categories are exact.  (Recall that a $\cC$-module category $\cM$ is called exact if $p \otimes m \in M$ is projective whenever $p \in \cC$ is projective.)  Toward that end, we first prove an auxiliary result concerning exactness of contravariantly, monoidally-oppositely twisted modules.  For a left $\cC$-module $\cM$ and $\cF: \cC \rightarrow \cC^\mop$ a tensor equivalence, let $\cM^{\op(\cF)}$ denote the right $\cC$-module with underlying linear category $\cM^\op$ and action $m \cdot c := F(c) \otimes m$.
\begin{lemma}
Let $\cC$ be a finite tensor category, and $\cM$ an exact left $\cC$-module category.  For any tensor equivalence $\cF: \cC \rightarrow \cC^\mop$, the right $\cC$-module category $\cM^{\op(\cF)}$ is exact.
\end{lemma}
\begin{proof}
Recall from Theorem~\ref{Thm:ExactModCatOmnibus}, using Example~\ref{ex:exactness}, that projective objects in $\cC$ are injective and vice versa, and similarly for objects in $\cM$.  Consider an object $m \in \cM^{\op(\cF)}$ and a projective object $p \in \cC$.  Because $p \in \cC$ is projective, it follows that $\cF(p)$ is projective as an object of $\cC^\mop$, therefore injective as an object of $\cC^\mop$, therefore, by reversing the arrows, projective as an object of $\cC^\mp$.  By the exactness of $\cM$, it follows that $\cF(p) \otimes m$ is projective as an object of $\cM$, therefore injective as an object of $\cM$, and then, by reversing arrows, projective as an object of $\cM^{\op(\cF)}$, as required.
\end{proof}

\noindent Applying this lemma with the functor $\cF$ being an appropriate combination of duals (taking account of Lemma~\ref{lemma:DualTwist}) provides the desired result:

\begin{corollary} \label{cor:adjoint-exactness}
Let $\cC$ and $\cD$ be finite tensor categories and $\cM$ an exact $\cC$--$\cD$-bimodule category.  The dual bimodule categories $\cM^*$ and ${}^* \cM$ are exact $\cD$--$\cC$-bimodule categories.
\end{corollary}

\subsection{The dual bimodule category is the functor dual}

The dual of a finite-dimensional $k$-vector space $V$ can of course be characterized as a dual object in the category of vector spaces, but it can also be constructed explicitly as the linear dual vector space $\Hom_k(V,k)$.  For an ordinary bimodule ${}_A M_B$ (which is finitely generated and projective both as an $A$-module and as a $B$-module), there are left and right adjoints, which may be characterized as adjoint morphisms or may be constructed explicitly as the left and right linear duals $\Hom_A(M,A)$ and $\Hom_B(M,B)$.  The situation for bimodule categories is analogous, as we will see later on: for a bimodule category ${}_\cC \cM_\cD$, the left and right adjoint bimodule categories may be characterized abstractly or may be constructed explicitly as the linear functor dual categories $\Fun_\cC(\cM,\cC)$ and $\Fun_\cD(\cM,\cD)$, respectively.  

We now show that the left and right dual bimodule categories ${}^* \cM$ and $\cM^*$, defined above, are equivalent to the functor left and right dual categories $\Fun_\cC(\cM,\cC)$ and $\Fun_\cD(\cM,\cD)$, respectively; the dual bimodule categories will therefore provide alternative explicit realizations of adjoint bimodules.

We will first show that the dual of a functor bimodule category is itself a functor bimodule category.  For that we will need to temporarily break from our convention and consider functors that are not-necessarily right exact.  Given bimodule categories ${}_\cC \cM_\cD$ and ${}_{\cC} \cN_{\cE}$, let ${}_{\cD} \Fun^L_{\cC}(\cN,\cM)_{\cE}$ denote the $\cD$--$\cE$-bimodule category of left exact (not-necessarily right exact) $\overline{Vect}_k$-enriched $\cC$-module functors.  Bimodule categories of left exact right-module functors will be denoted similarly.

\begin{lemma}
Let ${}_\cC \cM_\cD$ and ${}_\cC \cN_\cE$ be finite bimodule categories between finite tensor categories.  Taking the right adjoint of a functor is an equivalence of $\cD$--$\cE$-bimodule categories from $\Fun_\cC(\cM,\cN)$ to ${}^* \Fun^L_\cC(\cN,\cM)$.  Similarly, if ${}_\cC \cM_\cD$ and ${}_\cE \cN_\cD$ are finite bimodule categories, then taking the right adjoint is an equivalence from $\Fun_\cD(\cM,\cN)$ to $\Fun^L_\cD(\cN,\cM)^*$.
\end{lemma}
\begin{proof}
Note that the right adjoint of a $\cC$-module functor from $\cM$ to $\cN$ is a linear functor from $\cN$ to $\cM$, which by Lemma~\ref{lma:module-adjoint} has a canonical $\cC$-module structure.  Taking right adjoints is certainly an equivalence of linear categories, with inverse taking left adjoints.  It therefore suffices to check that this process preserves the $\cD$--$\cE$-bimodule structure.  Observe that in any finite tensor category, we have the following two sequences of adjunctions of tensor product functors:
\begin{alignat*}{5}
	& x^* \otimes (-)  \quad & &  \dashv \quad &&  \quad x \otimes (-)  \quad &&  \dashv &&  \quad {}^*x \otimes (-), \\
	& (-) \otimes {}^* y  \quad & &  \dashv \quad &&  \quad (-) \otimes y  \quad &&  \dashv &&  \quad (-) \otimes y^*. 
\end{alignat*}  
The same adjunctions hold when the tensor represents the action on a module category.  For $\cF \in \Fun_{\cC}(\cM,\cN)$, the left $\cD$ action is by definition $(d \cdot \cF)(m) = \cF(m \otimes d)$.  The right adjoint of the composite functor $\cF(- \otimes d)$ is therefore $\cF^R(-) \otimes d^*$, as required.  The right $\cE$ action is similar, as is the case of right module functors.
\end{proof}

We can now prove the desired equivalence between the dual and functor categories.
\begin{proposition} \label{prop:dual-formula-for-adjoints}
Given a finite bimodule category ${}_\cC \cM_\cD$ between finite tensor categories, there are equivalences ${}^* \cM \simeq \Fun_\cC(\cM,\cC)$ and $\cM^* \simeq \Fun_\cD(\cM,\cD)$.  The first equivalence takes an object ${}^* m \in {}^* \cM$ to the left adjoint of the functor $c \mapsto c \otimes m$; the second equivalence takes an object $m^* \in \cM^*$ to the left adjoint of the functor $d \mapsto m \otimes d$.
\end{proposition}
\begin{proof}
By the previous lemma, there is an equivalence between $\Fun_\cC(\cM,\cC)$ and ${}^* \Fun^L_\cC(\cC,\cM)$.  The category of all (not-necessarily right or left exact) $\cC$-module functors from $\cC$ to $\cM$ is certainly equivalent to $\cM$.  However, because $\cC$ is finite tensor, it is exact over itself (as in Example~\ref{ex:exactness}), and therefore all $\cC$-module functors from it are exact, in particular left exact.  The other equivalence is obtained in the same fashion.
\end{proof}

\subsection{The relative Deligne tensor product as a functor category}

In constructing the duality between a (finite-dimensional) vector space $V$ and the linear dual $\Hom_k(V,k)$, it is convenient to know that that $\Hom_k(V,k) \otimes V \cong \Hom_k(V,V)$ (or more generally that $\Hom_k(V,k) \otimes W \cong \Hom_k(V,W)$): the coevaluation of the duality can then be expressed simply as $1 \mapsto \id_V$.  Analogously, in constructing the adjoint of an (appropriately finite) bimodule ${}_A M_B$, it is convenient to know that $\Hom_A(M,A) \otimes_A N \cong \Hom_A(M,N)$ and that $N \otimes_B \Hom_B(M,B) \cong \Hom_B(M,N)$.  Our next endeavor is to prove the analogous facts, namely $\Fun_\cC(\cM,\cC) \boxtimes_\cC \cN \simeq \Fun_\cC(\cM,\cN)$ and $\cN \boxtimes_\cD \Fun_\cD(\cM,\cD) \simeq \Fun_\cD(\cM,\cN)$, for bimodule categories.  When $\cC$ and $\cD$ are finite semisimple tensor categories and ${}_\cC \cM_\cD$ is a finite semisimple bimodule category, these facts are useful, but in the end merely a convenience, for constructing the adjoint of the bimodule category ${}_\cC \cM_\cD$.  However, when $\cC$ and $\cD$ are not semisimple, the reexpression of the relative tensor as a functor category is completely indispensable in constructing the necessary adjoints, and therefore for proving our dualizability results.  (The expression of a relative tensor product as a functor category allows the construction of functors into the relative tensor product, whose universal definition a priori only characterizes functors out.)

As discussed in Remark~\ref{rmk:Deligne_pdt_as_mod_functor} below, the following proposition and corollary are closely related to  \cite[Prop. 3.5 and Rmk. 3.6]{0909.3140} and \cite[Thm. 3.20]{0911.4979}.
\begin{proposition} \label{prop:FunctorsAsATensorPdt}
Let $\cC$, $\cD$, and $\cE$ be finite tensor categories and ${}_\cC \cM_\cD$ and ${}_\cC \cN_\cE$ finite bimodule categories.  The balanced functor $\Fun_\cC(\cM,\cC) \times \cN \ra \Fun_\cC(\cM,\cN)$, $(\cF,n) \mapsto \cF(-) \otimes n$, induces an equivalence, of $\cD$--$\cE$-bimodule categories,
\[
		\Fun_\cC(\cM, \cC) \boxtimes_\cC \cN \simeq \Fun_\cC(\cM,\cN).
\]
Here $\Fun_\cC$ denotes the category of (right exact) $\cC$-module functors.  If instead the bimodule categories are ${}_\cC \cM_\cD$ and ${}_\cE \cN_\cD$, then the analogous balanced functor induces an $\cE$--$\cC$-bimodule category equivalence
\[
\cN \boxtimes_\cD \Fun_{\cD}(\cM,\cD) \simeq \Fun_{\cD}(\cM,\cN).
\]
\end{proposition}

\begin{proof}
We focus on the first equivalence, as the second is analogous.  The functor induced by the balanced functor is certainly a bimodule functor, so it suffices to see that it is an equivalence of linear categories.  We prove that $\Fun_\cC(\cM, \cC) \boxtimes_\cC \cN$ and $\Fun_\cC(\cM,\cN)$ are both equivalent to a category of the form $\Mod{A}{B}(\cC)$, by equivalences intertwining the functor in question.

Using Theorem~\ref{thm:EGNO2.11.6}, choose algebra objects $A \in \cC$ and $B \in \cC$ and equivalences $\cM \simeq \Mod{}{A}(\cC)$ and $\cN \simeq \Mod{}{B}(\cC)$.  The functor $\tau: \Mod{A}{B}(\cC) \rightarrow \Fun_\cC(\cM,\cN)$, given by $\tau(M) = - \otimes_A M$, is an equivalence, as we will see presently and as observed in \cite[Prop. 7.11.1]{egno-book}.  A functor $\cF \in \Fun_\cC(\cM,\cN)$ is naturally equivalent to the functor $\tau(\cF(A)) = - \otimes_A \cF(A)$, so $\tau$ is essentially surjective.  (To see this note that $\cF$ and $\tau(\cF(A))$ agree on free $A$-modules, and any object $M \in \cM \simeq \Mod{}{A}(\cC)$ is the canonical coequalizer $\textrm{coeq}_\cC(M \otimes A \leftleftarrows M \otimes A \otimes A)$ of free $A$-modules.)  The functor $\tau$ is evidently fully faithful.  A special case of the equivalence $\tau$ gives an equivalence $\Mod{A}{}(\cC) \simeq \Fun_\cC(\cM,\cC)$.  The equivalence between $\Fun_\cC(\cM, \cC) \boxtimes_\cC \cN$ and $\Mod{A}{B}(\cC)$ now follows from Theorem~\ref{thm:DelignePrdtOverATCExists}.
\end{proof}

\begin{corollary} \label{cor:tensasfunct}
For ${}_\cD \cM_\cC$ and ${}_\cC \cN_\cE$ finite bimodule categories over finite tensor categories, there are equivalences of $\cD$--$\cE$-bimodule categories
\begin{align*}
\cM \boxtimes_\cC \cN &\simeq \Fun_{\mod{\cC}{}}(\cM^*,\cN), \\
\cM \boxtimes_\cC \cN &\simeq \Fun_{\mod{}{\cC}}({}^* \cN,\cM).
\end{align*}
\end{corollary}
\begin{proof}
By Propositions \ref{prop:dual-formula-for-adjoints} and \ref{prop:FunctorsAsATensorPdt}, we have
\[
\cM \boxtimes_\cC \cN \simeq {}^*(\cM^*) \boxtimes_\cC \cN \simeq \Fun_{\mod{\cC}{}}(\cM^*,\cC) \boxtimes_\cC \cN \simeq \Fun_{\mod{\cC}{}}(\cM^*,\cN).
\]
Similarly
\[
\cM \boxtimes_\cC \cN \simeq \cM \boxtimes_\cC ({}^* \cN)^* \simeq \cM \boxtimes_\cC \Fun_{\mod{}{\cC}}({}^* \cN,\cC) \simeq \Fun_{\mod{}{\cC}}({}^* \cN,\cM).\qedhere
\]
\end{proof}

\begin{remark} \label{rmk:Deligne_pdt_as_mod_functor}
This corollary is a correction of \cite[Remark 3.6]{0909.3140} and \cite[Thm. 3.20]{0911.4979}, both of which are off by a twist by a double dual functor.\footnote{Note also that, even appropriately corrected, the proof of \cite[Thm 3.20]{0911.4979} is incomplete.  For instance, it uses \cite[Lemma 3.21]{0911.4979}, which is false.  A counterexample to that Lemma is given by $\cC = \Vect[G]$, the category of $G$-graded vector spaces for a finite group $G$, and $\cM = \cN = \Vect$.  The Lemma would have $\Vect \boxtimes \Vect \simeq \Vect$ equivalent to $\Fun_{\Vect[G]}(\Vect,\Vect) \simeq \Rep(G)$, but no such equivalence exists.} 
\end{remark}

\subsection{Dual bimodule categories as modules over a double dual}

By Theorem \ref{thm:EGNO2.11.6}, any module category ${}_\cC \cM$ can be expressed as a category of modules $\Mod{}{A}(\cC)$ for an algebra object $A \in \cC$.  We now describe explicitly how the left and right duals of the module category ${}_\cC \cM$ can themselves be expressed as module categories.  The following lemma and its corollary are based on calculations in \cite[\S 3]{MR2097289}.

\begin{lemma}\label{lem:dualing-amod}
Let $A \in \cC$ and $B \in \cC$ be algebra objects in a finite tensor category $\cC$, and let $M \in \cC$ be an $A$--$B$-bimodule object.  The objects ${}^{**} A$ and $B^{**}$ are naturally algebra objects, the object $M^*$ is naturally a $B^{**}$--$A$-bimodule object, and the object ${}^* M$ is naturally a $B$--${}^{**} A$-bimodule object.  Moreover, the left and right duals provide equivalences of linear categories:
\begin{equation*}
\begin{array}{ccccc}
	\Mod{B^{**}}{A}(\cC) &\simeq& \Mod{A}{B}(\cC) &\simeq& \Mod{B}{{}^{**}A}(\cC) \\
	M^* &\leftmapsto& M &\rightmapsto& {}^* M
\end{array}
\end{equation*}
\end{lemma}

\begin{proof}
The left and right double dual functors are tensor equivalences, thus take algebra objects to algebra objects.  The right dual of the action map $\mu: A \otimes M \otimes B \rightarrow M$ is $\mu^*: M^* \rightarrow B^* \otimes M^* \otimes A^*$; taking the left adjoint of the functor $B^* \otimes - \otimes A^*$, we have a map $B^{**} \otimes M^* \otimes A \rightarrow M^*$, as desired.  The action on ${}^* M$ is similar, and it is clear that the duals provide equivalences as stated.
\end{proof}

\begin{corollary} \label{cor:dualamod}
Let $A \in \cC$ be an algebra object in a finite tensor category $\cC$.  There is an equivalence of right $\cC$-module categories
	\begin{align*}
		(\Mod{}{A}(\cC))^* & \simeq \Mod{A^{**}}{}(\cC) \\
		(L)^* & \mapsto L^*.
	\end{align*}
Here $(L)^*$ denotes the object $L$ viewed as an object of $(\Mod{}{A}(\cC))^*$ as in Definition~\ref{def:Dual_bimodule_notation}, whereas $L^*$ denotes the right dual of $L$ as an object of $\cC$.  There are analogous equivalences of $\cC$-module categories
\begin{align*}
(\Mod{A}{}(\cC))^* &\simeq \Mod{}{A}(\cC), \\
{}^*(\Mod{}{A}(\cC)) &\simeq \Mod{A}{}(\cC), \\
{}^*(\Mod{A}{}(\cC)) &\simeq \Mod{}{{}^{**}A}(\cC).
\end{align*}
\end{corollary}

\section{Separable module categories and separable tensor categories} \label{sec:tc-separable}

In the previous section, we laid the foundation for understanding when a bimodule category ${}_\cC \cM_\cD$ has adjoints.  However, because we are studying dualizability in a 3-category, we care not only about whether a bimodule category has adjoints, but also whether the units and counits of those adjunctions themselves have adjoints---in other words, we need two additional layers of dualizability above the bimodule category itself.  For inspiration, we can recall the situation governing the 2-dualizability of ordinary algebras: a finite-dimensional $k$-algebra $A$ is 2-dualizable if and only if it is projective as an $A \otimes A^\op$-module.  This projectivity condition is called `separability' and is equivalent to requiring that the multiplication map $\mu: A \otimes A \rightarrow A$ splits as an $A$--$A$-bimodule map.  

Recall that the bimodule category ${}_\cC \cM_\cD$ can be written (as a $\cC$-module) as $\Mod{}{A}(\cC)$ for an algebra object $A \in \cC$, or (as a $\cD$-module) as $\Mod{B}{}(\cD)$ for an algebra object $B \in \cD$.  These expressions of bimodule categories as categories of modules for algebra objects invites the following definitions.
\begin{definition}
An algebra object $A \in \cC$ in a semisimple finite tensor category $\cC$ is \emph{separable} if the multiplication $\mu: A \otimes A \ra A$ splits as an $A$--$A$-bimodule map, or equivalently if $A$ is projective as an $A$--$A$-bimodule.
\end{definition}
\noindent (That these two definitions of separability of an algebra object are equivalent is seen as follows: the projectivity condition certainly implies the splitting condition; conversely, because $\cC$ is semisimple, the unit $1 \in \cC$ is projective, which implies that $A \otimes A$ is projective as an $A$--$A$-bimodule, and using the splitting this implies that $A$ is projective as an $A$--$A$-bimodule.)
\begin{definition}
A finite left module category ${}_\cC \cM$ over a semisimple finite tensor category $\cC$ is \emph{separable} if it is equivalent as a $\cC$-module to the category of modules $\Mod{}{A}(\cC)$ over a separable algebra object $A \in \cC$.  Similarly, a finite right module category $\cM_\cD$ is separable if it is equivalent to $\Mod{B}{}(\cD)$ for a separable algebra object $B \in \cD$.  A finite bimodule category ${}_\cC \cM_\cD$ is separable if it is separable as a $\cC$-module and it is separable as a $\cD$-module.
\end{definition}
\noindent (Separability is a Morita-invariant notion, in the sense that if ${}_\cC \cM$ is equivalent to $\Mod{}{A}(\cC)$ and to $\Mod{}{B}(\cC)$, then $A$ is separable if and only if $B$ is separable.)  We will see later, in Section~\ref{sec:separableisfd}, that this separability condition is, as hoped, exactly what is needed to ensure a bimodule category is maximally dualizable.  

In this section, we prove two crucial properties of separable module categories, namely (1) that a module category is separable if and only if its category of endofunctors is semisimple and (2) that the relative Deligne tensor product of two separable bimodule categories is a separable bimodule category.  The second is a generalization to arbitrary fields of Etingof--Nikshych--Ostrik's theorem, over an algebraically closed field of characteristic zero, that a functor category between finite semisimple module categories is semisimple~\cite[Theorem 2.16]{MR2183279}.  Using a simple case of the second result, the first result provides a generalization to arbitrary perfect fields (and non-spherical categories) of M\"uger's theorem, over algebraically closed fields, that the Drinfeld center of a spherical finite semisimple tensor category, with simple unit and nonzero global dimension, is semisimple~\cite[Theorem 3.16]{MR1966525}.    These results are inspired by a suggestion of Ostrik and by results in \cite[\S 2.4]{MR3039775} in characteristic zero.  The reader unconcerned with finite characteristic can safely skip this and the next section, consulting Corollaries~\ref{cor:charzerosep} and~\ref{cor:charzeromodulesep} for characterizations of separability in the characteristic zero case, and keeping in mind Theorem~\ref{thm:compositeOfSep} and Corollary~\ref{cor:septc} at appropriate moments in Section~\ref{sec:separableisfd}.

\subsection{Separability and semisimplicity} \label{sec:sepandsemi}

Over a perfect field (for instance, a finite field, an algebraically closed field, or a field of characteristic zero), an algebra $A$ is separable if and only if it is finite-dimensional semisimple; over an arbitrary field, separability is a stronger condition than finite-dimensional semisimplicity \cite[Ch.~2]{MR0280479}.  The corresponding statements hold for linear categories: over a perfect field, a finite $\Vect$-module category is separable if and only if it is finite semisimple, while over an arbitrary field separability merely implies finite semisimplicity.  In general, a finite linear category $\cL$ is separable (as a $\Vect$-module) if and only if for all simple objects $S \in \cL$, the extension from the base field to the center of the division ring of endomorphisms of $S$ is a finite separable field extension.  In other words, a finite linear category is separable if and only if it remains semisimple after arbitrary base changes---we refer to this condition as `absolute semisimplicity' and for clarity use that term instead of `separable as a $\Vect$-module'.\footnote{For a semisimple $k$-linear category $\cL$ and an extension $l$ of the field $k$, let $[\cL_l]$ denote the category with the same objects as $\cL$ and with morphism spaces $\Hom_{[\cL_l]}(a,b) := \Hom_{\cL}(a,b) \otimes_k l$.  The base change from $k$ to $l$ of the category $\cL$ is, by definition, the idempotent completion $\cL_l$ of the category $[\cL_l]$.  When the category $\cL$ is finite and the field extension is finite, the base change $\cL_l$ of $\cL$ is equivalent to the product $\cL \boxtimes_{\Vect_k} \Vect_l$.}

For tensor categories $\cC$ other than $\Vect$, a separable $\cC$-module is still semisimple:
\begin{proposition} \label{prop:sepmodissemi}
A separable module category ${}_\cC \cM$, over any finite semisimple tensor category $\cC$, is semisimple.
\end{proposition}
\begin{proof}
It suffices to show that all objects of $\cM$ are projective.  Express ${}_\cC \cM$ as $\Mod{}{A}(\cC)$ for a separable algebra object $A$.  Because $\cC$ is semisimple, the unit $1 \in \cC$ is projective, which implies any free right $A$-module is projective in $\Mod{}{A}(\cC)$.  Any object $M \in \Mod{}{A}(\cC)$ is a summand of a free $A$-module, via $M \otimes_A A \xra{\id \otimes s} M \otimes_A (A \otimes A)$ where $s$ is the bimodule splitting, and therefore is projective as required.
\end{proof}
\noindent Thus, by Example~\ref{eg:semiexact}, separable module categories are exact.

We can moreover characterize the separability of $\cC$-module categories in terms of a semisimplicity condition, though on the category of endofunctors.
\begin{theorem} \label{thm:SepModCats}
Let $\cC$ be a finite semisimple tensor category and let ${}_\cC \cM$ be a finite $\cC$-module category.  The following conditions are equivalent:
\begin{enumerate}
\item the module category ${}_\cC \cM$ is separable,
\item the linear category $\Fun_{\cC}(\cM,\cM)$ is semisimple.  
\end{enumerate}
The same statements hold for right $\cC$-module categories.
\end{theorem}
\begin{proof}
Stitching together Theorem~\ref{thm:DelignePrdtOverATCExists}, Corollary~\ref{cor:tensasfunct}, and Corollary \ref{cor:dualamod}, we can write ${}_\cC \cM$ as $\Mod{}{A}(\cC)$ and $\Fun_{\cC}(\cM,\cM)$ as $\Mod{A}{A}(\cC)$ for an algebra object $A \in \cC$.  

If we assume $\Mod{A}{A}(\cC)$ is semisimple, then all its objects are projective; in particular $A$ is projective as a bimodule object, and so ${}_\cC \cM$ is separable.  Conversely, assume that $A$ is projective as an $A$--$A$-bimodule, in other words that the unit $1 \in \Fun_{\cC}(\cM,\cM)$ is projective.  Because $\Fun_{\cC}(\cM,\cM)$ is rigid (by item (5) of Theorem~\ref{Thm:ExactModCatOmnibus}), for any object $F \in \Fun_{\cC}(\cM,\cM)$, the functor $\Hom(F,-) \simeq \Hom(1,{}^* F \otimes -)$ is right exact and so $F$ is projective; thus the functor category is semisimple, as required.
\end{proof}

A semisimple module category over a semisimple tensor category need not be separable.  For example, in characteristic $p$, the trivial module category $\Vect$, over the tensor category $\Vect[\ZZ/p]$ of $\ZZ/p$-graded vector spaces, is not separable (because the functor category $\Fun_{\Vect[\ZZ/p]}(\Vect,\Vect) \simeq \Rep(\ZZ/p)$ is not semisimple).

\subsection{Separable bimodules compose}

Over an algebraically closed field of characteristic zero, finite semisimple bimodule categories over finite semisimple tensor categories have the excellent feature that they compose, that is the relative Deligne tensor of two such is again finite semisimple---this is a reformulation of~\cite[Theorem 2.16]{MR2183279}.  We generalize that result to an arbitrary field.\footnote{Kuperberg~\cite[\S 5]{MR1995781} observed that over an imperfect field, the Deligne tensor $\cM \boxtimes \cN$ of semisimple linear categories need not be semisimple, but that the product is semisimple provided the categories $\cM$ and $\cN$ are absolutely semisimple.}
\begin{theorem} \label{thm:compositeOfSep}
Let ${}_\cB \cM_\cC$ and ${}_\cC \cN_\cD$ be separable bimodule categories over finite semisimple tensor categories.  The relative Deligne tensor product ${}_{\cB} \cM \boxtimes_\cC \cN_\cD$ is a separable bimodule category.
\end{theorem}

\begin{proof}
We prove that the product is separable as a $\cB$-module; separability as a $\cD$-module is similar.  Choose separable algebra objects $A \in \cB$ and $C \in \cC$ and an algebra object $B \in \cC$, such that
\begin{itemize}
\item $\cM \simeq \Mod{}{A}(\cB)$ as left $\cB$-module categories
\item $\cM \simeq \Mod{B}{}(\cC)$ as right $\cC$-module categories
\item $\cN \simeq \Mod{}{C}(\cC)$ as left $\cC$-module categories
\end{itemize}
It follows that $\cM \boxtimes_{\cC} \cN \simeq \Mod{B}{C}(\cC)$.  We suppress all four of these equivalences in what follows, freely transporting objects from $\Mod{}{A}(\cB)$ to $\cM$ to $\Mod{B}{}(\cC)$, and such, without comment or notation.  For instance, ``$A$" refers both to the algebra object $A \in \cB$ and as usual to the corresponding module object $A_A \in \Mod{}{A}(\cB)$, and therefore also to an object $A \in \cM$ and to a module object ${}_B A \in \Mod{B}{}(\cC)$.  We will however sometimes indicate by the notation $[-]_\cE$ when we forget to an ambient tensor category $\cE$.  Thus, for instance, the expression $[A]_\cC$ would be the object of $\cC$ obtained by forgetting the left $B$-action on ${}_B A \in \Mod{B}{}(\cC)$.  Moreover, we will drop the subscript from $[-]$ when the ambient tensor category in question is clear from context.

To see that the product $\cM \boxtimes_\cC \cN \simeq \Mod{B}{C}(\cC)$ is separable, we need to identify it (as a left $\cB$-module) as a category of modules for a separable algebra object in $\cB$.  As a linear category, the product is certainly the category of algebras (that is, modules) for the monad $T := - \otimes C : \cM \ra \cM$.  Here, a priori, $T$ is a monad on $\cM$ considered as a linear category (that is, $T$ is an algebra object in $\Fun(\cM,\cM)$).  However because the left $\cB$-action commutes with the right $\cC$-action on $\cM$, in fact $T$ is a monad on $\cM$ considered as a left $\cB$-module (that is, $T$ is an algebra object in $\Fun_\cB(\cM,\cM)$); the product $\cM \boxtimes_\cC \cN$ is, as a left $\cB$-module, the category of algebras for the monad $T$ acting on $\cM$ as a left $\cB$-module.  

We can reformulate the category of algebras for this monad as a category of modules for an algebra object.  We give $[T(A)]_\cB$ the structure of an algebra objects in $\cB$, as follows.  The multiplication is the composite
\[
[T(A)] \otimes [T(A)] \cong [[T(A)] \otimes T(A)] \cong [T([T(A)] \otimes A)] \xra{[T(\alpha)]} [T(T(A))] \ra [T(A)]
\]
The first isomorphism exists because the forgetful functor from $\Mod{}{A}(\cB)$ to $\cB$ is left $\cB$-linear.  The second isomorphism comes from the left $\cB$-linearity of the monad $T$.  For any object $M \in \cM$, the object $[M] \in \cB$ is a right $A$-module; the action map $[M] \otimes A \ra [M]$ (in $\cB$) is a right $A$-module map, so there is a corresponding morphism $[M] \otimes A \ra M$ in $\cM$.  In particular, there is the morphism $\alpha: [T(A)] \otimes A \ra T(A)$ used in the third map above.  The fourth map is the composition of the monad.  The unit of the algebra object $[T(A)]$ is simply the composite $1 \ra A \ra [T(A)]$, where the second map is obtained, from the unit $A \ra T(A)$ of the monad, by forgetting to $\cB$.

We know that ${}_\cB \cM \boxtimes_\cC \cN$ is the category of $T$-algebras in ${}_\cB (\Mod{}{A}(\cB))$; we now check that that this category of $T$-algebras is precisely ${}_\cB (\Mod{}{[T(A)]}(\cB))$, the left $\cB$-module category of right $[T(A)]$-modules in $\cB$.  As we now work exclusively within $\cB$, we will dispense with the $[-]$ notation.  Observe that because the monad $T$ is right exact and left $\cB$-linear, there is, for any object $M \in \Mod{}{A}(\cB)$, an isomorphism $T(M) \cong M \otimes_A T(A)$.  Therefore, on the one hand, given a $T$-algebra in $\Mod{}{A}(\cB)$, we have a (right $A$-module) action map $M \otimes_A T(A) \ra M$, which determines a (right $A$-module) action map $M \otimes T(A) \ra M$ by precomposition along $M \otimes T(A) \ra M \otimes_A T(A)$.  On the other hand, the unit map $A \ra T(A)$ is a homomorphism of algebra objects in $\cB$, and it follows that every action map $M \otimes T(A) \ra M$ is $A$-balanced and so determines an action map $M \otimes_A T(A) \ra M$, that is, a $T$-algebra structure.

It remains only to check that $T(A) \in \cB$ is a separable algebra object.  Pick bimodule splittings $s: A \ra A \otimes A$ and $\sigma: C \ra C \otimes C$ of the multiplication maps of $A$ and $C$.  These combine to provide a bimodule splitting of the multiplication map on $T(A)$, as required:
\[
T(A) \xra{\id_A \otimes \sigma} T(T(A)) \cong T(T(A) \otimes_A A) \xra{T(\id_{T(A)} \otimes_A s)} T(T(A) \otimes A) \cong T(A) \otimes T(A). \qedhere
\]

\end{proof}

\begin{corollary} \label{cor:funsemi}
For absolutely semisimple, separable $\cC$-modules ${}_\cC \cM$ and ${}_\cC \cN$ over a finite semisimple tensor category $\cC$, the functor category $\Fun_{\cC}(\cM,\cN)$ is absolutely semisimple.
\end{corollary}

\begin{corollary} \label{cor:septc}
Assume the base field is perfect.  Finite semisimple tensor categories, separable bimodule categories, bimodule functors, and bimodule transformations, form a symmetric monoidal sub-3-category $\TCss$ of $\TC$.
\end{corollary}
\begin{proof}
Theorem~\ref{thm:compositeOfSep} ensures that the specified collection of 1-morphisms is closed under composition, and so we have a sub-3-category.  Because of the perfection assumption, every finite semisimple linear category is absolutely semisimple.  By the same theorem, the ordinary Deligne tensor of two such categories is absolutely semisimple, so in particular semisimple; the collection of objects is therefore closed under the monoidal structure.  To see that the morphisms are closed under the monoidal structure, observe that for $B \in \cC$ and $C \in \cD$ separable algebra objects, the product $B \boxtimes C \in \cC \boxtimes \cD$ is separable; the result therefore follows from part (3) of Theorem~\ref{thm:DelignePrdtOverATCExists}.
\end{proof}

As mentioned, we will see that separable bimodule categories ${}_\cC \cM_\cD$ are maximally dualizable.  What we really care about, though, is the dualizability of tensor categories.  Thus, what concerns us most is the separability of the bimodules ${}_{\cC \boxtimes \cC^\mp} \cC$ and $\cC_{\cC^\mp \boxtimes \cC}$ arising in the 1-dualizability of $\cC$.  This suggests the following definition.
\begin{definition} \label{def:sepcat}
A finite semisimple tensor category $\cC$ over a perfect field is \emph{separable} if it is separable as a $\cC \boxtimes \cC^\mp$--$\Vect$-bimodule.
\end{definition}

Recall that the Drinfeld center of a finite tensor category $\cC$ is by definition the linear category $\cZ(\cC) := \Fun_{\cC \boxtimes \cC^\mp}(\cC,\cC)$.  By Theorem~\ref{thm:compositeOfSep}, if $\cC$ is absolutely semisimple, then $\cC \boxtimes \cC^\mp$ is absolutely semisimple.  As a corollary of Theorem~\ref{thm:SepModCats}, we now have the generalization of M\"uger's theorem on the semisimplicity of Drinfeld centers.
\begin{corollary} \label{cor:Sep=semisimplecenter}
Let $\cC$ be a finite semisimple tensor category over a perfect field.  The following conditions are equivalent:
\begin{enumerate}
\item the tensor category $\cC$ is separable,
\item the Drinfeld center $\cZ(\cC)$ is semisimple.
\end{enumerate}
\end{corollary}

(A semisimple finite tensor category need not be separable.  Note that the representation category $\Rep(G)$ of a finite group appears as a tensor subcategory of the center $\cZ(\Vect[G])$.  In characteristic $p$, the category $\Rep(\ZZ/p)$ is not semisimple, so the center $\cZ(\Vect[\ZZ/p])$ is not semisimple, and thus $\Vect[\ZZ/p]$ is semisimple but not separable.)

Equipped with this characterization of separable tensor categories, we can also characterize separable module categories over separable tensor categories:

\begin{proposition} \label{prop:SSModuleCatsAreSep}
Let $\cC$ be a separable tensor category over any perfect field.  A finite module category ${}_\cC \cM$ is separable if and only if it is semisimple.
\end{proposition}
\begin{proof}
Proposition~\ref{prop:sepmodissemi} shows separability implies semisimplicity.  By Theorem~\ref{thm:SepModCats}, the module category $\cM$ is separable provided the linear category $\Fun_\cC(\cM,\cM)$ is semisimple.  By Example~\ref{eg:semiexact}, the module category ${}_\cC \cM$ is exact and so by part (5) of Theorem~\ref{Thm:ExactModCatOmnibus}, the monoidal category $\Fun_\cC(\cM,\cM)$ is rigid. Every tensor category $\cD$ is exact over its center, by~\cite[Lemma 7.12.7]{egno-book} applied to the exact module category ${}_{\cD \boxtimes \cD^\mp} \cD$; thus $\Fun_\cC(\cM,\cM)$ is an exact $\cZ(\Fun_\cC(\cM,\cM))$-module category.  By \cite{scha} and \cite[Cor. 3.35]{EO-ftc}, there is an equivalence $\cZ(\cC) \simeq \cZ(\Fun_\cC(\cM,\cM))$.  The center $\cZ(\cC)$ is semisimple by Corollary~\ref{cor:Sep=semisimplecenter}, and so $\Fun_\cC(\cM,\cM)$ is semisimple as required.
\end{proof}

The full sub-3-category of $\TCss$ on the separable tensor categories is symmetric monoidal and will be of special significance: in Section~\ref{sec:separableisfd} we will prove it is a fully dualizable 3-category. 
\begin{corollary} \label{cor:tcsepexists}
Over a perfect field, separable tensor categories, finite semisimple bimodule categories, bimodule functors, and bimodule transformations form a symmetric monoidal sub-3-category $\TCsep$ of $\TCss$.
\end{corollary}
\nid That the ordinary Deligne tensor of separable tensor categories is separable follows from the fact, established in Corollary~\ref{cor:septc}, that the tensor of separable module categories is separable.

\section{Separability and global dimension} \label{sec:tc-fusion}

In this section we give a computable criterion for checking the separability of a finite semisimple tensor category, at least when the base field is algebraically closed and the unit is simple.  Specifically, we show that such a tensor category is separable if and only if its global dimension (a sum of certain products of quantum traces on its simple objects) is nonzero.\footnote{When the tensor category is the representation category $\Rep(G)$ of a finite group, this result reduces to Maschke's classical observation that $\Rep(G)$ is semisimple precisely when the characteristic of the base field does not divide the order of the group $G$~\cite{maschke}.}  Over algebraically closed fields of characteristic zero, the global dimension is always nonzero, and so the characterization of separability simplifies substantially; in fact, all finite semisimple tensor categories over a field of characteristic zero are separable.

As we are investigating the separability condition on tensor categories, we need to assume the base field is perfect, but in fact in this section only we will assume the base field is algebraically closed.  The global dimension is constructed from endomorphisms of the unit and will be a numerical invariant when the unit is simple.  Recall that a finite semisimple tensor category (over an algebraically closed field) with simple unit is called a \emph{fusion category}.  Combining Corollary~\ref{cor:Sep=semisimplecenter} with the nonzero global dimension criterion for separability (Theorem~\ref{thm:NonzeroDimension} below) shows that a fusion category has semisimple Drinfeld center if and only if it has nonzero global dimension.  This result, which assumes nothing about the pivotality or sphericality of the category nor about the characteristic of the base field, sharpens results of M\"uger, Etingof--Nikshych--Ostrik, and Brugui\`eres--Virelizier.  M\"uger proved that a spherical fusion category of nonzero global dimension has semisimple center~\cite[Thm. 3.16]{MR1966525}; ENO proved that, in characteristic zero, fusion categories of nonzero global dimension have semisimple center~\cite[Thm. 2.15]{MR2183279}---see also Section 9 of that paper for a discussion of the positive characteristic situation; BV proved that pivotal fusion categories with semisimple center have nonzero global dimension~\cite{MR3079759}.  Victor Ostrik suggested many of the ideas contained in this section.

\subsection{Global dimension via quantum trace} \label{sec:gdim}

Fusion categories abstract the structure present in the representation category of a finite group.  Provided the base field is algebraically closed, the order of a finite group can be computed from its representation category as the sum of the squares of the dimensions of the irreducible representations.\footnote{This connection between the size of the group and its representation theory certainly depends on the algebraically closed assumption: the group $\ZZ/3$ has two irreducible real representations, one of dimension 1 and one of dimension 2.}  The global dimension of a fusion category, introduced by M\"uger~\cite{MR1966524}, is similarly defined as a sum of `squared norms' of simple objects of the category.  

The notion of squared norm of simple objects of a fusion category is defined using the notion of quantum trace:
\begin{definition}
Let $\cC$ be a fusion category over $k$, with unit $1 \in \cC$.  Let $x \in \cC$ be an object and choose morphisms $\ev_x: x \otimes {}^* x \ra 1$ and $\coev_x : 1 \ra {}^* x \otimes x$ witnessing ${}^* x$ as a left dual of $x$; also choose morphisms $\ev_{({}^* x)} : {}^* x \otimes {}^{**} x \ra 1$ and $\coev_{({}^* x)}: 1 \ra {}^{**} x \otimes {}^* x$ witnessing ${}^{**} x$ as a left dual of ${}^* x$.  The \emph{quantum trace} of a morphism $a: {}^{**} x \ra x$ is
\[
\Tr(a) := \ev_x \circ (a \otimes \id_{({}^* x)}) \circ \coev_{({}^* x)} \in \End_{\cC}(1) \cong k.
\]
\end{definition}
\nid The quantum trace of the morphism $a$ depends not only on that morphism but also on the choices of the evaluation map $\ev_x$ and coevaluation map $\coev_{({}^* x)}$.  Indeed we can change either that evaluation or coevaluation independently by any nonzero scalar, and so altogether the quantum trace appears to do nothing more than detect whether the morphism $a$ is zero or not.  

However, when the object $x$ is simple, we can eliminate the dependency on the evaluation and coevaluation, and even on the morphism $a$, by judiciously combining two distinct quantum traces.  Note that the simple object $x$ is non-canonically isomorphic to its double dual ${}^{**} x$.  (The isomorphism class of a simple object $y \in \cC$ is determined by the condition $\dim \Hom(x \otimes y,1) = 1$, and both ${}^* x$ and $x^*$ satisfy that condition; this uses the fact that $\dim \Hom(x \otimes x^*,1) = \dim \Hom(1, x \otimes x^*)$.)  We can therefore take the quantum trace of an isomorphism $a: {}^{**} x \ra x$ and combine it with the quantum trace of a dual of the inverse isomorphism:
\begin{definition}
Pick evaluation and coevaluation maps witnessing ${}^* x$ as a left dual of $x$, and maps witnessing ${}^{**} x$ as a left dual of ${}^* x$, and maps witnessing ${}^{***} x$ as a left dual of ${}^{**} x$.  The \emph{squared norm} of a simple object $x \in \cC$ of a fusion category is
\[
\lVert x \rVert := \Tr(a) \cdot \Tr({}^*(a^{-1})) \in k,
\]
where $a: {}^{**} x \ra x$ is any choice of isomorphism, and the dual ${}^*(a^{-1})$ of $a^{-1}$ and both the quantum traces are computed using the chosen evaluation and coevaluation maps.
\end{definition}
\nid The terminology is somewhat unfortunate, as in general the squared norm is not canonically the square of anything.  Note that the squared norm is independent of the choices of evaluation and coevaluation maps and of the choice of isomorphism $a$: each such choice is fixed up to a scalar, and each such scalar appears twice in the squared norm expression with opposite multiplicative signs.  The squared norm is invariant under isomorphism and the squared norm of an object is isomorphic to the squared norm of its dual: $\lVert x \rVert = \lVert {}^* x \rVert$.
\begin{definition}
The \emph{global dimension} of a fusion category $\cC$ is
\[
\dim(\cC) := \sum_{x} \lVert x \rVert,
\]
where the sum ranges over a choice of representatives of the isomorphism classes of simple objects of $\cC$.
\end{definition}

We can characterize the squared norm of a simple object $x$ more abstractly in terms of the properties of evaluation and coevaluation maps for dualities among simple objects, as follows.  Let $x$ and $y$ be simple objects of $\cC$ such that $\dim \Hom(x \otimes y,1)=1$.  For any nonzero map $v: x \otimes y \ra 1$, there is a unique map $\gamma(v): 1 \ra y \otimes x$ such that $(v,\gamma(v))$ is a coevaluation and evaluation pair witnessing $y$ as a left dual of $x$.  Similarly, for any nonzero map $u: 1 \ra y \otimes x$, there is a unique map $\gamma(u): x \otimes y \ra 1$ such that $(\gamma(u),u)$ is a coevaluation and evaluation pair witnessing $y$ as a left dual of $x$.  In other words, we denote by $\gamma$ both the association to a coevaluation of the corresponding evaluation, and the association to an evaluation of the corresponding coevaluation.  Note that this association $\gamma$ is homogeneous of degree $-1$, that is $\gamma(\lambda f) = \lambda^{-1} \gamma(f)$ for $\lambda \in k^\times$.  

This pairing $\gamma$ between evaluations and coevaluations commutes with taking the dual of a map: for any nonzero map $f: x \otimes y \ra 1$, we have
\[
{}^*(\gamma(f)) = \gamma({}^* f).
\]
However the pairing $\gamma$ does not intertwine the left and right inverse operations.  Given a map $f$ with a right inverse, we will denote the right inverse by $f^{-}$, and given a map $g$ with a left inverse, we will denote the left inverse by ${}^- \!g$.  Indeed, the failure of commutation between the evaluation--coevaluation pairing and the left/right inverse is exactly the squared norm:
\[
\gamma(f^{-}) = \lVert x \rVert \;{}^-\!(\gamma(f)).
\]
To see this is suffices to check that $\gamma(f^{-}) \circ \gamma(f) = \lVert x \rVert$.  For this, note that $\gamma(f^{-}) \circ \gamma(f) = (f \circ f^{-}) \cdot (\gamma(f^{-}) \circ \gamma(f))$ and any expression at all of the form $(f \circ g) \cdot (\gamma(g) \circ \gamma(f))$, for $f: x \otimes y \ra 1$ and $g: 1 \ra x \otimes y$ nonzero, gives the squared norm.

\subsection{The algebra of enriched endomorphisms of the unit} \label{sec:enrichedendo}

Recall that we aim to characterize the separability of a fusion category in terms of its global dimension.  A fusion category $\cC$ is separable precisely when it is equivalent as a left $\cC \boxtimes \cC^\mp$-module category to the category of modules $\Mod{}{A}(\cC \boxtimes \cC^\mp)$ for a separable algebra object $A \in \cC \boxtimes \cC^\mp$.  In this subsection, following \cite{MR2097289,MR1966524}, we describe explicitly an algebra object $A \in \cC \boxtimes \cC^\mp$ realizing any fusion category $\cC$ as a category of modules.  The construction is based on Ostrik's notion of enriched $\Hom$ for module categories over finite tensor categories: 
\begin{definition}
Let $\cC$ be a finite tensor category and let $\cM$ be a finite left $\cC$-module category.  For objects $m \in \cM$ and $n \in \cM$, the \emph{enriched Hom object} $\IHom_\cC(m,n) \in \cC$ is the object representing the functor $\Hom_\cM(- \otimes m,n): \cC \ra \Vect$.
\end{definition}
\nid See \cite{MR1976459, EO-ftc, egno-book, BTP} for further discussion about enriched endomorphisms for module categories, and in particular for a proof that the enrichment indeed exists.

We specialize to the situation of interest, where the tensor category is the product $\cC \boxtimes \cC^\mp$ and the left module category is the finite tensor category $\cC$ itself.  We will denote by $A := \IHom_{\cC \boxtimes \cC^\mp}(1_\cC,1_\cC) \in \cC \boxtimes \cC^\mp$ the algebra of enriched endomorphism of the unit $1_\cC$ of $\cC$.  Note that by the definition of the enriched Hom object, we have, for any $x \in \cC$ and $y \in \cC^\mp$, natural isomorphisms
\[
\Hom_{\cC \boxtimes \cC^\mp}(x \boxtimes y, \IHom_{\cC \boxtimes \cC^\mp}(a,b)) \cong \Hom_\cC(x \otimes a \otimes y,b).
\]
The object $A \in \cC \boxtimes \cC^\mp$ has a canonical algebra structure and there is an equivalence of left $\cC \boxtimes \cC^\mp$-module categories
\[
\IHom_{\cC \boxtimes \cC^\mp}(1_\cC,-) : \cC \ra \Mod{}{A}(\cC \boxtimes \cC^\mp);
\]
see~\cite[\S 2]{MR2097289} for a more detailed treatment of the algebra $A$ and its properties.

When the category $\cC$ is fusion, we can describe the algebra object $A \in \cC \boxtimes \cC^\mp$ more explicitly in terms of simple objects of $\cC$.  Let $\{L_i\}$ denoted a fixed collection of representatives of the isomorphism classes of simple objects of $\cC$, with the representative $L_1$ of the isomorphism class of the unit chosen to actually be the unit object $1_\cC$.  Also fix choices of dual objects $\{{}^* L_i\}$ together with evaluation and coevaluation maps witnessing each ${}^* L_i$ as a left dual of $L_i$.  Observe that, for any objects $x \in \cC$ and $y \in \cC^\mp$, we have isomorphisms
\[
\begin{split}
\Hom_{\cC \boxtimes \cC^\mp}( x \boxtimes y , \oplus_i L_i \boxtimes {}^* L_i)
\cong
\oplus_i \Hom_{\cC}(x, L_i) \otimes \Hom_{\cC}(y, {}^*L_i) \hspace*{.6in}\\\hspace*{.6in}
\cong
\oplus_i \Hom_{\cC}(x, L_i) \otimes \Hom_{\cC}(L_i, y^*)
\cong
\Hom_{\cC}(x, y^*)
\cong
\Hom_{\cC}(x \otimes y, 1).
\end{split}
\]
That is, the sum $\oplus_i \, L_i \boxtimes {}^* L_i$ satisfies the defining representing property of $A := \IHom_{\cC \boxtimes \cC^\mp}(1_\cC,1_\cC)$ and so
\[
A \cong \oplus_i \, L_i \boxtimes {}^* L_i;
\]
this isomorphism is canonically determined by the previous choice of witnesses for the dual objects $\{{}^* L_i\}$.

We can describe the algebra structure explicitly in terms of this decomposition.  The unit map $u: 1 \boxtimes 1 \ra A$ is simply the inclusion of the factor $1 \boxtimes 1 = 1 \boxtimes {}^*1$.  Fix a basis $\{e^r_{ij,k}\}_{r \in I(i,j,k)}$ for each of the vector spaces $\Hom_\cC(L_i \otimes L_j,L_k)$; here $I(i,j,k)$ is a finite index set.  Let $\{\hat{e}^r_{ij,k}\}$ denote the basis for $\Hom_\cC(L_k,L_i \otimes L_j)$ defined by $e^r_{ij,k} \circ \hat{e}^s_{ij,k} = \delta_{rs} \, \id_{L_k}$.  There is an associated collection of left dual maps $\{{}^*\hat{e}^r_{ij,k}\}$ forming a basis for $\Hom_\cC({}^* L_j \otimes {}^* L_i, {}^* L_k)$.  The multiplication map
\[
m: A \otimes A \cong \oplus_{i,j} \, (L_i \otimes L_j) \boxtimes ({}^* L_j \otimes {}^* L_i) \ra \oplus_k \, L_k \boxtimes {}^* L_k \cong A
\]
is given by the sum $\sum_r (e^r_{ij,k}) \boxtimes ({}^*\hat{e}^r_{ij,k})$; note that here the tensor product ${}^* L_j \otimes {}^* L_i$ has been taken in $\cC$, not in $\cC^\mp$.

\subsection{Fusion categories are modules over a Frobenius algebra}

In the last subsection, we saw that for any finite tensor category $\cC$, the module ${}_{\cC \boxtimes \cC^\mp} \cC$ can be presented as the category of modules $\Mod{}{A}(\cC \boxtimes \cC^\mp)$, where the algebra object $A$ is the enriched endomorphism object $\IHom_{\cC \boxtimes \cC^\mp}(1_\cC,1_\cC)$.  When $\cC$ is fusion, the algebra object $A$ admits canonically the structure of a Frobenius algebra object---M\"uger proved this for spherical fusion categories of nonzero global dimension~\cite{MR1966525}.  Recall that an object $R$ of a monoidal category is a Frobenius algebra object given maps
\begin{equation*}
	\begin{array}{l c  r c c c}
		\textrm{(unit)} & \quad & u: & 1 & \to & R \\
		\textrm{(multiplication)} && m: & R \otimes R & \to & R \\
		\textrm{(counit)} && \lambda: & R &\to &1 \\
		\textrm{(comultiplication)} && \Delta: & R &\to& R \otimes R,
	\end{array}
\end{equation*}
such that the pair $(u,m)$ gives $R$ the structure of an algebra object, the pair $(\lambda, \Delta)$ gives $R$ the structure of a coalgebra object, and the comultiplication $\Delta$ is an $R$--$R$-bimodule map.  From these structure maps we can form the composites
\begin{equation*}
	\begin{array}{l c  r c c c}
		\textrm{(pairing)} && b := \lambda \circ m: & R \otimes R& \to& 1 \\
		\textrm{(copairing)} && c := \Delta \circ u: & 1&\to&  R \otimes R,
	\end{array}
\end{equation*}
which satisfy the usual zig-zag equations.  Note that the comultiplication $\Delta$ is uniquely determined by the algebra structure $(u,m)$ and the counit $\lambda$.

For the algebra object $A \cong \oplus_i \, L_i \boxtimes {}^* L_i \in \cC \boxtimes \cC^\mp$, there is an obvious candidate for a counit map $\lambda$, namely the projection onto the factor $1 \boxtimes 1$.  Indeed, up to a scalar, this is the only possibility for the counit, and we can characterize this choice of counit as the one satisfying the condition $\lambda \circ u = \id_{1 \boxtimes 1}$; we refer to such a counit as \emph{normalized}.  This counit does indeed provide a Frobenius algebra structure on $A$:
\begin{proposition}
Let $\cC$ be a fusion category and let $A = \IHom_{\cC \boxtimes \cC^\mp}(1_\cC,1_\cC) \in \cC \boxtimes \cC^\mp$ be the algebra object of enriched endomorphisms of the unit.  The normalized counit map $\lambda: A \ra 1 \boxtimes 1$, with $\lambda \circ u = \id_{1 \boxtimes 1}$, determines the structure of a Frobenius algebra on $A$.
\end{proposition}
\begin{proof}
It suffices to define a copairing $c : 1 \boxtimes 1 \ra A \otimes A$ satisfying the zig-zag relations with the pairing map $b = \lambda \circ m : A \otimes A \ra 1 \boxtimes 1$; there is at most one such copairing.  Recall that $\{e^r_{ij,k}\}$ is a chosen basis for $\Hom_\cC(L_i \otimes L_j,L_k)$ and $\{\hat{e}^r_{ij,k}\}$ is the `dual basis' for $\Hom_\cC(L_k,L_i \otimes L_j)$.  For each index $i$, we denote by $\overline{i}$ the unique index such that the morphism space $\Hom_\cC(L_i \otimes L_{\overline{i}},1)$ is nonzero.  We will use the abbreviation $e_i := e^1_{i\overline{i},1} : L_i \otimes L_{\overline{i}} \ra 1$.  Note that the right inverse of $e_i$ is $e_i^{-} = \hat{e}^1_{i\overline{i},1}$.

As mentioned, the normalized counit map $\lambda: A \ra 1 \boxtimes 1$ is the projection $\oplus_i \, L_i \boxtimes {}^* L_i \ra 1 \boxtimes 1$.  Thus the pairing $\lambda \circ m: A \otimes A \ra 1 \boxtimes 1$ is the map $\oplus_{i,j} \, (L_i \otimes L_j) \boxtimes ({}^* L_j \otimes {}^* L_i) \ra 1 \boxtimes 1$ that first projects onto the summands with $\{i,j\}$ of the form $\{i,\overline{i}\}$ and then applies the map $e_i \boxtimes {}^* (e_i^{-})$ on the $\{i,\overline{i}\}$ factor.  Define the copairing $c : 1 \boxtimes 1 \ra A \otimes A$ to be the composite
\[
1 \boxtimes 1 \ra \oplus_i \, (L_{\overline{i}} \otimes L_i) \boxtimes ({}^* L_i \otimes {}^* L_{\overline{i}}) \hookrightarrow \oplus_{i,j}  \, (L_i \otimes L_j) \boxtimes ({}^* L_j \otimes {}^* L_i) \cong A \otimes A
\]
where the first map is given by $\gamma(e_{\overline{i}}) \boxtimes \gamma({}^*(e_{\overline{i}}^{-}))$ on the $i$-th component.  By the definition of the association $\gamma$, this copairing satisfies the necessarily zig-zag relations with the pairing.
\end{proof}
\nid When referring to a Frobenius algebra structure on the enriched endomorphism object $A$, we will always mean the normalized Frobenius structure given by this proposition.

\subsection{The window element of the representing Frobenius algebra}

In any Frobenius algebra object $R$, there is a distinguished map $w: 1 \ra R$, called the \emph{window element}, with the property that the multiplication of the comultiplication $m \circ \Delta: R \ra R$ is given by (left or right) multiplication by $w$; the window element is always given by the composite $w = m \circ \Delta \circ u: 1 \ra R$.  For the algebra object $A \cong \oplus_i L_i \boxtimes {}^* L_i$ of a fusion category, we have that $\Hom_{\cC \boxtimes \cC^\mp}(1 \boxtimes 1, A)$ is 1-dimensional, and so the window element is a scalar multiple of the unit map; by slight abuse of notation we will also denote this scalar by $w \in k$.  Because of the normalization condition $\lambda \circ u = \id_{1 \boxtimes 1}$ on our Frobenius algebra structure on $A$, note that the scalar window element is given by the composite $\lambda \circ m \circ \Delta \circ u \in \End(1 \boxtimes 1) = k$.
\begin{proposition}
Let $A \in \cC \boxtimes \cC^\mp$ be the normalized Frobenius algebra object of enriched endomorphisms of the unit of a fusion category $\cC$.  The window element of $A$ is the global dimension of $\cC$.
\end{proposition}
\begin{proof}
In the proof of the previous proposition, we saw that the copairing was given by the inclusion of $\oplus_i \, \gamma(e_{\overline{i}}) \boxtimes \gamma({}^*(e_{\overline{i}}^{-}))$ into $\oplus_{i,j} \, (L_i \otimes L_j) \boxtimes ({}^* L_j \otimes {}^* L_i)$, and the pairing was giving by projection onto the $\{i,\overline{i}\}$ factors followed by the map $\oplus_i \, e_i \boxtimes {}^*(e_i^{-})$.  The window element is therefore
\begin{align*}
w &= \sum_i (e_i \circ \gamma(e_{\overline{i}})) \cdot ({}^*(e_i^{-}) \circ \gamma({}^*(e_{\overline{i}}^{-}))) \\
&= \sum_i \lVert L_i \rVert (e_i \circ \gamma(e_{\overline{i}})) \cdot ({}^*(e_i^{-}) \circ {}^* ({}^-\!(\gamma(e_{\overline{i}}))))
= \sum_i \lVert L_i \rVert = \dim \cC.
\end{align*}
Here the second equality follows from the commutation relations mentioned at the end of Section~\ref{sec:gdim}.  For the third equality, note that $(e_i \circ \gamma(e_{\overline{i}})) \cdot ({}^*(e_i^{-}) \circ {}^* ({}^-\!(\gamma(e_{\overline{i}}))))$ is homogeneous of degree zero in the term $\gamma(e_{\overline{i}}): 1 \ra L_i \otimes L_{\overline{i}}$, which term may therefore be replaced by the term $e_i^{-}: 1 \ra L_i \otimes L_{\overline{i}}$.
\end{proof}

We are now in a position to characterize the separability of a fusion category in terms of its global dimension.
\begin{theorem} \label{thm:NonzeroDimension}
A fusion category is separable if and only if its global dimension is nonzero.
\end{theorem}
\begin{proof}
Suppose the global dimension of the fusion category $\cC$ is nonzero.  The comultiplication (of the Frobenius algebra $A \in \cC \boxtimes \cC^\mp$ of enriched endomorphisms of the unit $1_\cC$) is a bimodule map $\Delta: {}_A A_A \ra {}_A A \otimes A_A$, and the composite $m \circ \Delta$ is multiplication by the scalar $\dim \cC$.  Thus $\Delta / \dim \cC$ is a bimodule splitting of the multiplication map and so the algebra $A$ is separable; hence the fusion category $\cC$ is separable.

Now assume the fusion category $\cC$ is separable.  By Corollary~\ref{cor:Sep=semisimplecenter}, the center $\cZ(\cC) = \Fun_{\cC \boxtimes \cC^\mp}(\cC,\cC)$ is semisimple.  Note that this center is equivalent to the category $\Mod{A}{A}(\cC \boxtimes \cC^\mp)$ of $A$--$A$-bimodules in $\cC \boxtimes \cC^\mp$.  Observe that for any object $M \in \cC \boxtimes \cC^\mp$, there is a canonical isomorphism $\Hom_{\Mod{A}{A}}(A \otimes A,M) \cong \Hom_{\cC \boxtimes \cC^\mp}(1 \boxtimes 1,M)$; that is, the bimodule ${}_A A \otimes A_A$ is the free $A$--$A$-bimodule generated by $1 \boxtimes 1$.  In particular it follows that $\Hom_{\Mod{A}{A}}(A \otimes A, A) \cong k$ and, by the semisimplicity of the category of $A$--$A$-bimodules, therefore $\Hom_{\Mod{A}{A}}(A, A \otimes A) \cong k$.  The comultiplication $\Delta \in \Hom_{\Mod{A}{A}}(A, A \otimes A)$ of the Frobenius algebra $A$ is certainly nonzero.  By the separability of $A$, there exists some element of $\Hom_{\Mod{A}{A}}(A, A \otimes A)$ that splits the multiplication $m: A \otimes A \ra A$.  It follows that up to a nonzero scalar, the comultiplication $\Delta$ splits the multiplication, and therefore the global dimension of $\cC$ is nonzero, as required.
\end{proof}

Etingof, Nikshych, and Ostrik proved that every fusion category over an algebraically closed field of characteristic zero has nonzero global dimension~\cite[Thm. 2.3]{MR2183279}.  Thus over such a field, every fusion category is separable.  In fact the assumption of simplicity of the unit is immaterial: every finite semisimple tensor category over an algebraically closed field of characteristic zero is separable---this can be seen by combining~\cite[Thm 2.18]{MR2183279} with Corollary~\ref{cor:Sep=semisimplecenter}.  Even better, by noting that taking the Drinfeld center commutes with base change, by~\cite[Lemma 5.1]{1002.0168}, and that over a perfect field semisimplicity is preserved by descent, we can drop the assumption of algebraic closure:
\begin{corollary} \label{cor:charzerosep}
Every finite semisimple tensor category over an arbitrary field of characteristic zero is separable.
\end{corollary}

\nid Thus in characteristic zero, Proposition~\ref{prop:SSModuleCatsAreSep} reduces to the following characterization:
\begin{corollary} \label{cor:charzeromodulesep}
Let $\cC$ be a finite semisimple tensor category over an arbitrary field of characteristic zero.  A finite module category ${}_\cC \cM$ is separable if and only if it is semisimple.
\end{corollary}

\begin{remark} \label{rem:nonsimp}
We expect one can generalize Theorem~\ref{thm:NonzeroDimension} to all perfect fields and all semisimple tensor categories, as follows.  Let $\cC$ be a semisimple tensor category over a perfect field, and let $\overline{\cC}$ denote its base extension to the algebraic closure.  Semisimplicity of the center is invariant under base change, so $\cC$ is separable if and only if $\overline{\cC}$ is separable.  The unit of $\overline{\cC}$ splits as a direct sum $\oplus \, 1_i$ of simple objects, and this provides a decomposition of the linear category $\overline{\cC}$ as a sum $\oplus \,\overline{\cC}_{ij}$, where $\overline{\cC}_{ij} := 1_i \otimes\overline{\cC}\otimes 1_j$, as in~\cite[\S 2.4]{MR2183279}.  When $\overline{\cC}$ is indecomposable, the $\overline{\cC}$--$\overline{\cC}_{ii}$-bimodule $\oplus_j \,\overline{\cC}_{ji}$ provides a Morita equivalence between $\overline{\cC}$ and $\overline{\cC}_{ii}$; more generally, $\overline{\cC}$ is Morita equivalent to a sum $\oplus_I \,\overline{\cC}_{ii}$, where the set $I$ contains one index from each indecomposable component of $\overline{\cC}$.  Thus the center of $\overline{\cC}$ is the sum $\oplus_I \,\cZ(\overline{\cC}_{ii})$.  Altogether, a semisimple tensor category $\cC$ is separable if and only if all the fusion categories $\overline{\cC}_{ii}$ have nonzero global dimension.
\end{remark}

\chapter{Dualizability} \label{sec:dualizability}

Every algebra is 1-dualizable, with dual the opposite algebra, and similarly every tensor category is 1-dualizable, with dual the monoidally-opposite tensor category---we show this in Section~\ref{sec:df-objects} and then describe the resulting twice categorified 1-dimensional field theory associated to a tensor category.  By contrast, 2-dualizability is a substantive condition on a tensor category.  In Section~\ref{sec:df-morphisms} we prove that the functor dual of a finite bimodule category between finite tensor categories provides an adjoint bimodule category, and thereby we see that every finite tensor category is 2-dualizable.  There is therefore a categorified 2-dimensional field theory associated to any finite tensor category.  We show that the value of the Serre bordism under this field theory is the bimodule associated to the double dual automorphism of the tensor category---this calculation provides the fundamental connection between the topology of low-dimensional framed manifolds and the algebra of duality in tensor categories.  

It is certainly not the case that every finite tensor category is 3-dualizable, but nevertheless the 2-dimensional field theory associated to a finite tensor category does always extend to take values on certain 3-manifolds and therefore provides a kind of `non-compact' 3-dimensional field theory.  We can indicate which 3-manifolds are allowed as follows.  Recall that the Radford bordism is a 3-framed 2-dimensional bordism from the inverse Serre bordism to the Serre bordism.  The Radford bordism is an equivalence: there is an `inverse Radford bordism' and 3-framed 3-manifolds---let's call them Radford witnesses---witnessing that the Radford bordism and the inverse Radford bordism are indeed inverse.  The field theory associated to a finite tensor category can take values on 3-manifolds built using the 3-framed 3-handles arising in the Radford witnesses.  We prove all this, formalized in the statement that finite tensor categories are Radford objects of the 3-category of tensor categories, in Section~\ref{sec:radfordftc}.  This provides an immediate, transparent topological proof of the quadruple dual theorem for finite tensor categories: the associated field theory takes the inverse Serre bordism to the left double dual bimodule and takes the Serre bordism to the right double dual bimodule, and the image of the Radford bordism therefore provides an equivalence between those bimodules.

In Section~\ref{sec:separableisfd} we finally investigate the existence of adjoints for functors between bimodule categories.  We prove our main result, that separable tensor categories are fully dualizable.  This provides a 3-framed 3-dimensional topological field theory associated to any separable tensor category, thus in particular to any fusion category of nonzero global dimension over an algebraically closed field.  We also prove a converse result, showing that fully dualizable finite tensor categories must be separable, and in fact we identify the maximal fully dualizable 3-category of finite tensor categories---this supplies a classification of local framed field theories with defects.

We conclude in Section~\ref{sec:spherical} by describing the connection between pivotal and spherical structures and homotopy fixed point structures on finite tensor categories.  We observe that a pivotal structure on a finite tensor category (that is a monoidal trivialization of the double dual functor) provides a trivialization of the Serre automorphism of the tensor category---such a trivialization is exactly the data of an $\Omega S^2$ homotopy fixed point structure on the tensor category.  A pivotal structure trivializes the double dual, so the square of a pivotal structure trivializes the quadruple dual; but we already have a trivialization of the quadruple dual, namely the one provided by the Radford equivalence.  We say that a pivotal finite tensor category is spherical if the square of the pivotal structure is the Radford equivalence.  (This notion agrees with the classical notion of sphericality when the tensor category is semisimple, but provides the correct generalization to the non-semisimple case.)  A spherical tensor category has canonically the structure of an $\Omega \Sigma \RP^2$ homotopy fixed point.  We close with various questions and conjectures concerning fixed points and descent properties of the associated field theories.  For instance, we conjecture that every pivotal fusion category admits not only the structure of an $\Omega S^2$ fixed point, but the structure of an $SO(2)$ fixed point, and therefore provides a combed local field theory.  We also conjecture that every spherical fusion category admits not only the structure of an $\Omega \Sigma \RP^2$ fixed point but in fact the structure of an $SO(3)$ fixed point, and therefore provides an oriented local field theory.

Throughout this section, we assume the base field is perfect, and we work entirely within the symmetric monoidal 3-category $\TC$ of finite tensor categories, finite bimodule categories, right exact bimodule functors, and bimodule transformations.

\section{Duals of tensor categories and invariants of 1-framed bordisms} \label{sec:df-objects}

In this section, we show that $\cC^\mp$ is a dual to $\cC$, for any tensor category $\cC$.  This result follows readily from the basic structure of the relative Deligne tensor product, and does not depend on any substantive results about finite tensor categories---indeed, if one defined a relative Deligne tensor product in the context of bimodules between not-necessarily finite tensor categories, then even a non-finite tensor category $\cC$ would have $\cC^\mp$ as its dual.  We then describe the 0- and 1-manifold invariants given by the 1-dimensional field theory associated to a tensor category.

\subsection{The dual tensor category is the monoidal opposite}

The relevant property of the Deligne tensor product is that the process of flipping an action on a bimodule, for instance from ${}_{\cC \boxtimes \cD} \cM_\cE$ to ${}_\cC \cM_{\cD^\mp \boxtimes \cE}$ or vice versa, can be implemented by tensoring with a bimodule.
\begin{lemma} \label{lemma:flip}
For a bimodule ${}_{\cC \boxtimes \cD} \cM_\cE$, the flipped bimodule ${}_\cC \cM_{\cD^\mp \boxtimes \cE}$ is naturally equivalent, as a $\cC$--$(\cD^\mp \boxtimes \cE)$-bimodule to
\[
(\cC \boxtimes \cD) \boxtimes_{\cC \boxtimes \cD^\mp \boxtimes \cD} (\cD^\mp \boxtimes \cM).
\]
Here, the left factor is $({}_\cC \cC_\cC) \boxtimes ({}_\Vect \cD_{\cD^\mp \boxtimes \cD})$ and the right factor is $({}_{\cD^\mp} {\cD^\mp}_{\cD^\mp}) \boxtimes ({}_{\cC \boxtimes \cD} \cM_\cE)$; notice that the action of $\cC \boxtimes \cD^\mp \boxtimes \cD$ on the right factor uses the canonical symmetric monoidal switch $\cC \boxtimes \cD^\mp \ra \cD^\mp \boxtimes \cC$.

Similarly, for a bimodule ${}_\cC \cM_{\cD \boxtimes \cE}$, the flipped bimodule ${}_{\cC \boxtimes \cD^\mp} \cM_\cE$ is naturally equivalent as a bimodule to
\[
(\cM \boxtimes \cD^\mp) \boxtimes_{\cD \boxtimes \cD^\mp \boxtimes \cE} (\cD \boxtimes \cE).
\]
Here the left factor is $({}_\cC \cM_{\cD \boxtimes \cE}) \boxtimes ({}_{\cD^\mp} {\cD^\mp}_{\cD^\mp})$ and the right factor is $({}_{\cD \boxtimes \cD^\mp} \cD_\Vect) \allowbreak \boxtimes ({}_\cE \cE_\cE)$; there is a symmetric monoidal switch in the action of $\cD \boxtimes \cD^\mp \boxtimes \cE$ on the left factor, between $\cD^\mp$ and $\cE$.
\end{lemma}
\begin{proof}
For simplicity, we restrict attention to the special case of a flip from ${}_\cD \cM_\cE$ to ${}_\Vect \cM_{\cD^\mp \boxtimes \cE}$---that is, we will construct an equivalence of $\Vect$--$(\cD^\mp \boxtimes \cE)$-bimodules between $\cD \boxtimes_{\cD^\mp \boxtimes \cD} (\cD^\mp \boxtimes \cM)$ and $\cM$; the general case, and flipping the other way, are entirely similar.  

A functor $\cD \boxtimes_{\cD^\mp \boxtimes \cD} (\cD^\mp \boxtimes \cM) \ra \cM$ is determined by a $(\cD^\mp \boxtimes \cD)$-balanced bilinear functor $\cD \times (\cD^\mp \boxtimes \cM) \ra \cM$.  By definition, the data of such a balanced functor is a bilinear functor $\cD \times (\cD^\mp \boxtimes \cM) \ra \cM$ together with a coherent natural isomorphism between the two obvious trilinear functors $\cD \times (\cD^\mp \boxtimes \cD) \times (\cD^\mp \boxtimes \cM) \ra \cM$.  That data, however, is determined by a trilinear functor $\cD \times (\cD^\mp \times \cM) \ra \cM$ together with a coherent natural isomorphism of the two obvious pentalinear functors $\cD \times (\cD^\mp \times \cD) \times (\cD^\mp \times \cM) \ra \cM$.  The trilinear functor
\begin{align*}
\cD \times (\cD^\mp \times \cM) & \ra \cM \\
(d,e,m) & \mapsto e \otimes d \otimes m,
\end{align*}
together with the obvious natural isomorphism, defines the desired equivalence.
\end{proof}

We can now verify the zigzag equations of the duality between $\cC$ and $\cC^\mp$.
\begin{proposition} \label{prop:onedual}
The bimodule categories ${}_{\cC \boxtimes \cC^\mp} \cC_\Vect$ and ${}_\Vect \cC_{\cC^\mp \boxtimes \cC}$ form the evaluation and coevaluation of a duality between the tensor categories $\cC$ and $\cC^\mp$.  Thus, the 3-category $\TC$ is 1-dualizable.
\end{proposition}
\begin{proof}
Applying Lemma~\ref{lemma:flip} with the bimodule ${}_{\cC \boxtimes \cD} \cM_\cE$ being ${}_{\cC \boxtimes \cC^\mp} \cC_{\Vect}$ we obtain an equivalence
\[
{}_\cC \cC_\cC
\simeq
[({}_\cC \cC_\cC) \boxtimes ({}_\Vect {\cC^\mp}_{\cC \boxtimes \cC^\mp})] \boxtimes_{\cC \boxtimes \cC \boxtimes \cC^\mp} [({}_\cC \cC_\cC) \boxtimes ({}_{\cC \boxtimes \cC^\mp} \cC_\Vect)].
\]
Note crucially that here the action of $\cC \boxtimes \cC \boxtimes \cC^\mp$ on the term $({}_\cC \cC_\cC) \boxtimes ({}_{\cC \boxtimes \cC^\mp} \cC_\Vect)$ uses a symmetric monoidal switch between the two factors of $\cC$.  

The term $({}_\cC \cC_\cC) \boxtimes ({}_{\cC \boxtimes \cC^\mp} \cC_\Vect)$ has an identity and an evaluation, as desired, but the term $({}_\cC \cC_\cC) \boxtimes ({}_\Vect {\cC^\mp}_{\cC \boxtimes \cC^\mp})$ does not yet explicitly use the coevaluation.  Let $\sigma: \cC^\mp \boxtimes \cC \ra \cC \boxtimes \cC^\mp$ denote the symmetric monoidal switch.  Observe that the identity functor on the underlying category is a bimodule equivalence from  ${}_\Vect {\cC^\mp}_{\cC \boxtimes \cC^\mp}$ to the twisted bimodule $({}_\Vect \cC_{\cC^\mp \boxtimes \cC})_{\langle \sigma \rangle}$.  The product in question is therefore equivalent to
\begin{align*}
&[({}_\cC \cC_\cC) \boxtimes ({}_\Vect {\cC}_{\cC^\mp \boxtimes \cC})_{\langle \sigma \rangle}] \boxtimes_{\cC \boxtimes \cC \boxtimes \cC^\mp} [({}_\cC \cC_\cC) \boxtimes ({}_{\cC \boxtimes \cC^\mp} \cC_\Vect)] \\
& \hspace{2cm} \simeq
[({}_\cC \cC_\cC) \boxtimes ({}_\Vect {\cC}_{\cC^\mp \boxtimes \cC})] \boxtimes_{\cC \boxtimes \cC^\mp \boxtimes \cC} [({}_{\cC \boxtimes \cC^\mp} \cC_\Vect) \boxtimes ({}_\cC \cC_\cC)].
\end{align*}
This last equivalence is the identity on the first two copies of $\cC$ and the symmetric switch on the last two copies of $\cC$, and in the second line both the left and right actions of $\cC \boxtimes \cC^\mp \boxtimes \cC$ are without any switches.  This second line is therefore one of the usual duality zigzags for the given evaluation and coevaluation.  The other zigzag is similar.
\end{proof}

\subsection[The twice categorified 1-dimensional field theory associated to a tensor category]{\for{toc}{The twice categorified 1-dim field theory associated to a tensor cat}\except{toc}{The twice categorified 1-dimensional field theory associated to a tensor category}}

By the cobordism hypothesis, the 1-dualizability of the 3-category $\TC$ implies that there is a twice categorified 1-dimensional field theory associated to every finite tensor category.  Here ``twice categorified" means that the associated invariants of closed 1-manifolds are not number (objects of a 0-category), but linear categories (objects of a 2-category).  More precisely, we have the following.
\begin{corollary}
For each finite tensor category $\cC$, there is a unique (up to equivalence) symmetric monoidal functor
\[
\cF_\cC: \FrBord_1 \ra \TC_{(3,1)}
\]
whose value $\cF_\cC(\pt_+)$ on the standard positively framed point is the tensor category $\cC$.\footnote{A priori, the cobordism hypothesis only guarantees that there is a field theory whose value on the standard positive point is equivalent to the tensor category $\cC$.  However, as explained further in Remark~\ref{rem:hep}, by working with a cofibrant model of the bordism category and a fibrant model of $\TC$, we can choose a field theory whose value on the standard positive point is exactly $\cC$.  We suppress this distinction in the remainder of this book.\label{ftn:exactly}}  Here $\FrBord_1$ is the $(\infty,1)$-category of 1-framed 0- and 1-manifolds, and $\TC_{(3,1)}$ is the maximal sub-$(3,1)$-category of the 3-category $\TC$ of tensor categories.
\end{corollary}

The values of the field theory $\cF_\cC$ on certain elementary bordisms are listed in Table~\ref{table:egdimone}.  There, following the notation established in Section~\ref{sec:notation}, the gray fuzz on the points indicates the normal framing of the immersion in $\RR^1$.  The arrows (or lack thereof) on the intervals indicate which boundary components are outgoing (or, respectively, incoming).  The picture of the circle does not follow our usual notation using normally framed immersions: there, instead, the framing is specified directly by arrows indicating a trivialization of the tangent bundle.  Note that this is the unique 1-framed circle.

\begin{table}[!h]
\begin{tabular}{c|c}
\cb{
\begin{tikzpicture}
\draw[emptylinestyle,fuzzrightwide] (.02,-.01*\pointrad) -- (.02,.01*\pointrad);
\filldraw (0,0) circle (\pointrad); 
\end{tikzpicture}
}
& $\cC$ \\[3pt]
\cb{
\begin{tikzpicture}
\draw[emptylinestyle,fuzzrightwide] (-.02,.01*\pointrad) -- (-.02,-.01*\pointrad);
\filldraw (0,0) circle (\pointrad); 
\end{tikzpicture}
}
& $\cC^\mp$ \\[3pt]
\cb{
\begin{tikzpicture}
\draw[linestyle] (0,0) -- (1,0);
\end{tikzpicture}
}
& $\bimod{\cC \btimes \cC^\mp}{\cC}{\Vect}$ \\[3pt]
\cb{
\begin{tikzpicture}
\draw[emptylinestyle, white] (0,-.1) -- (0,.1);
\draw[linestyle] (0,0) -- (1,0);
\begin{pgfonlayer}{background}
\draw[->,outstyle] (0,0) -- +(-\arrowlength,0);
\draw[->,outstyle] (1,0) -- +(\arrowlength,0);
\end{pgfonlayer}
\end{tikzpicture}
}
& $\bimod{\Vect}{\cC}{\cC^\mp \btimes \cC}$ \\[3pt]
\cb{
\begin{tikzpicture}
\draw[emptylinestyle,white] (0,0) -- (0,\circlerad+3);
\draw[emptylinestyle,white] (0,0) -- (0,-\circlerad-3);
\draw[emptylinestyle,white] (0,0) -- (\circlerad+3,0);
\draw[emptylinestyle,white] (0,0) -- (-\circlerad-3,0);
\draw[linestyle] (0,0) circle (\circlerad);
\begin{pgfonlayer}{background}
\draw[-left to,line width=1.25*\arrowwidth,black!50] (\circlerad,0) -- +(-90:1.6*\arrowlength);
\draw[-left to,line width=1.25*\arrowwidth,black!50] (-\circlerad,0) -- +(90:1.6*\arrowlength);
\draw[-left to,line width=1.25*\arrowwidth,black!50] (0,\circlerad) -- +(0:1.6*\arrowlength);
\draw[-left to,line width=1.25*\arrowwidth,black!50] (0,-\circlerad) -- +(180:1.6*\arrowlength);
\end{pgfonlayer}
\end{tikzpicture}
}
& $\cC \btimes_{\cC \btimes \cC^\mp} \cC$
\end{tabular}
\vspace{10pt}
\caption{1-framed manifolds and their corresponding invariants.} \label{table:egdimone}
\end{table}

The invariant of the 1-framed circle deserves special attention.
\begin{definition}
The \emph{trace} of a finite tensor category $\cC$ is the relative Deligne tensor product 
\[
\cT(\cC) := \cC \boxtimes_{\cC \boxtimes \cC^\mp} \cC.
\]
In this product, the two actions are the most naive ones, namely the right action of $c \boxtimes d \in \cC \boxtimes \cC^\mp$ on $m \in \cC$ is $m \cdot (c \boxtimes d) := d \otimes m \otimes c$, and the left action of $c \boxtimes d \in \cC \boxtimes \cC^\mp$ on $m \in \cC$ is $(c \boxtimes d) \cdot m := c \otimes m \otimes d$.
\end{definition}
Contrary perhaps to expectation, the trace $\cT(\cC)$ need not be equivalent to the Drinfeld center $\cZ(\cC)$.\footnote{See~\cite{bzfn} for a comparison of derived centers and traces of categories of quasicoherent sheaves.}  Indeed, if $\cC$ is not pivotal, there is no a priori reason for $\cT(\cC)$ to even be a monoidal category.  (From a field-theoretic perspective, the monoidal structure on $\cZ(\cC)$ comes from a 2-framed pair-of-pants bordism, all of whose boundary circles are the 2-framed circle shown as the target in Example~\ref{ex:disk_bordism_immersed}.  There is no 2-framed pair-of-pants bordism, all of whose boundary circles are the 2-framed stabilization of the 1-framed circle.)  We can see the precise difference between the trace and the center by expressing the trace as a functor category.
\begin{corollary} \label{cor:trace}
For $\cC$ a finite tensor category, there is an equivalence of linear categories
\[
\cT(\cC) \simeq \Fun_{\mod{\cC}{\cC}}(({}_\cC \cC_\cC)_{\langle \mathfrak{r r} \rangle},{}_\cC \cC_\cC).
\]
An object of $\cT(\cC)$ can therefore be described as an object $x \in \cC$ together with a ``twisted half-braiding", that is a natural isomorphism between the endofunctors $\cC \ra \cC$ given by $c \mapsto x \otimes c$ and $c \mapsto c^{**} \otimes x$.
\end{corollary}

\begin{proof}
We have equivalences
\begin{align*}
\cC \boxtimes_{\cC \boxtimes \cC^\mp} \cC
&\simeq
\Fun_{\mod{\cC \boxtimes \cC^\mp}{}}((\cC_{\cC \boxtimes \cC^\mp})^*,{}_{\cC \boxtimes \cC^\mp} \cC) \\
&\simeq
\Fun_{\mod{\cC}{\cC}}((({}_\cC \cC_\cC)^*)_{\langle \mathfrak{r r} \rangle},{}_\cC \cC_\cC) \\
&\simeq
\Fun_{\mod{\cC}{\cC}}(({}_\cC \cC_\cC)_{\langle \mathfrak{r r} \rangle},{}_\cC \cC_\cC).
\end{align*}
The first equivalence is by Corollary~\ref{cor:tensasfunct}, the second equivalence is by Lemma \ref{lemma:DualTwist}, and the third equivalence is the fact that there is a $\cC$--$\cC$-bimodule equivalence between $({}_\cC \cC_\cC)^*$ and ${}_\cC \cC_\cC$, namely the functor taking $(c)^* \in \cC^*$ (the object $c \in \cC^\op$ viewed as an object of $\cC^*$) to $c^* \in \cC$ (the dual of the object $c \in \cC$).  The last statement of the corollary follows simply by expressing a left module functor ${}_\cC \cC \ra {}_\cC \cC$ as $- \mapsto - \otimes x$ and then recording the data of a right $\cC$-module structure on that functor.
\end{proof}

\begin{remark}
Though the bordisms listed in Table~\ref{table:egdimone} generate all isomorphism classes of 1-framed 0- and 1-manifolds, it is not the case that specifying the invariants of those bordisms is sufficient to determine the 1-dimensional field theory $\cF_\cC : \FrBord_1 \ra \TC_{(3,1)}$ associated to a finite tensor category $\cC$.  The trouble is that there are moduli spaces of bordisms, and because the target $\TC_{(3,1)}$ is a 3-category, it can detect the topology of those moduli spaces as higher automorphisms of the invariants.  In fact, the field theory is specified by the invariants in the table, together with one additional piece of data (coming from the 2-dimensional homotopy class in $B\mathrm{Diff}(S^1)$), namely a natural automorphism of the identity endofunctor of the trace $\cT(\cC)$.
\end{remark}

\begin{remark} \label{rem:hep}
In the above discussion (for instance in Table~\ref{table:egdimone}), and for the remainder of the book, when we say, ``the values of the field theory $\cF$ on the bordisms $M$, $N$, $P$, etc. are $C$, $D$, $E$, etc.", what we mean is that there are obvious canonical equivalences $\cF(M) \simeq C$, $\cF(N) \simeq D$, $\cF(P) \simeq E$, etc. and moreover that these equivalences commute with taking the source and target of the bordisms, respectively invariants.  In fact one can do better: in any such situation, the functor $\cF$ is equivalent to another functor $\cF'$ such that $\cF'(M) = C$, $\cF'(N) = D$, $\cF'(P) = E$, etc.---in other words, the latter functor takes exactly the desired values.  The existence of such a functor follows from a homotopy extension property for field theories; 
that property is not formal, but depends essentially on constructing and using a cofibrant model for the framed bordism category and a fibrant model for the target category $\TC$.  We can safely ignore this issue, as none of our calculations will depends on specifying invariants exactly rather than merely up to equivalence.
\end{remark}

\section[Adjoints of bimodule cats and invariants of 2-framed bordisms]{Adjoints of bimodule categories and invariants of 2-framed bordisms}   \label{sec:df-morphisms}

Next we prove that any finite bimodule category ${}_\cC \cM_\cD$ between finite tensor categories $\cC$ and $\cD$ has both a left and a right adjoint.  It follows that any finite tensor category is 2-dualizable and we describe the 0-, 1-, and 2-manifold invariants of the associated 2-dimensional field theory.

\subsection{The adjoint bimodule category is the functor dual} \label{sec:df-modules}

The left adjoint of ${}_\cC \cM_\cD$ is the functor category $\Fun_{\mod{\cC}{}}(\cM,\cC)$; this linear category has the structure of a left $\cD$-module category by precomposition by right multiplication, and has the structure of a right $\cC$-module category by postcomposition by right multiplication.  The right adjoint of ${}_\cC \cM_\cD$ is the functor category $\Fun_{\mod{}{\cD}}(\cM,\cD)$.  Recall that all functors, in particular the objects of both $\Fun_{\mod{\cC}{}}(\cM,\cC)$ and $\Fun_{\mod{}{\cD}}(\cM,\cD)$, are assumed to be right exact.

The counits of the desired adjunctions are straightforward to define as evaluation functors.  The first counit
\[
\varepsilon_1 : \cM \boxtimes_\cD \Fun_\cC(\cM,\cC) \ra \cC
\]
is induced by the $\cD$-balanced bilinear functor $\cM \times \Fun_\cC(\cM,\cC)$ given by $(m,f) \mapsto f(m)$.  Observe that $\varepsilon_1$ naturally has the structure of a $\cC$--$\cC$-bimodule functor.  The second counit
\[
\varepsilon_2 : \Fun_\cD(\cM,\cD) \boxtimes_\cC \cM \ra \cD
\]
is induced by the $\cC$-balanced functor $(f,m) \mapsto f(m)$, and naturally has the structure of a $\cD$--$\cD$-bimodule functor.

The units of the adjunctions are more subtle.  They are maps into a relative Deligne tensor product, and so cannot simply be induced by a multilinear functor.  Instead, we rely crucially on Proposition~\ref{prop:FunctorsAsATensorPdt} to reexpress the tensor product as a functor category.  The first unit is the composite
\[
\eta_1 : \cD \ra \Fun_\cC(\cM,\cM) \simeq \Fun_\cC(\cM,\cC) \boxtimes_\cC \cM,
\]
where the first map sends $d \in \cD$ to the functor $- \mapsto - \otimes d$; this map, and thus the composite, has the structure of a $\cD$--$\cD$-bimodule functor.  The second unit is the composite
\[
\eta_2 : \cC \ra \Fun_\cD(\cM,\cM) \simeq \cM \boxtimes_\cD \Fun_\cD(\cM,\cD),
\]
where the first map sends $c \in \cC$ to the functor $- \mapsto c \otimes -$; this map, and thus the composite, has the structure of a $\cC$--$\cC$-bimodule functor.
\begin{proposition} \label{prop:evcoev}
Let ${}_\cC \cM_\cD$ be a finite bimodule category between finite tensor categories.  The pair of functors $(\eta_1,\varepsilon_1)$ described above forms the unit and counit of an adjunction
\[
\Fun_\cC(\cM,\cC) \dashv \cM.
\]
Similarly, the pair of functors $(\eta_2,\varepsilon_2)$ forms the unit and counit of an adjunction
\[
\cM \dashv \Fun_\cD(\cM,\cD).
\]
\end{proposition}
\begin{proof}
Observe, using the definition of $\eta_1$ and the construction of the equivalence in Proposition~\ref{prop:FunctorsAsATensorPdt}, that the following diagram commutes:
		\begin{center}
			\begin{tikzpicture}[align=left]	
				\node (A) at (0in,2in) {${}_\cC \cM_\cD$};
				\node (B) at (0in,1.5in) {${}_\cC \cM \boxtimes_\cD \cD_\cD$};
				\node (C) at (0in,1in) {${}_\cC \cM \boxtimes_\cD \Fun_{\cC\text{-mod}}(\cM,\cC) \boxtimes_\cC \cM_\cD$};
				\node (D) at (2.5in,1in) {${}_\cC \cM \boxtimes_\cD \Fun_{\cC\text{-mod}}(\cM,\cM)_\cD$};
				\node (E) at (0in, .5in) {${}_\cC \cC \boxtimes_\cC \cM_\cD$};
				\node (F) at (0in,0in) {${}_\cC \cM_\cD$};
				\draw [<-] (A) to node [left] {$\simeq$} (B);
				\draw [->] (B) to node [left=.1cm] {$\id \boxtimes \eta_1$}  (C);
				\draw [->,bend left=14] (B) to node [above=.1cm]{$\id \boxtimes (d \mapsto (- \otimes d))$} (D);
				\draw [->] (C) to node [above] {$\simeq$} (D);
				\draw [->] (C) to node [left=.1cm]{$\varepsilon_1 \boxtimes \id$} (E);
				\draw [->] (E) to node [left] {$\simeq$} (F);
				\draw [->,bend left=22] (D) to node [right=.25cm]{$m \boxtimes f \mapsto f(m)$} (F);
			\end{tikzpicture}
		\end{center}
The composite along the right is certainly naturally equivalent to the identity, while the composite along the left is the required adjunction zigzag.  The other three zigzag relations are similar.
\end{proof}

\subsection[The categorified 2-dimensional field theory associated to a finite tensor category]{\for{toc}{The categorified 2-dim field theory associated to a finite tensor cat}\except{toc}{The categorified 2-dimensional field theory associated to a finite tensor category}}

Combining Propositions~\ref{prop:onedual} and~\ref{prop:evcoev}, we obtain the following.
\begin{theorem} \label{thm:TC_is_2Dualizable}
The 3-category $\TC$ of finite tensor categories, finite bimodule categories, their functors, and transformations, is 2-dualizable.
\end{theorem}
\begin{corollary} \label{cor:2dualtft}
For each finite tensor category $\cC$, there is a unique (up to equivalence) symmetric monoidal functor
\[
\cF_\cC : \FrBord_2 \ra \TC_{(3,2)},
\]
whose value $\cF_\cC(\pt_+)$ on the standard positively 2-framed point is the tensor category $\cC$.\footnote{As in Footnote~\ref{ftn:exactly}, to ensure that the field theory takes exactly the value $\cC$ on the standard positive point (rather than merely a value equivalent to $\cC$), requires a particular choice of model of the 3-category of tensor categories.}  Here $\FrBord_2$ is the 2-framed bordism $(\infty,2)$-category and $\TC_{(3,2)}$ is the maximal sub-$(3,2)$-category of the 3-category $\TC$ of tensor categories.
\end{corollary}

\nid Note that this field theory is `categorified' in the sense that the invariants of closed 2-manifolds are not numbers, but vector spaces, and in the sense that it assigns an isomorphism of vector spaces to each isotopy class of diffeomorphisms of a surface.

As in Section~\ref{sec:framed-duality}, we focus on the following standard positively 2-framed point $\pt_+$, together with the following chosen negatively 2-framed point, denoted $\pt_-$, and evaluation and coevaluation bordisms $\ev$ and $\coev$ witnessing a duality between $\pt_+$ and $\pt_-$:
$$\cb{
\begin{tikzpicture}
\filldraw (0,0) circle (\pointrad);
\begin{pgfonlayer}{background}
\draw[->,outstyle] (0,0) -- +(0:\arrowlength) node[anchor=south west,inner sep=1pt] {\tiny 1};
\draw[->,outstyle] (0,0) -- +(90:\arrowlength) node[anchor=south west,inner sep=1pt] {\tiny 2};
\draw[->,outstyle,white] (0,0) -- +(-90:\arrowlength) node[anchor=north west,inner sep=1pt] {\tiny 2};
\end{pgfonlayer}
\end{tikzpicture}
}
,
\cb{
\begin{tikzpicture}
\filldraw (0,0) circle (\pointrad);
\begin{pgfonlayer}{background}
\draw[->,outstyle] (0,0) -- +(0:\arrowlength) node[anchor=north west,inner sep=1pt] {\tiny 1};
\draw[->,outstyle] (0,0) -- +(-90:\arrowlength) node[anchor=north west,inner sep=1pt] {\tiny 2};
\draw[->,outstyle,white] (0,0) -- +(90:\arrowlength) node[anchor=south west,inner sep=1pt] {\tiny 2};
\end{pgfonlayer}
\end{tikzpicture}
}
,
\cb{
\begin{tikzpicture}
\draw[linestyle,fuzzright] (0,0) arc (-90:90:\smcirclerad);
\end{tikzpicture}
}
,
\cb{
\begin{tikzpicture}
\draw[emptylinestyle, white] (0,.1) -- (0,-2*\smcirclerad) -- +(0,-.1);
\draw[linestyle,fuzzleft] (0,0) arc (90:270:\smcirclerad);
\begin{pgfonlayer}{background}
	\draw[->,outstyle] (0,0) -- +(0:\arrowlength);
	\draw[->,outstyle] (0,-2*\smcirclerad) -- +(0:\arrowlength);
\end{pgfonlayer}
\end{tikzpicture}
}$$

\nid We also focus on the particular adjoints to $\ev$ and $\coev$ drawn in Figure~\ref{fig:adjointchains} in Section~\ref{sec:framed-duality}, and on various composites of these particular bordisms.  

The invariants of the field theory $\cF_\cC$ on these various 2-framed manifolds are listed in Table~\ref{table:intervals}.\footnote{We are considering a field theory $\cF_\cC$ whose value on the standard positive point $\pt_+$ is $\cC$.  As the functor $\cF_\cC$ preserves dualities, it send our chosen negative point $\pt_-$ to a dual of $\cC$.  Note, though, that the value $\cF_\cC(\pt_-)$ is canonically equivalent to $\cC^\mp$: the maps $\cF_\cC(\ev)$ and $\cF_\cC(\coev)$ witness a duality between $\cF_\cC(\pt_+) = \cC$ and $\cF_\cC(\pt_-)$, and we already have a chosen dual $\cC^\mp$ to $\cC$ witnessed by ${}_{\cC \boxtimes \cC^\mp} \cC_\Vect$ and ${}_\Vect \cC_{\cC^\mp \boxtimes \cC}$, and therefore we have a canonical equivalence between $\cF_\cC(\pt_-)$ and $\cC^\mp$.  The existence of this canonical equivalence is the meaning of listing $\cC^\mp$ as the value of the field theory in the table.  Similarly, for all the other manifolds $M$ in the table, there is a canonical equivalence between $\cF_\cC(M)$ and the listed value.  More directly, if brutally, we may simply homotope the field theory $\cF_\cC$ to a theory that takes exactly the listed values, as mentioned in Remark~\ref{rem:hep}.}  In that table, the equivalence in each of the last four rows is an application of Proposition~\ref{prop:dual-formula-for-adjoints}.  The values of the adjoints $\ev^L$, $\coev^L$, $\ev^R$, and $\coev^R$ are determined from the values of $\ev$ and $\coev$ by Proposition~\ref{prop:evcoev}.

\begin{table}[!hbt] 
\begin{tabular}{ccc|rcl}
$\pt_+$ & $=$ & 
\cb{
\begin{tikzpicture}
\filldraw (0,0) circle (\pointrad);
\begin{pgfonlayer}{background}
\draw[->,outstyle] (0,0) -- +(0:\arrowlength) node[anchor=south west,inner sep=1pt] {\tiny 1};
\draw[->,outstyle] (0,0) -- +(90:\arrowlength) node[anchor=south west,inner sep=1pt] {\tiny 2};
\draw[->,outstyle,white] (0,0) -- +(-90:.8*\arrowlength) node[anchor=north west,inner sep=1pt] {\tiny 2};
\end{pgfonlayer}
\end{tikzpicture}
}
& \multicolumn{1}{c}{$\cC$} &&\\[-10pt]
$\pt_-$ & $=$ & 
\cb{
\begin{tikzpicture}
\filldraw (0,0) circle (\pointrad);
\begin{pgfonlayer}{background}
\draw[->,outstyle] (0,0) -- +(0:\arrowlength) node[anchor=north west,inner sep=1pt] {\tiny 1};
\draw[->,outstyle] (0,0) -- +(-90:\arrowlength) node[anchor=north west,inner sep=1pt] {\tiny 2};
\draw[->,outstyle,white] (0,0) -- +(90:.8*\arrowlength) node[anchor=south west,inner sep=1pt] {\tiny 2};
\end{pgfonlayer}
\end{tikzpicture}
}
& \multicolumn{1}{c}{$\cC^\mp$} &&\\[10pt]
$\ev$ & $=$ & \cb{
\begin{tikzpicture}
\draw[linestyle,fuzzright] (0,0) arc (-90:90:\smcirclerad);
\end{tikzpicture}
}
& \multicolumn{1}{c}{$\bimod{\cC \btimes \cC^\mp}{\cC}{\Vect}$} 
& &  \\[8pt]
$\coev$ & $=$ & \cb{
\begin{tikzpicture}
\draw[emptylinestyle, white] (0,.1) -- (0,-2*\smcirclerad) -- +(0,-.1);
\draw[linestyle,fuzzleft] (0,0) arc (90:270:\smcirclerad);
\begin{pgfonlayer}{background}
	\draw[->,outstyle] (0,0) -- +(0:\arrowlength);
	\draw[->,outstyle] (0,-2*\smcirclerad) -- +(0:\arrowlength);
\end{pgfonlayer}
\end{tikzpicture}
}
& \multicolumn{1}{c}{$\bimod{\Vect}{\cC}{\cC^\mp \btimes \cC}$} 
& 
& \\[8pt]
$\ev^L$ & $=$ & \cb{
\begin{tikzpicture}
\draw[linestyle,fuzzright] (0,0) arc (90:270:\smcirclerad);
\begin{pgfonlayer}{background}
	\draw[->,outstyle] (0,0) -- +(0:\arrowlength);
	\draw[->,outstyle] (0,-2*\smcirclerad) -- +(0:\arrowlength);
\end{pgfonlayer}
\end{tikzpicture}
}
& $\Fun_{\mod{\cC \boxtimes \cC^\mp}{}}(\cC,\cC \boxtimes \cC^\mp)$
& $\simeq$ & ${}^*({}_{\cC \boxtimes \cC^\mp} \cC_\Vect)$
\\[8pt]
$\coev^L$ & $=$ & \cb{
\begin{tikzpicture}
\draw[linestyle,fuzzleft] (0,0) arc (-90:90:\smcirclerad);
\end{tikzpicture}
} 
& $\Fun_{\Vect}(\cC,\Vect)$
& $\simeq$ & ${}^*({}_\Vect \cC_{\cC^\mp \boxtimes \cC})$ 
\\[10pt]
$\ev^R$ & $=$ & 
\setlength{\linewid}{15pt}
\setlength{\fuzzwidth}{25pt}
\setlength{\arrowlength}{80pt}
\setlength{\arrowwidth}{7.5pt}
\cb{
\scalebox{.1}{
\begin{tikzpicture}
\draw[linestyle,fuzzright]
(0,5) to [out=180, in=70] (-4,2.75)
	to [looseness=1.6, out=-110, in=-90] (-7,2.75)
	to [looseness=1.6, out=90, in=110] (-4,2.75)
	to [out=-70, in=70] (-4,-2.75)
	to [looseness=1.6, out=-110, in=-90] (-7,-2.75)
	to [looseness=1.6, out=90, in=110] (-4,-2.75)
	to [out=-70, in=180] (0,-5);
\begin{pgfonlayer}{background}
	\draw[-scalehead,outstyle] (0,5) -- +(0:\arrowlength);
	\draw[-scalehead,outstyle] (0,-5) -- +(0:\arrowlength);
\end{pgfonlayer}
\end{tikzpicture}
}
}
\setlength{\linewid}{1.5pt}
\setlength{\fuzzwidth}{2.5pt}
\setlength{\arrowlength}{8pt}
\setlength{\arrowwidth}{.75pt}
& $\Fun_{\Vect}(\cC,\Vect)$
& $\simeq$ & $({}_{\cC \boxtimes \cC^\mp} \cC_\Vect)^*$
\\[16pt]
$\coev^R$ & $=$ & 
\setlength{\linewid}{15pt}
\setlength{\fuzzwidth}{25pt}
\setlength{\arrowlength}{80pt}
\setlength{\arrowwidth}{7.5pt}
\cb{
\scalebox{.1}{
\begin{tikzpicture}
\draw[linestyle,fuzzright]
(0,5) to [out=0, in=110] (4,2.75)
	to [looseness=1.6, out=-70, in=-90] (7,2.75)
	to [looseness=1.6, out=90, in=70] (4,2.75)
	to [out=-110, in=110] (4,-2.75)
	to [looseness=1.6, out=-70, in=-90] (7,-2.75)
	to [looseness=1.6, out=90, in=70] (4,-2.75)
	to [out=-110, in=0] (0,-5);
\end{tikzpicture}
}
}
\setlength{\linewid}{1.5pt}
\setlength{\fuzzwidth}{2.5pt}
\setlength{\arrowlength}{8pt}
\setlength{\arrowwidth}{.75pt}
& $\Fun_{\mod{}{\cC^\mp \boxtimes \cC}}(\cC,\cC^\mp \boxtimes \cC)$
& $\simeq$ & $({}_\Vect \cC_{\cC^\mp \boxtimes \cC})^*$ \\
\end{tabular}
\vspace{6pt}
\caption{Invariants associated to 2-framed points and intervals.} \label{table:intervals}
\end{table}

We are now, finally, in a position to calculate the Serre automorphism of a finite tensor category; the result of this calculation is the central connection between the topology of framed bordisms and the algebra of duality in tensor categories.
\begin{theorem} \label{thm:serreisdoubledual}
The Serre automorphism $\cS_\cC$ of the finite tensor category $\cC$ is the double-right-dual twist of the identity bimodule of $\cC$; the inverse Serre automorphism $\cS^{-1}_\cC$ is the double-left-dual twist of the identity bimodule of $\cC$.  In other words, the 2-dimensional 2-framed field theory $\cF_\cC$ associated to $\cC$ has the following values on the Serre $\cS$ and inverse Serre $\cS^{-1}$ bordisms, respectively:
\begin{align*}
\cF_\cC\left(
\cb{
\begin{tikzpicture}
\draw[emptylinestyle, white] (.7,.1) -- (.7,-.1);
\draw[linestyle,fuzzright] 
(.7,0) to [out=180, in=-20] (0,.1)
	to [looseness=1.6, out=160, in=180] (0,.4)
	to [looseness=1.6, out=0, in=20] (0,.1)
	to [out=-160, in=0] (-.7,0);
\begin{pgfonlayer}{background}
	\draw[->,outstyle] (.7,0) -- +(0:\arrowlength);
\end{pgfonlayer}
\end{tikzpicture}
}
\right)
&=
{}_{\langle \mathfrak{r r} \rangle \cC} \cC_\cC \\
\cF_\cC\left(
\cb{
\begin{tikzpicture}
\draw[linestyle,fuzzright] 
(.7,0) to [out=180, in=20] (0,-.1)
	to [looseness=1.6, out=-160, in=180] (0,-.4)
	to [looseness=1.6, out=0, in=-20] (0,-.1)
	to [out=160, in=0] (-.7,0);
\begin{pgfonlayer}{background}
	\draw[->,outstyle] (.7,0) -- +(0:\arrowlength);
\end{pgfonlayer}
\end{tikzpicture}
}
\right)
&=
{}_{\langle \mathfrak{l l} \rangle \cC} \cC_\cC
\end{align*}
\end{theorem}
\begin{proof}
By Proposition~\ref{prop:serrecalc}, we know the Serre and inverse Serre automorphisms of $\cC$ are given by the following composites:
\begin{align*}
\cS_\cC &= (\id_\cC \boxtimes \ev_\cC^R) \boxtimes_{\cC \boxtimes \cC \boxtimes \cC^\mp} (\tau_{\cC,\cC} \boxtimes \id_{\cC^\mp}) \boxtimes_{\cC \boxtimes \cC \boxtimes \cC^\mp} (\id_\cC \boxtimes \ev_\cC), \\
\cS^{-1}_\cC &= (\id_\cC \boxtimes \ev_\cC^L) \boxtimes_{\cC \boxtimes \cC \boxtimes \cC^\mp} (\tau_{\cC,\cC} \boxtimes \id_{\cC^\mp}) \boxtimes_{\cC \boxtimes \cC \boxtimes \cC^\mp} (\id_\cC \boxtimes \ev_\cC).
\end{align*}
Here $\ev_\cC$, $\ev^R_\cC$, and $\ev^L_\cC$ denote the values of the field theory $\cF_\cC$ on the bordisms $\ev$, $\ev^R$, and $\ev^L$.  Substituting these values from Table~\ref{table:intervals}, we have the formulas
\begin{align*}
\cS_\cC &= (({}_\cC \cC_\cC) \boxtimes ({}_\Vect \Fun_\Vect(\cC,\Vect)_{\cC \boxtimes \cC^\mp})) \boxtimes_{\cC \boxtimes \cC \boxtimes \cC^\mp} (({}_\cC \cC_\cC) \boxtimes ({}_{\cC \boxtimes \cC^\mp} \cC_\Vect)), \\
\cS^{-1}_\cC &= (({}_\cC \cC_\cC) \boxtimes ({}_\Vect \Fun_{\mod{\cC \boxtimes \cC^\mp}{}}(\cC,\cC \boxtimes \cC^\mp)_{\cC \boxtimes \cC^\mp})) \boxtimes_{\cC \boxtimes \cC \boxtimes \cC^\mp} \\ &\hspace*{3.3in} (({}_\cC \cC_\cC) \boxtimes ({}_{\cC \boxtimes \cC^\mp} \cC_\Vect)),
\end{align*}
where in both formulas we have now let be implicit the symmetric monoidal switch between the two factors of $\cC$ in the middle product $\boxtimes_{\cC \boxtimes \cC \boxtimes \cC^\mp}$.

By judicious application of Lemma~\ref{lemma:flip}, we recognize the above expressions for $\cS_\cC$ and $\cS^{-1}_\cC$ as, respectively, the flipped bimodules 
\begin{gather*}
{}_\cC ({}_\Vect \Fun_\Vect(\cC,\Vect)_{\cC \boxtimes \cC^\mp})_\cC, \\
{}_\cC ({}_\Vect \Fun_{\mod{\cC \boxtimes \cC^\mp}{}}(\cC,\cC \boxtimes \cC^\mp)_{\cC \boxtimes \cC^\mp})_\cC.
\end{gather*}
By Proposition~\ref{prop:dual-formula-for-adjoints}, Lemma~\ref{lemma:DualTwist}, and the existence of bimodule equivalences between $({}_\cC \cC_\cC)^*$ and ${}^*({}_\cC \cC_\cC)$ and ${}_\cC \cC_\cC$ (see the proof of Corollary~\ref{cor:trace}), we now have equivalences
\begin{alignat*}{3}
{}_\cC ({}_\Vect \Fun_\Vect(\cC,\Vect)_{\cC \boxtimes \cC^\mp})_\cC 
& \simeq  {}_\cC(({}_{\cC \boxtimes \cC^\mp} \cC_{\Vect})^*)_\cC
&& \simeq  \bimod{\langle \mathfrak{r r} \rangle \cC}{\cC}{\cC}, &&\\
{}_\cC ({}_\Vect \Fun_{\mod{\cC \boxtimes \cC^\mp}{}}(\cC,\cC \boxtimes \cC^\mp)_{\cC \boxtimes \cC^\mp})_\cC
& \simeq  {}_\cC ({}^*(\bimod{\cC \boxtimes \cC^\mp}{\cC}{\Vect}))_\cC
&& \simeq  \bimod{\langle \mathfrak{l l} \rangle \cC}{\cC}{\cC}, &&
\end{alignat*}
as required.  
\end{proof}

By composing the invariants of the intervals in Table~\ref{table:intervals}, we can compute the invariants for various 2-framed circles; the three most important such circles and their resulting invariants are listed in Table~\ref{table:circles}.  In that table, the first equivalence is by Proposition~\ref{prop:FunctorsAsATensorPdt}; the second and third equivalences are by Corollary~\ref{cor:tensasfunct} and Lemma~\ref{lemma:DualTwist}.  Finally, in Table~\ref{table:discs}, we record the values of the 2-framed discs arising as the units and counits of the adjunctions $\ev^L_\cC \dashv \ev_\cC$ and $\coev^L_\cC \dashv \coev_\cC$.
\begin{table}[!ht]
\begin{tabular}{c|rcl}
\cb{
\begin{tikzpicture}
\draw[linestyle,fuzzright] (0,0) circle (\smcirclerad);
\end{tikzpicture}
}
& $\Fun_{\cC \boxtimes \cC^\mp}(\cC,\cC \boxtimes \cC^\mp)\boxtimes_{\cC \btimes \cC^\mp} \cC$ 
& $\simeq$ & $\Fun_{\cC\text{-mod-}\cC}(\cC, \cC) =: \cZ(\cC)$ \\[6pt]
\cb{
\begin{tikzpicture}
\draw[rotate=-90,linestyle,fuzzleft,looseness=2]
(0,.5) to [out=0, in=10] (0,0)
	to [out=-170, in=180] (0,-.5)
	to [out=0, in=-10] (0,0)
	to [out=170, in=180] (0,.5);
\end{tikzpicture}
}
& $\cT(\cC) := \cC \boxtimes_{\cC \boxtimes \cC^\mp} \cC$
& $\simeq$ & $\Fun_{\cC\text{-mod-}\cC}(\cC_{\langle \mathfrak{r r} \rangle}, \cC)$\\[8pt]
\cb{
\begin{tikzpicture}
\draw[linestyle,fuzzleft] (0,0) circle (\smcirclerad);
\end{tikzpicture}
}
& $\cC \boxtimes_{\cC^\mp \boxtimes \cC} \Fun(\cC, \Vect)$ 
& $\simeq$ & $\Fun_{\cC\text{-mod-}\cC}(\cC_{\langle \mathfrak{r r} \rangle}, \cC_{\langle \mathfrak{l l} \rangle}) =: \coZ(\cC)$
\end{tabular}
\vspace{6pt}
\caption{Invariants associated to 2-framed circles.} \label{table:circles}
\end{table}

\begin{table}[!hbt] 
\begin{tabular}{c|l}
\cb{
\begin{tikzpicture}
\filldraw[linestyle,fuzzright,fill=\fillcolor] (0,0) circle (\circlerad);
\end{tikzpicture}
}
& $\Vect \xra{k \mapsto \id} {\Fun_{\cC\text{-mod-}\cC}(\cC,\cC)}$ \\[6pt]
\cb{
\begin{tikzpicture}
\filldraw[linestyle,fill=\fillcolor] 
	(0,0) .. controls (.25,.25) and (.75,.25) .. (1,0)
		.. controls (.75,.25) and (.75,.75) .. (1,1)
		.. controls (.75,.75) and (.25,.75) .. (0,1)
		.. controls (.25,.75) and (.25,.25) .. (0,0);
\draw[linestyle,fuzzright]
	(0,0) .. controls (.25,.25) and (.75,.25) .. (1,0);
\draw[linestyle,fuzzleft]
	(0,1) .. controls (.25,.75) and (.75,.75) .. (1,1);
\begin{pgfonlayer}{background}
	\draw[->,outstyle] (1,1) -- +(45:\arrowlength);
	\draw[->,outstyle] (1,0) -- +(-45:\arrowlength);
\end{pgfonlayer}
\end{tikzpicture}
}
& $\cC \boxtimes \Fun_{\cC \btimes \cC^\mp\text{-mod}}(\cC,\cC \btimes \cC^\mp) \xra{\mathrm{eval}} \cC \btimes \cC^\mp$ \\[6pt]
\cb{
\begin{tikzpicture}
\filldraw[linestyle,fill=\fillcolor] 
	(0,0) .. controls (.25,.25) and (.75,.25) .. (1,0)
		.. controls (.75,.25) and (.75,.75) .. (1,1)
		.. controls (.75,.75) and (.25,.75) .. (0,1)
		.. controls (.25,.75) and (.25,.25) .. (0,0);
\draw[linestyle, fuzzleft]
	(0,0) .. controls (.25,.25) and (.25,.75) .. (0,1);
\draw[linestyle, fuzzright]
	(1,0) .. controls (.75,.25) and (.75,.75) .. (1,1);
\begin{pgfonlayer}{background}
	\draw[->,outstyle] (1,1) -- +(45:\arrowlength);
	\draw[->,outstyle] (1,0) -- +(-45:\arrowlength);
\end{pgfonlayer}
\end{tikzpicture}
}
& $\cC^\mp \btimes \cC \xra{1 \btimes 1 \mapsto \id} \Fun_{\Vect}(\cC,\cC)$ \\[6pt]
\cb{
\begin{tikzpicture}
\filldraw[linestyle,fill=\fillcolor] (0,0) circle (\circlerad);
\end{tikzpicture}
}& $\cC \boxtimes_{\cC^\mp \boxtimes \cC} \Fun_{\Vect}(\cC,\Vect) \xra{\mathrm{eval}} \Vect$
\end{tabular}
\vspace{6pt}
\caption{Invariants associated to 2-framed discs.} \label{table:discs}
\end{table}

\vspace{-10pt}
\begin{remark}
For any finite tensor category $\cC$, the center $\cZ(\cC)$, as a category of endofunctors, is a monoidal category.  The unit, namely the identity functor, is the value of the field theory $\cF_\cC$ on the first 2-framed disc in Table~\ref{table:discs}; the tensor product, namely composition, is the value of the field theory on the 2-framed pair-of-pants obtained by removing two discs from the interior of that unit disc.  However, the monoidal category $\cZ(\cC)$ is not, a priori, even weakly rigid: there is no 2-framing of the annular bordism from the empty set to two circles (or from two circles to the empty set), such that both circles inherit the 2-framing corresponding to $\cZ(\cC)$.  

The trace $\cT(\cC)$ is not a priori monoidal, as there is no 2-framed pair-of-pants bordism from two circles to one circle such that all three circles inherit the 2-framing corresponding to $\cT(\cC)$.  The cocenter $\coZ(\cC)$, defined in Table~\ref{table:circles}, is naturally a comonoidal category.  The counit is the last 2-framed disc in Table~\ref{table:discs}, and the coproduct is that disc with two smaller discs removed.  Though a priori $\cZ(\cC)$ has no pairing operation, there is a pairing between the center and the cocenter: the field theory assigns a functor $\cZ(\cC) \boxtimes \coZ(\cC) \ra \Vect$ to the planar annulus with both boundaries incoming.

In the next section we will prove that a finite tensor category $\cC$ is not only 2-dualizable, but is actually a Radford object.  By Theorem~\ref{thm:Cat_Radford}, there is therefore a canonical isomorphism between the adjoints $\ev^L_\cC$ and $\ev^R_\cC$.  It follows that there is a canonical isomorphism between the center $\cZ(\cC)$ and the cocenter $\coZ(\cC)$, and therefore that the center does indeed have a pairing operation.  Implicitly, the field theory $\cF_\cC$ is taking values on certain 3-framed manifolds: there are only two distinct 3-framed circles (with representatives corresponding to $\cZ(\cC)$ and $\cT(\cC)$) and there is a 3-framed annulus with both outgoing boundary components giving $\cZ(\cC)$, namely the annulus immersed in $\RR^3$ as a macaroni.
\end{remark}

\section{The Radford adjoints and the quadruple dual} \label{sec:radfordftc}

Though in general a finite tensor category is not 3-dualizable and therefore does not provide a full 3-framed 3-dimensional field theory, it is nevertheless the case that the 2-dimensional field theory associated to a finite tensor category can always be extended to take values on \emph{some} 3-dimensional bordisms, providing a kind of ``non-compact" 3-dimensional field theory.  This theory, and the terminology, is in the same spirit as work of Costello~\cite{MR2298823} and Lurie~\cite[\S 4.2]{lurie-ch} on non-compact field theories in dimension 2.

In particular, the field theory associated to a finite tensor category extends to take values on the 3-manifolds arising as the units and counits of the adjunctions $u_2 \dashv u_2^R$ and $v_2 \dashv v_2^R$, where $u_2$ and $v_2$ are the 2-manifolds drawn in Example~\ref{eg:evrevadj}.  The existence of those two ``Radford" adjunctions is formalized in the notion of a Radford object, from Definition~\ref{def:Radford-Object}.  In Section~\ref{sec:finiteradford}, we prove that every finite tensor category $\cC$ is a Radford object, and as a consequence that the categorified 2-framed 2-dimensional field theory $\cF_\cC$ extends to a categorified 3-framed 2-dimensional field theory $\widetilde{\cF_\cC}$.  

In Section~\ref{sec:topquaddual}, evaluating the field theory $\widetilde{\cF_\cC}$ on the 3-framed Radford bordism provides a transparent topological proof of the theorem, originally due to Radford~\cite{MR0407069} and Etingof-Nikshych-Ostrik~\cite{MR2097289}, that the quadruple dual functor is (nearly) trivial.  Roughly speaking the argument is simply: a 360 degree rotation implements the double dual, so the Dirac belt trick trivializes the quadruple dual.  In Section~\ref{sec:computeradford}, we explicitly compute the trivialization in question by reexpressing the various module categories involved as categories of internal modules.

The reader who is content to concentrate on separable tensor categories (for instance by restricting to fusion categories of nonzero global dimension over an algebraically closed field (see Theorem~\ref{thm:NonzeroDimension}), or by restricting to finite semisimple tensor categories over a field of characteristic zero (see Corollary~\ref{cor:charzerosep})) can safely skip this section.  In Section~\ref{sec:separableisfd} we will prove that separable tensor categories are fully dualizable, and the results of this section for separable tensor categories (that they are Radford, the existence of an associated 3-framed 2-dimensional field theory, and the triviality of the quadruple dual) are immediate consequences.

\subsection{Finite tensor categories are Radford objects} \label{sec:finiteradford}

In the previous section we saw that every finite tensor category $\cC$ is 2-dualizable.  In particular, there are two infinite chains of adjunctions $\cdots \dashv \ev_\cC^{LL} \dashv \ev_\cC^L \dashv \ev_\cC \dashv \ev_\cC^R \dashv \ev_\cC^{RR} \dashv \cdots$ and $\cdots \dashv \coev_\cC^{LL} \dashv \coev_\cC^L \dashv \coev_\cC \dashv \coev_\cC^R \dashv \coev_\cC^{RR} \dashv \cdots$, where as before $\ev_\cC$ and $\coev_\cC$ are the standard evaluation and coevaluation maps witnessing the duality between $\cC$ and $\cC^\mp$.  It is not the case that all the units and counits of the adjunctions in these chains admit left and right adjoints (in which case the tensor category would be 3-dualizable), but some of these adjoints do exist.  The most important of these adjoints are isolated in the notion of a Radford object: recall from Definition~\ref{def:Radford-Object} that a 2-dualizable tensor category $\cC$ is Radford if the unit and the counit of the adjunction $\ev_\cC \dashv \ev_\cC^R$ both have right adjoints.  The perhaps surprising fact is that all finite tensor categories admits these adjoints.  The proof relies crucially on Theorem~\ref{thm:tensor-exactness}, that tensor products of exact module categories are exact.
\begin{theorem} \label{thm:TCisRadford}
Every finite tensor category is a Radford object of the 3-category $\TC$ of tensor categories.
\end{theorem}
\begin{proof}
We need to construct right adjoints to the $\cC \boxtimes \cC^\mp$--$\cC \boxtimes \cC^\mp$-bimodule unit map $u: \cC \boxtimes \cC^\mp \ra \ev_\cC \boxtimes \ev_\cC^R$ and to the $\Vect$--$\Vect$-bimodule counit map $v: \ev_\cC^R \boxtimes_{\cC \boxtimes \cC^\mp} \ev_\cC \ra \Vect$.

By assumption, the bimodule functors $u$ and $v$ in question are right exact linear functors, and therefore as linear functors have not-necessarily-right-exact right adjoints $u^R$ and $v^R$.  By Lemma~\ref{lma:module-adjoint}, these functors $u^R$ and $v^R$ can be promoted to right adjoints as bimodule functors.  Now recall from Theorem~\ref{Thm:ExactModCatOmnibus} that every not-necessarily-right-exact module functor out of an exact module category is exact, in particular right exact.  Therefore, provided the sources of $u^R$ and $v^R$ are exact bimodule categories, then $u$ and $v$ have right adjoints in $\TC$, as required.

The $\Vect$--$\Vect$-bimodule $\Vect$ is evidently exact.  By Example~\ref{ex:exactness}, the $\cC \boxtimes \cC^\mp$--$\Vect$-bimodule evaluation $\ev_\cC$ is exact.  Combining Proposition~\ref{prop:evcoev}, Proposition~\ref{prop:dual-formula-for-adjoints}, and Corollary~\ref{cor:adjoint-exactness}, it follows that the adjoint $\Vect$--$\cC \boxtimes \cC^\mp$-bimodule $\ev_\cC^R$ is exact.  Applying Theorem~\ref{thm:tensor-exactness}, we conclude that the composite bimodule $\ev_\cC \boxtimes \ev_\cC^R$ is exact, as required.
\end{proof}

\nid The reasoning in the above proof, together with further similar arguments, provides the following general condition for the existence of adjoints of bimodule functors:
\begin{proposition}
Let ${}_\cC \cM_\cD$ and ${}_\cC \cN_\cD$ be finite bimodule categories between finite tensor categories, and let $\cF : \cM \ra \cN$ be a (right exact) bimodule functor.  If the bimodule category $\cN$ is exact, then $\cF$ has a right adjoint (right exact) bimodule functor.  If the bimodule category $\cM$ is exact, then $\cF$ has a left adjoint (right exact) bimodule functor.
\end{proposition}

The Radford property of a finite tensor category allows us to promote the associated 2-framed 2-dimensional field theory to a 3-framed 2-dimensional field theory.  (This provides, for instance, projective representations of the spin mapping class groups of spin surfaces.)

\begin{corollary} \label{cor:3fr2d}
For each finite tensor category $\cC$, there is a unique (up to equivalence) symmetric monoidal functor
\[
\widetilde{\cF_\cC} : \mathrm{Bord}^{3\text{-}\mathrm{fr}}_2 \ra \TC_{(3,2)},
\]
extending the 2-framed 2-dimensional field theory $\cF_\cC: \FrBord_2 \ra \TC_{(3,2)}$ to the $(\infty,2)$-category $\mathrm{Bord}^{3\text{-}\mathrm{fr}}_2$ of 3-framed 0-, 1-, and 2-dimensional bordisms.
\end{corollary}

\nid The proof assumes more familiarity with the cobordism hypothesis than we have heretofore presumed (see Theorem~\ref{thm:chstruc}), and the reader can skip it without consequence for the remainder of the book.

\begin{proof}
As before, let $\cF_\cC$ denote the 2-framed 2-dimensional field theory associated to the finite tensor category $\cC$.  By the cobordism hypothesis, a descent of $\cF_\cC$ to a 3-framed theory is given by providing $\cC$ with the structure of an $\Omega(O(3)/O(2))$-homotopy fixed point.  Here $\Omega(O(3)/O(2))$ acts on $\TC$ via the map $\Omega(O(3)/O(2)) \simeq \Omega S^2 \xra{2} SO(2) \ra \mathrm{Aut}(\TC)$, where the $SO(2)$ action on $\TC$ is via the Serre automorphism.  Such a homotopy fixed point is precisely determined by a null homotopy of the square of the Serre automorphism, which is provided for any Radford object by Corollary~\ref{cor:serreinvserre}.
\end{proof}

The next step is to study the invariants of the 3-framed 2-dimensional field theory $\widetilde{\cF_\cC}$ given by this corollary.  However, for the 3-framed manifolds of primary interest, we can directly construct the invariants in question without appealing to the existence of a full-fledged field theory; in particular, none of the results in the remainder of Section~\ref{sec:radfordftc} depend on the cobordism hypothesis.

\subsection{A topological proof of the quadruple dual theorem} \label{sec:topquaddual}

For our purposes, the most important 3-framed 2-manifold is the Radford bordism, depicted as an immersed surface in Figure~\ref{fig:Radford_bordism}.  Recall that this is the unique genus-zero 3-framed bordism from the left adjoint $\ev^L$ of the evaluation bordism to the right adjoint $\ev^R$ of the evaluation bordism.  As described in Section~\ref{sec:Serre}, the inverse loop and loop bordisms are both obtained by composing the bordisms $\ev^L$ and $\ev^R$, respectively, with the evaluation bordism $\ev$.  As such, the Radford bordism may just as well be viewed as a bordism from the inverse loop bordism to the loop bordism.  

By Theorem~\ref{thm:serreisdoubledual}, the field theory $\cF_\cC$ takes the inverse loop and loop bordisms to the identity bimodule twisted, respectively by the double left and double right dual: $\cF_\cC(\cS^{-1}) = {}_{\langle \mathfrak{l l} \rangle \cC} \cC_\cC$ and $\cF_\cC(\cS) = {}_{\langle \mathfrak{r r} \rangle \cC} \cC_\cC$.  The field theory $\widetilde{\cF_\cC}$ therefore takes the Radford bordism to a bimodule isomorphism $\cR_\cC : {}_{\langle \mathfrak{l l} \rangle \cC} \cC_\cC \xra{\simeq} {}_{\langle \mathfrak{r r} \rangle \cC} \cC_\cC$, called the Radford equivalence.  As described in Lemma~\ref{lem:BimoduleToFunctor}, any such bimodule isomorphism provides a comparison between the two twisting functors, in this case between the double left and double right dual.  The quadruple dual theorem follows immediately:
\begin{theorem} \label{thm:quaddual}
Let $\cC$ be a finite tensor category.  There is a canonical tensor-invertible object $D \in \cC$ and a canonical monoidal natural isomorphism
\[
{}^{**}(-) \simeq D^{-1} \otimes (-)^{**} \otimes D.
\]
Equivalently, there is a canonical monoidal natural isomorphism
\[
D \otimes (-) \otimes D^{-1} \simeq (-)^{****}.
\] 
\end{theorem}
\nid (As described in the proof of Lemma~\ref{lem:BimoduleToFunctor}, the object $D$ is the image of $1 \in \cC$ under the Radford equivalence $\cR_\cC : {}_{\langle \mathfrak{l l} \rangle \cC} \cC_\cC \ra {}_{\langle \mathfrak{r r} \rangle \cC} \cC_\cC$, and the monoidal natural isomorphism in the theorem is simply given by the left $\cC$-module structure of the Radford equivalence.  The object $D$ is usually referred to simply as the ``distinguished invertible object".)

This result was proved for categories of modules over a Hopf algebra by Radford~\cite{MR0407069}, and then proved for finite tensor categories over an algebraically closed field by Etingof--Nikshych--Ostrik~\cite{MR2097289}.  Though in principle our result generalizes ENO's theorem (to perfect fields), the real purpose of discussing the result here is to provide a topological explanation of this a-priori highly algebraic result, by seeing it as an immediate corollary of the Dirac belt trick.

\subsection{A computation of the Radford equivalence} \label{sec:computeradford}

We know more than merely that the Radford bordism exists, though.  As in Figure~\ref{fig:Radford_bordism}, we have an explicit handle decomposition of the Radford bordism as the following composite:
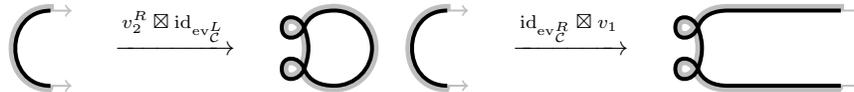
\[
\cb{
\begin{tikzpicture}
						\draw[linestyle,fuzzright] (4,7) to [looseness=1.6,out = 180, in = 180] (4,6);
						\begin{pgfonlayer}{background}
							\draw[->,outstyle] (4,7) -- +(0:\arrowlength);
							\draw[->,outstyle] (4,6) -- +(0:\arrowlength);
						\end{pgfonlayer}
\end{tikzpicture}
}
\quad \xra{v_2^R \:\boxtimes\: \id_{\ev_\cC^L}} \quad
\setlength{\linewid}{15pt}
\setlength{\fuzzwidth}{25pt}
\setlength{\arrowlength}{80pt}
\setlength{\arrowwidth}{7.5pt}
\cb{
\scalebox{.1}{
\begin{tikzpicture}
\draw[linestyle,fuzzright] (15,5) to [looseness=1.6,out=180,in=180] (15,-5);
\begin{pgfonlayer}{background}
	\draw[-scalehead,outstyle] (15,5) -- +(0:\arrowlength);
	\draw[-scalehead,outstyle] (15,-5) -- +(0:\arrowlength);
\end{pgfonlayer}
\draw[linestyle,fuzzright]
(0,5) to [out=180, in=70] (-4,2.75)
	to [looseness=1.6, out=-110, in=-90] (-7,2.75)
	to [looseness=1.6, out=90, in=110] (-4,2.75)
	to [out=-70, in=70] (-4,-2.75)
	to [looseness=1.6, out=-110, in=-90] (-7,-2.75)
	to [looseness=1.6, out=90, in=110] (-4,-2.75)
	to [out=-70, in=180] (0,-5) arc (-90:90:5cm);
\end{tikzpicture}
}
}
\setlength{\linewid}{1.5pt}
\setlength{\fuzzwidth}{2.5pt}
\setlength{\arrowlength}{8pt}
\setlength{\arrowwidth}{.75pt}
\quad \xra{\id_{\ev_\cC^R} \:\boxtimes\: v_1}	 \quad					
\setlength{\linewid}{15pt}
\setlength{\fuzzwidth}{25pt}
\setlength{\arrowlength}{80pt}
\setlength{\arrowwidth}{7.5pt}
\cb{
\scalebox{.1}{
\begin{tikzpicture}
\draw[linestyle,fuzzright]
(15,5) to (0,5) to [out=180, in=70] (-4,2.75)
	to [looseness=1.6, out=-110, in=-90] (-7,2.75)
	to [looseness=1.6, out=90, in=110] (-4,2.75)
	to [out=-70, in=70] (-4,-2.75)
	to [looseness=1.6, out=-110, in=-90] (-7,-2.75)
	to [looseness=1.6, out=90, in=110] (-4,-2.75)
	to [out=-70, in=180] (0,-5) to (15,-5);
\begin{pgfonlayer}{background}
	\draw[-scalehead,outstyle] (15,5) -- +(0:\arrowlength);
	\draw[-scalehead,outstyle] (15,-5) -- +(0:\arrowlength);
\end{pgfonlayer}
\end{tikzpicture}
}
}
\setlength{\linewid}{1.5pt}
\setlength{\fuzzwidth}{2.5pt}
\setlength{\arrowlength}{8pt}
\setlength{\arrowwidth}{.75pt}
\]

\nid By directly analyzing each handle, we can compute the Radford equivalence explicitly.  This computation is somewhat technical and for most readers can safely be skipped.  It is most convenient to express the answer in terms of categories of internal modules, and this will allow a direct comparison with the isomorphism constructed by ENO.

Recall from Section~\ref{sec:enrichedendo} that for a finite tensor category $\cC$, the evaluation bimodule ${}_{\cC \boxtimes \cC^\mp} \cC_\Vect$ is equivalent (as a bimodule) to the category $\Mod{}{A}(\cC \boxtimes \cC^\mp)$, where $A := \IHom_{\mod{\cC \boxtimes \cC^\mp}{}}(1_\cC,1_\cC) \in \cC \boxtimes \cC^\mp$ is the algebra object of enriched endomorphisms of the unit of the $\cC \boxtimes \cC^\mp$-module $\cC$.  For brevity, we abbreviate $\Mod{}{A}(\cC \boxtimes \cC^\mp)$ by $\Mod{}{A}$, and similarly for related categories of modules; that is, all categories of modules occurring in this section occur inside $\cC \boxtimes \cC^\mp$.  By Corollary~\ref{cor:dualamod}, we can express the adjoints to evaluation as follows: $\ev_\cC^L \simeq {}^*(\Mod{}{A}) \simeq \Mod{A}{}$ and $\ev_\cC^R \simeq (\Mod{}{A})^* \simeq \Mod{A^{**}}{}$.  Recall from Lemma~\ref{lem:dualing-amod} that the object $A^* \in \cC \boxtimes \cC^\mp$ naturally has the structure of an internal $A^{**}$--$A$-bimodule.
\begin{proposition} \label{prop:computeradford}
Let $\cC$ be a finite tensor category, and let $A := \linebreak \IHom_{\mod{\cC \boxtimes \cC^\mp}{}}(1_\cC,1_\cC) \in \cC \boxtimes \cC^\mp$ be the algebra object such that ${}_{\cC \boxtimes \cC^\mp} \cC_\Vect \simeq \Mod{}{A}(\cC \boxtimes \cC^\mp)$.  The Radford equivalence $\cR_\cC$ is given by the map
\begin{align*}
\Mod{A}{} &\ra \Mod{A^{**}}{}, \\
M &\mapsto A^* \otimes_A M.
\end{align*}
\end{proposition}
\nid ENO's equivalence $\Mod{A}{} \simeq \Mod{A^{**}}{}$ was also given by tensoring with the bimodule ${}_{A^{**}} {A^*}_A$; this proposition therefore establishes that our Radford equivalence agrees with ENO's equivalence, and in particular that our distinguished invertible object $D$ agrees with ENO's distinguished invertible object.  We will need these facts to utilize ENO's computation of the distinguished invertible object in the semisimple case and, later on, their characterization of the Radford equivalence via a quantum trace property.
\begin{proof}
By definition, the Radford equivalence is the composite $(\id_{\ev_\cC^R} \boxtimes v_1) \circ (v_2^R \boxtimes \id_{\ev_\cC^L}) : \ev_\cC^L \ra \ev_\cC^R$, where $v_1$ is the counit of the adjunction $\ev_\cC^L \dashv \ev_\cC$, and $v_2$ is the counit of the adjunction $\ev_\cC \dashv \ev_\cC^R$.  Our initial task, therefore, is to compute $v_1$ and $v_2$ explicitly in terms of categories of internal modules.

We first check that the counit map $v_1$ is given by
\begin{align*}
\Mod{}{A} \boxtimes \Mod{A}{} &\ra \cC \boxtimes \cC^\mp \\
M \boxtimes N &\mapsto M \otimes_A N.
\end{align*}
Combining in turn the equivalence in Corollary~\ref{cor:dualamod}, the equivalence in Proposition~\ref{prop:dual-formula-for-adjoints}, and the counit map from Proposition~\ref{prop:evcoev}, we see that the counit $v_1$ is the following composite:
\[
\begin{array}{rcl}
\Mod{}{A} \boxtimes \Mod{A}{} \!\!\!
&\xra{\simeq}& \!\!\! \Mod{}{A} \boxtimes {}^*(\Mod{}{A})   \hspace*{3in}
\\
M \boxtimes N \!\!\!
&\mapsto& \!\!\! M \boxtimes {}^*(N^*) 
\end{array}
\] \nopagebreak \vspace{-6pt}
\[
\begin{array}{rcl}
\hspace*{1.1in}
&\xra{\simeq}& \!\!\! \Mod{}{A} \boxtimes \Fun_{\Mod{\cC \boxtimes \cC^\mp}{}}(\Mod{}{A},\cC \boxtimes \cC^\mp) \xra{\text{eval}} \cC \boxtimes \cC^\mp
\\
&\mapsto& \!\!\! M \boxtimes (- \otimes N^*)^L
\end{array}
\] 
Because every $A$-module is a finite colimit of free $A$-modules, it suffices to check that this (right exact) composite functor agrees with $M \boxtimes N \mapsto M \otimes_A N$ on objects of the form $M \boxtimes A$.  Observe that the $\cC \boxtimes \cC^\mp$-module functor $\cC \boxtimes \cC^\mp \ra \Mod{}{A}$, $X \mapsto X \otimes A^*$ is right adjoint to the forgetful functor $\Mod{}{A} \ra \cC \boxtimes \cC^\mp$, $N \mapsto N$.  (The left adjoint to the forgetful functor is the usual induction functor $M \mapsto M \otimes A$.)  Now compute $v_1(M \boxtimes A)$ to be
\[
M \boxtimes A \mapsto M \boxtimes {}^*(A^*) \mapsto M \boxtimes (- \mapsto -) \mapsto M,
\]
as required.

We next check that the counit map $v_2$ is given by
\begin{align*}
\Mod{A^{**}}{} \boxtimes_{\cC \boxtimes \cC^\mp} \Mod{}{A} \simeq \Mod{A^{**}}{A} &\ra \Vect \\
M &\mapsto \Hom_{\Mod{A^{**}}{A}}(M,A^*)^\vee.
\end{align*}
By the same reasoning as in the previous paragraph, we know that the counit $v_2$ is the following composite:\pagebreak
\[
\begin{array}{rcl}
\Mod{A^{**}}{} \boxtimes \Mod{}{A} \!\!\!
&\xra{\simeq}& \!\!\! (\Mod{}{A})^* \boxtimes \Mod{}{A} \hspace*{3in}
\\
M \boxtimes N \!\!\!
&\mapsto& \!\!\! ({}^* M)^* \boxtimes N 
\end{array}
\] \nopagebreak \vspace{-6pt}
\[
\begin{array}{rcl}
\hspace*{1.9in}
&\xra{\simeq}& \!\!\! \Fun_\Vect(\Mod{}{A},\Vect) \boxtimes \Mod{}{A} 
\xra{\text{eval}} \Vect
\\
&\mapsto& \!\!\! ({}^* M \otimes -)^L \boxtimes N
\end{array}
\] 
Here all of the Deligne tensor products are implicitly over $\cC \boxtimes \cC^\mp$.  It suffices to check that this composite functor agrees with $M \boxtimes N \mapsto \Hom_{\Mod{A^{**}}{A}}(M \otimes N,A^*)^\vee$ on objects of the form $A^{**} \boxtimes N$.  Observe that the linear functor $\Vect \ra \Mod{}{A}$, $X \mapsto A^* \otimes X$ is right adjoint to the functor $\Mod{}{A} \ra \Vect$, $M \mapsto \Hom_{\Mod{}{A}}(M,A^*)^\vee$.  Now compute $v_2(A^{**} \boxtimes N)$ to be
\[
A^{**} \boxtimes N \mapsto (A^*)^* \boxtimes N \mapsto (- \mapsto \Hom_{\Mod{}{A}}(-,A^*)^\vee) \boxtimes N \mapsto \Hom_{\Mod{}{A}}(N,A^*)^\vee,
\]
and note that $\Hom_{\Mod{}{A}}(N,A^*)^\vee \cong \Hom_{\Mod{A^{**}}{A}}(A^{**} \otimes N,A^*)^\vee$.

The right adjoint of the map $\Mod{A^{**}}{A} \ra \Vect$, $M \mapsto \Hom_{\Mod{A^{**}}{A}}(M,A^*)^\vee$ is the map $\Vect \ra \Mod{A^{**}}{A}$, $k \mapsto A^*$.  The Radford equivalence is therefore the composite
\[
M \mapsto A^* \boxtimes M \mapsto A^* \otimes_A N,
\]
as desired.
\end{proof}

Provided we are willing to make strong semisimplicity assumptions about our tensor category, the statement of Theorem~\ref{thm:quaddual} simplifies.  An object of a semisimple finite linear category is called absolutely simple if it is simple and remains so after arbitrary base changes. \begin{corollary} \label{cor:ssquaddual}
Let $\cC$ be a semisimple finite tensor category whose unit $1 \in \cC$ is absolutely simple. There is a canonical monoidal natural equivalence between the identity functor and the quadruple dual functor:
\[
(-) \simeq (-)^{****}.
\]
\end{corollary} 
\begin{proof}
When $\cC$ is a semisimple finite tensor category with simple unit, over an algebraically closed field, Etingof--Nikshych--Ostrik prove in~\cite[Cor 6.4]{MR2097289} that the distinguished invertible object $D$ is isomorphic to the unit $1$.  Note, for instance using the semisimplicity of $\cC$, that the property of two objects being isomorphic is invariant under base change.  The result now follows from Theorem~\ref{thm:quaddual}, because the functor $D^{-1} \otimes (-)^{**} \otimes D$ is canonically isomorphic to $(-)^{**}$.
\end{proof}

When $\cC$ is semisimple, it is possible to directly calculate the image of $1 \in \cC$ under the Radford equivalence $\ev_\cC^L \ra \ev_\cC^R$, without reexpressing the module categories as categories of internal modules; such a calculation provides a proof of Corollary~\ref{cor:ssquaddual} that does not rely on Proposition~\ref{prop:computeradford}.

\begin{remark}
We expect the absolute simplicity assumption in Corollary~\ref{cor:ssquaddual} can be removed, as follows.  It suffices to check that for any semisimple finite tensor category $\cC$ over an algebraically closed field, the distinguished invertible object $D$ is trivial.  As in Remark~\ref{rem:nonsimp}, the simple decomposition $\oplus \, 1_i$ of the unit of $\cC$ provides a decomposition of $\cC$ as $\oplus \, \cC_{ij}$.  The result of ENO mentioned in the above proof shows that the distinguished invertible object $D_i$ of each fusion category $\cC_{ii}$ is trivial.  One is left to check only that the distinguished invertible object of $\cC$ is the sum $\oplus_i \, D_i$ of the distinguished objects of the diagonal subtensor categories.
\end{remark}

\section[Adjoints of bimodule functors: sep. tensor cats are dualizable]{\for{toc}{Adjoints of bimodule functors: separable tensor cats are dualizable}\except{toc}{Adjoints of bimodule functors: separable tensor categories are fully dualizable}} \label{sec:separableisfd}

In this section we prove our last two main theorems, which together determine the fully dualizable finite tensor categories: separable tensor categories are fully dualizable, and any fully dualizable finite tensor category is separable.  We also identify the maximal fully dualizable sub-3-category of the 3-category of tensor categories, and thereby obtain a classification of 3-dimensional local framed field theories with defects.  (The reader who is content to restrict attention to characteristic zero can, by consulting Corollaries~\ref{cor:charzerosep} and~\ref{cor:charzeromodulesep}, simply replace the word `separable' by `finite semisimple' in the statements of this section.)

As we have already established the 2-dualizability of all finite tensor categories, it remains only to investigate when bimodule functors have adjoints.
\begin{proposition} \label{prop:bimodfunctadjoints}
Let ${}_\cC \cM_\cD$ and ${}_\cC \cN_\cD$ be finite semisimple bimodule categories between finite tensor categories.  Any bimodule functor $\cM \ra \cN$ admits both left and right adjoint bimodule functors.
\end{proposition}
\begin{proof}
A functor out of a semisimple category is both left and right exact.  Any exact linear functor between finite linear categories admits both a left and a right adjoint~\cite{BTP}; because $\cN$ is also semisimple, those adjoints are themselves exact, in particular right exact.  By Lemma~\ref{lma:module-adjoint}, those adjoints can be promoted to adjoints as bimodule functors.
\end{proof}

\begin{remark}
In fact, the left and right adjoints of a bimodule functor $\cF: \cM \ra \cN$, between finite semisimple bimodule categories, are isomorphic, and can be described explicitly as a transpose-dual functor.  That is, let $\{m_i\}$ be a set of representatives of the isomorphism classes of simple objects of $\cM$, and similarly $\{n_i\}$ for $\cN$.  The functor $\cF$ is isomorphic to a functor $m_i \mapsto \bigoplus_j F_{ij} n_j$, where $F_{ij} \in \Vect$.  The transpose-dual functor $n_j \mapsto \bigoplus_i F_{ij}^* m_i$ is both a left and a right adjoint for $\cF$, where $F_{ij}^* \in \Vect$ is the ordinary dual vector space to $F_{ij}$.
\end{remark}

With adjoints for bimodule functors at hand, we can now assemble the pieces to prove full dualizability.
\begin{theorem} \label{thm:TC-dualizable}
The 3-category $\TCsep$ (of separable tensor categories, finite semisimple bimodule categories, bimodule functors, and bimodule transformations) is fully dualizable.
\end{theorem}
\begin{proof}
By Proposition~\ref{prop:onedual}, a separable tensor category $\cC$ has $\cC^\mp$ as a dual tensor category; the tensor category $\cC^\mp$ is evidently also separable, and the unit and counit of the duality are semisimple because $\cC$ itself is semisimple.  Combining Proposition~\ref{prop:evcoev} and Proposition~\ref{prop:dual-formula-for-adjoints}, a finite semisimple bimodule category ${}_\cC \cM_\cD$ has the finite semisimple bimodule categories ${}^* \cM$ and $\cM^*$ as left and right adjoints.  Finally, Proposition~\ref{prop:bimodfunctadjoints} ensures that all bimodule functors between finite semisimple bimodule categories have adjoints.
\end{proof}

\nid Note that Schaumann has independently investigated the duality properties (though not explicitly the dualizability in a cobordism-hypothesis sense) of a 3-category of spherical fusion categories over an algebraically closed field of characteristic zero~\cite{Schaumann-PhD}.

\begin{corollary} \label{cor:septcisdualizable}
Separable tensor categories are 3-, that is fully, dualizable objects of the 3-category of finite tensor categories.
\end{corollary}

\begin{remark}
Though this corollary is a trivial consequence of the theorem, and the theorem is a straightforward application of the earlier constructions of duals and adjoints, this presentation conceals a substantial aspect of the work involved in proving that separable tensor categories are fully dualizable.  It is essential that the collection of tensor categories, bimodules, and functors that we can prove have duals and adjoints actually form a sub-3-category of the larger category $\TC$; that fact, stated in Corollary~\ref{cor:tcsepexists}, depends at root on Theorem~\ref{thm:compositeOfSep}, that separable bimodule categories compose.  In particular, considering Theorem~\ref{thm:TC-dualizable} over a field of finite characteristic, we cannot replace ``separable tensor categories" by ``finite semisimple tensor categories": the composite of two semisimple bimodules over a semisimple tensor category need not be semisimple.
\end{remark}

\begin{corollary} \label{cor:3dtft}
For each separable tensor category $\cC \in \TC$, there is a unique (up to equivalence) 3-framed 3-dimensional field theory
\[
\cF_\cC : \FrBord_3 \ra \TC
\]
whose value on the standard positively 3-framed point is $\cC$.
\end{corollary}

\begin{remark}
Note well that this result requires no assumption about the algebraic closure or characteristic of the base field.  In particular, observe the following: given a separable tensor category $\cC$ over a field $l$, if there exists a separable form of $\cC$ over a subfield $k \subset l$ (that is, there is a separable tensor category over $k$ whose base change to $l$ is equivalent to $\cC$), then the 3-manifold invariants of $\cF_\cC$ are all contained in the smaller field $k$.
\end{remark}

\begin{corollary} \label{cor:charzerotft}
There is a 3-framed 3-dimensional field theory associated to any finite semisimple tensor category over a field of characteristic zero.
\end{corollary}
\begin{corollary} \label{cor:fusiontft}
There is a 3-framed 3-dimensional field theory associated to any fusion category of nonzero global dimension over an algebraically closed field.
\end{corollary}

In fact, when restricting attention to finite tensor categories, the separability condition is also necessary for full dualizability:
\begin{theorem} \label{thm:converse}
Fully dualizable finite tensor categories are separable.
\end{theorem}
\begin{proof}
By Corollary~\ref{cor:Sep=semisimplecenter}, it suffices to show that $\cC$ is semisimple and that $\cZ(\cC)$ is semisimple.  In fact, the semisimplicity of $\cZ(\cC)$ will imply the semisimplicity of $\cC$, as follows.  By part (6) of Theorem~\ref{Thm:ExactModCatOmnibus}, that is~\cite[Lemma 3.25]{EO-ftc}, if ${}_\cC \cM$ is an exact module category over a finite tensor category $\cC$, then ${}_{\Fun_\cC(\cM,\cM)^\mp} \cM$ is an exact module category over the commutant $\Fun_\cC(\cM,\cM)^\mp$; in particular if ${}_{\cC \boxtimes \cC^\mp} \cC$ is exact, then ${}_{\cZ(\cC)} \cC$ is exact.  But indeed ${}_{\cC \boxtimes \cC^\mp} \cC$ is exact by Example~\ref{ex:exactness}.  Finally, by Example~\ref{eg:semiexact}, an exact module category over a semisimple category is semisimple.

A finite tensor category $\cC \in \TC$ is always 2-dualizable, and the center $\cZ(\cC)$ is the composite of elementary dualization data, as in Table~\ref{table:circles}.  If $\cC \in \TC$ is fully dualizable, then $\cZ(\cC) \in \End_\TC(\Vect)$ is fully dualizable as an object of the 2-category $\End_\TC(\Vect)$ of endomorphisms of the unit $\Vect \in \TC$.  That 2-category has as objects the finite linear categories, as morphisms right exact functors, and as 2-morphisms natural transformations; it is therefore equivalent (by taking an algebra to its module category) to the 2-category $\Alg$ of finite-dimensional algebras, bimodules, and intertwiners.  Let $B$ be a finite-dimensional algebra with $\Mod{B}{} \simeq \cZ(\cC)$.  The full dualizability of $\cZ(\cC) \in \End_\TC(\Vect)$ is equivalent to the full dualizability of $B \in \Alg$.  As mentioned at the beginning of Section~\ref{sec:tc-separable}, a finite-dimensional algebra is fully dualizable precisely when it is separable.  Over a perfect field, as mentioned at the beginning of Section~\ref{sec:sepandsemi}, a finite-dimensional algebra is separable if and only if it is semisimple.  Thus the algebra $B$ is semisimple, and so $\Mod{B}{} \simeq \cZ(\cC)$ is semisimple, as required.
\end{proof}

\begin{corollary} \label{cor:maxfd}
The 3-category $\TCsep$ is the maximal fully dualizable sub-3-category of $\TC$.
\end{corollary}
\begin{proof}
As in the Appendix, let $d\TC$ denote the maximal fully dualizable subcategory of $\TC$.  By the theorem, any object of $d\TC$ is separable.  Let ${}_\cC \cM_\cD$ be a morphism of $d\TC$.  We will show that such a morphism $\cM$ must be a semisimple bimodule category as follows: we first check that any (right exact) $\cC$-module functor from $\cM$ to $\cC$ is in fact exact; we then apply this property to an appropriate internal Hom functor and conclude that all objects of $\cM$ are injective.

By Proposition~\ref{prop:evcoev}, the left adjoint of $\cM$ is $\Fun_\cC(\cM,\cC)$; the counit of this adjunction is the (right exact) evaluation $\cC$--$\cC$-bimodule functor $\varepsilon: \cM \boxtimes_\cD \Fun_\cC(\cM,\cC) \ra \cC$.  Because the functor $\varepsilon$ has a left adjoint, it is also left exact.  The functor $\cM \times \Fun_\cC(\cM,\cC) \ra \cM \boxtimes_\cD \Fun_\cC(\cM,\cC)$ is exact in each variable, by part (4) of Theorem~\ref{thm:DelignePrdtOverATCExists}.  Thus for any (right exact) functor $\cF \in \Fun_\cC(\cM,\cC)$, the composite $\cM \xra{\id \times \cF} \cM \times \Fun_\cC(\cM,\cC) \ra \cM \boxtimes_\cD \Fun_\cC(\cM,\cC) \xra{\varepsilon} \cC$ is exact; that composite just is the functor $\cF$.  

For any object $m \in \cM$, the functor ${}^* \IHom(-,m) : \cM \ra \cC$ is a (right exact) $\cC$-module functor, therefore is exact; thus $\IHom(-,m) : \cM \ra \cC$ itself is an exact contravariant functor.  For any injective object $j \in \cC$, the contravariant functor $\Hom(-,j \otimes m) \cong \Hom(j^* \otimes -,m) \cong \Hom(j^*,\IHom(-,m))$ is therefore exact, so $j \otimes m$ is injective as well; cf.~\cite[Prop. 3.16]{EO-ftc}.  As $\cC$ is semisimple, its unit is injective, so all objects of $\cM$ are injective, thus $\cM$ is semisimple, as required.
\end{proof}

\begin{remark}
This corollary provides a classification of surface, line, and point defects in 3-dimensional topological field theories of Turaev--Viro type.  It supplies a field theory taking values on 3-framed 3-manifolds with any arrangement of codimension-0, -1, -2, and -3 regions labelled, respectively, by separable tensor categories, semisimple bimodule categories, bimodule functors, and bimodule transformations.  Moreover, every local framed defect theory with target tensor categories is of this form.  These results dovetail with existing literature on defects in 3-dimensional topological field theory~\cite{kapustinsaulina,kitaevkong,fsv} and in 2-dimensional conformal field theory~\cite{ffrs-duality,frs-fusion}.
\end{remark}

\section{Spherical structures and structured field theories} \label{sec:spherical}

By Corollary~\ref{cor:fusiontft}, there is in particular a 3-dimensional 3-framed field theory $\cF_\cC$ associated to any fusion category $\cC\in \TC$ over an algebraically closed field of characteristic zero.  This theory differs from a Turaev--Viro field theory, which is an oriented field theory associated to a spherical fusion category, in two reciprocal ways: it does not depend on a choice of any additional structure on the fusion category (namely a spherical structure), but it does depend on a choice of an additional structure on the manifolds (namely a 3-framing).  

By the cobordism hypothesis, the group $SO(3)$ acts on the space of fully dualizable tensor categories, and the field theory $\cF_\cC : \FrBord_3 \ra \TC$ descends from a 3-framed to an oriented field theory $\widetilde{\cF_\cC} : \OrBord_3 \ra \TC$ when $\cC$ is equipped with the structure of an $SO(3)$ homotopy fixed point.  This situation suggests that the choice of a spherical structure should be related to the existence of homotopy fixed point structures.  

In this section, we recall the classical notions of a pivotal structure on a tensor category, and of a spherical structure on a semisimple tensor category.  We introduce a new definition of sphericality that provides the correct generalization of the classical notion to the non-semisimple case.  We explain, largely in the form of a preview, the relationship between pivotal and spherical structures and the existence of certain homotopy fixed point structures, and finally conclude with a number of further conjectures and questions regarding the descent properties of 3-dimensional field theories.

\subsection{Pivotal structures and trivializations of the Serre automorphism}

By Theorem~\ref{thm:TC_is_2Dualizable}, the 3-category $\TC$ of finite tensor categories is 2-dualizable; by the cobordism hypothesis, it follows that there is an $SO(2)$ action on (the maximal sub-3-groupoid of) $\TC$.  That $SO(2)$ action is simple to describe on objects: the 1-cell of $SO(2)$ takes an object $\cC \in \TC$ to the Serre automorphism $\cS_\cC$ of $\cC$.  One might imagine that trivializing the Serre automorphism provides $\cC$ with the structure of an $SO(2)$ homotopy fixed point, but this is not the case: because $BSO(2)$ has infinitely many homology groups, in general an $SO(2)$ homotopy fixed point structure requires a higher coherence condition.  Instead, a trivialization of the Serre automorphism of $\cC$ provides $\cC$ with the structure of a homotopy fixed point for the simpler group $\Omega S^2$.  (Here $\Omega S^2$ acts on $\TC$ via the map $\Omega S^2 \ra \Omega B SO(2) \simeq SO(2)$, and it is simpler in the sense that its classifying space has homology completely concentrated in degree 2.  In fact, the action of $\Omega S^2$ on $\TC$ is completely determined by the Serre automorphism, and so its existence does not depend on the cobordism hypothesis.)

By Theorem~\ref{thm:serreisdoubledual}, the Serre automorphism of a finite tensor category $\cC$ is the bimodule ${}_{\langle \mathfrak{r r} \rangle \cC} \cC_\cC$.  A trivialization of the Serre automorphism is therefore a bimodule equivalence ${}_\cC \cC_\cC \xra{\simeq} {}_{\langle \mathfrak{r r} \rangle \cC} \cC_\cC$.  In general, by Lemma~\ref{lem:BimoduleToFunctor}, such an equivalence is provided by a monoidal natural isomorphism from the identity functor $\id_\cC : \cC \ra \cC$ to the conjugation of the right double dual functor $E^{-1} \otimes (-)^{**} \otimes E: \cC \ra \cC$, for some invertible object $E \in \cC$.  We may focus attention on the special class of such equivalences arising from the case when the invertible object $E$ is trivial---such equivalences are known as pivotal structures:
\begin{definition}
A \emph{pivotal structure} on a tensor category $\cC$ is a choice of monoidal natural isomorphism $P$ from the identity monoidal functor to the right double dual monoidal functor $(-)^{**}: \cC \ra \cC$.
\end{definition}
\nid In particular, we see that a finite tensor category equipped with a pivotal structure has canonically the structure of an $\Omega S^2$ homotopy fixed point.  It is certainly not the case that all $\Omega S^2$ fixed point structures come from pivotal structures: for instance, the tensor category $\Vect[\ZZ/2]$ has four $\Omega S^2$ fixed point structures, but only two pivotal structures.

\subsection{Spherical structures as square roots of the Radford equivalence}

As described in Section~\ref{sec:topquaddual}, for any finite tensor category $\cC$, there is a Radford equivalence $\cR_\cC : {}_{\langle \mathfrak{l l} \rangle \cC} \cC_\cC \xra{\simeq} {}_{\langle \mathfrak{r r} \rangle \cC} \cC_\cC$ from the inverse Serre automorphism to the Serre automorphism; by composing with the Serre automorphism we may equally well think of the Radford equivalence as an equivalence from the identity to the square of the Serre automorphism.  The existence of such an equivalence provides an extension of the $\Omega S^2$ action on $\TC$ to an action of $\Omega \Sigma \RP^2$.  

Given a trivialization $T: {}_\cC \cC_\cC \xra{\simeq} {}_{\langle \mathfrak{r r} \rangle \cC} \cC_\cC$ of the Serre automorphism, and therefore an $\Omega S^2$ fixed point structure on $\cC$, we may ask what additional data is required to promote $\cC$ to having an $\Omega \Sigma \RP^2$ fixed point structure.  It turns out that the required data is an isomorphism from the square of the trivialization $T$ to the Radford equivalence, that is an isomorphism $T \boxtimes_\cC T \ra \cR_\cC$.  When the trivialization $T$ came from a pivotal structure on $\cC$, it makes sense to ask that the square of $T$ be simply equal to the Radford---this condition is called sphericality:
\begin{definition} \label{def:spherical}
A pivotal finite tensor category $(\cC, P: (-) \simeq (-)^{**})$, with absolutely simple unit, is \emph{spherical} if the distinguished invertible object $D_\cC$ is isomorphic to the unit $1_\cC$ and the square of the pivotal structure $P^2: (-) \simeq (-)^{****}$ is equal to the Radford equivalence $\cR_\cC: (-) \simeq (-)^{****}$ canonically associated to $\cC$.
\end{definition}
\nid We will see shortly that when the tensor category $\cC$ is semisimple, this definition agrees with classical notions of sphericality; when $\cC$ is not semisimple, this definition, not the classical notion, is the correct generalization.  (In this definition the absolute simplicity requirement is present to ensure that the functor $D^{-1} \otimes (-)^{**} \otimes D$ is canonically equivalent to $(-)^{**}$ and therefore that the Radford indeed provides a trivialization of the quadruple dual functor.  Also, by the square of the pivotal structure, we mean the natural isomorphism whose value on an object $V$ is the composite $P(V^{**}) \circ P(V)$.)  

To reiterate, a spherical tensor category has canonically the structure of an $\Omega \Sigma \RP^2$ fixed point.  Not all $\Omega \Sigma \RP^2$ fixed point structures come from spherical structures: evidently, any fixed point structure for a tensor category with a nontrivial distinguished invertible object cannot arise from a spherical structure.  An example of a category with nontrivial distinguished object but with an $\Omega \Sigma \RP^2$ fixed point structure is provided by the representation category of the Taft Hopf algebra $H_3 := \CC\langle x,g \rangle / (x^3=0,g^3=1,gx=\zeta xg), \Delta(g) = g \otimes g, \Delta(x) = x \otimes 1 + g \otimes x, S(g)=g^{-1}, S(x)=-g^{-1}x,$ where $\zeta$ is a third root of unity. 

Recall that a trivialization $T: {}_\cC \cC_\cC \xra{\simeq} {}_{\langle \mathfrak{r r} \rangle \cC} \cC_\cC$ of the Serre automorphism is precisely determined by an invertible object $E \in \cC$ together with a monoidal trivialization of the functor $E^{-1} \otimes (-)^{**} \otimes E$.  An isomorphism $T \boxtimes_\cC T \cong \cR_\cC$ is precisely determined by an isomorphism $E^2 \cong D$, from the square of $E$ to the distinguished invertible object $D$, such that the composite isomorphism ${}^{**} (-) \xra{\cong} E^{-1} \otimes (-) \otimes E \xra{\cong} E^{-2} \otimes (-)^{**} \otimes E^2 \xra{\cong} D^{-1} \otimes (-)^{**} \otimes D$ is the Radford equivalence.  It follows, for instance, that a tensor category for which the distinguished invertible object is not a tensor square cannot admit any $\Omega \Sigma \RP^2$ fixed point structures (and so also cannot admit a spherical structure).

\subsection{Semisimple sphericality as a trace condition}

In order to be able to clearly delineate it from Definition~\ref{def:spherical}, we will refer to the classical notion as `trace spherical':
\begin{definition} \label{def:trspherical}
A pivotal tensor category $(\cC, P: (-) \simeq (-)^{**})$ is \emph{trace spherical} if for all objects $V \in \cC$, the quantum trace of the pivotal structure on $V$ agrees with the quantum trace of the pivotal structure on ${}^* V$, that is $\Tr(P_V) = \Tr(P_{({}^* V)})$.
\end{definition} 
\nid In the semisimple case, the two notions correspond precisely:
\begin{proposition}
A semisimple pivotal finite tensor category with absolutely simple unit is spherical if and only if it is trace spherical.
\end{proposition} 
\begin{proof}
Both the property of sphericality and the property of trace sphericality are invariant under base change, so it suffices to compare them over an algebraically closed field.  When the base field is algebraically closed, Etingof-Nikshych-Ostrik prove in \cite[Thm 7.3 and Cor 7.4]{MR2097289} that the Radford equivalence $\cR_\cC : (-) \xra{\simeq} (-)^{****}$ is characterized by a property that is equivalent to the following property: for any object $V \in \cC$ and any morphism $f_V: V \ra V^{**}$, there is an equality of traces $\Tr(f_V) = \Tr((f_V)^* \circ \cR_{({}^* V)})$.  Here we have abbreviated $\cR_\cC({}^* V)$ as $\cR_{({}^* V)}$.  When $(\cC,P)$ is a pivotal tensor category, we have $(P_V)^* = (P_{V^*})^{-1}$ 
\cite[Lemma 4.11; SR72, Prop 5.2.3]{0908.3347}. \nocite{rivano}  
If $(\cC,P)$ is spherical, it follows that
\[
\Tr(P_V) = \Tr((P_V)^* \circ \cR_{({}^* V)}) 
= \Tr(P_{V^*}^{-1} \circ P_{V^*} \circ P_{({}^* V)}) = \Tr(P_{({}^* V)});
\]
that is, $(\cC,P)$ is trace spherical.  The converse follows by cyclically permuting this equation.
\end{proof}

In the setting of non-semisimple finite tensor categories, it is neither the case that spherical implies trace spherical, nor the case that trace spherical implies spherical, as the following two examples illustrate.

\vspace{5pt}
\begin{example}
Let $q$ be a primitive $p$-th root of unity.  Consider the Hopf algebra described in~\cite[Rem. 7.5]{MR2097289}: $H := \CC\langle E, F, K \rangle / (KE = q^2 EK, KF = q^{-2} FK, EF=FE, E^p = 0, F^p = 0, K^p =1)$, $\Delta(K) = K \otimes K$, $\Delta(E) = E \otimes K + 1 \otimes E$, $\Delta(F) = F \otimes 1 + K^{-1} \otimes F$, $S(K) = K^{-1}$, $S(E) = -EK^{-1}$, $S(F) = -KF$---this differs from the quantum group $U_q(sl_2)$ only in that $E$ and $F$ commute.  Observe that the natural transformation $V \ra V^{**}$, $v \mapsto K \cdot v$ is a pivotal structure on $\Rep(H)$.  One can check that the Radford equivalence $V \ra V^{****}$ is given by $v \mapsto K^2 \cdot v$, and therefore this pivotal structure is spherical.  

Given any Hopf algebra $G$ with an object $g$ such that $\Delta(g) = g \otimes g$, $S(g) = g^{-1}$, and $S^2$ equal to conjugation by $g$, we have that left multiplication $l_g$ by $g$ is a pivotal structure on $\Rep(G)$.  Barrett--Westbury show that this pivotal structure is trace spherical if and only if $\tr_V(l_g \theta) = \tr_V(l_{g^{-1}} \theta)$ for all representations $V$ and all endomorphisms $\theta \in \Hom_{\Rep(G)}(V,V)$ \cite{MR1686423}; note that here $\tr_V$ is the trace of a vector space endomorphism of $V$, not the quantum trace.  The Hopf algebra $H$ above has 1-dimensional representations where this trace equation certainly fails, and therefore the aforementioned pivotal structure on $\Rep(H)$ is not trace spherical.
\end{example}

\vspace{5pt}
\begin{example}
Consider Sweedler's 4-dimensional Hopf algebra $H := \CC \langle x, g \rangle /\allowbreak (x^2 = 0, g^2 = 1, gx = - xg)$, $\Delta(g) = g \otimes g$, $\Delta(x) = 1 \otimes x + x \otimes g$, $S(g) = g^{-1}$, and $S(x) = - xg^{-1}$.  This Hopf algebra has four indecomposable representations: there are two irreducible 1-dimensional representations $V_\pm$ (with actions $v \xmapsto{x} 0$ and $v \xmapsto{g} \pm v$), and there are two 2-dimensional projective representations $W_\pm = \CC\{w_0, w_1\}$ (with actions $w_1 \xmapsto{x} w_0 \xmapsto{x} 0$, $w_1 \xmapsto{g} \pm w_1$, and $w_0 \xmapsto{g} \minusplus w_0$).

Note that by the Barrett--Westbury condition mentioned in the previous example, we immediately see that the category of representations $\Rep(H)$ of Sweedler's Hopf algebra, with the pivotal structure given by left multiplication by $g$, is trace spherical.  Recall from~\cite{MR2097289} that in a finite tensor category over an algebraically closed field, the projective cover of the unit is dual to the projective cover of the dual of the distinguished invertible object.  The unit of $\Rep(H)$ is $V_+$; the projective cover of $V_+$ is $W_+$; the dual of $W_+$ is $W_-$; the projective cover of $V_-$ is $W_-$; and the dual of $V_-$ is $V_-$ itself.  The distinguished invertible object is therefore $V_-$, but note that $V_+$ and $V_-$ both square to the unit, so there is certainly no tensor square root of the distinguished invertible.  Thus, this representation category cannot have an $\Omega \Sigma \RP^2$ fixed point structure, and so in particular cannot have a spherical structure.
\end{example}

\begin{remark}
Because the representation category $\Rep(H) \in \TC$ of Sweed\-ler's Hopf algebra is not $\Omega \Sigma \RP^2$-fixed, it cannot be $SO(3)$ fixed for any action of $SO(3)$ on $\TC$ restricting to the $\Omega \Sigma \RP^2$ action provided by the combination of the Serre automorphism and the Radford equivalence.  Thus the associated field theory invariants cannot possibly be oriented invariants.  At best, $\Rep(H)$ could be an $SO(2)$-fixed point and thus provide a combed field theory.  (A combed field theory is by definition one taking values on $SO(2)$-structured manifolds; an $SO(2)$-structured 3-manifold is precisely a 3-manifold equipped with a nonvanishing vector field, that is a combing.)  This explains the necessity of the combing on the manifolds in Kuperberg's invariant~\cite{MR1394749}.  In particular, we see that, unlike for ordinary Turaev--Viro invariants in the semisimple case, trace sphericality is in general insufficient to produce an oriented invariant.
\end{remark}

\subsection{Oriented, combed, and spin field theory descent conjectures} \label{sec:descconj}

For definiteness, in this last section we presume the base field is algebraically closed of characteristic zero.  It is a well-known open problem to determine whether all fusion categories admit pivotal structures.  As we saw, the existence of a pivotal structure provides an $\Omega S^2$ fixed point structure, so one might ask directly whether all fusion categories admit an $\Omega S^2$ fixed point structure.  Somewhat stronger and more interesting is to ask about $SO(2)$ fixed point structures:
\begin{question}
Does every local 3-framed field theory with target tensor categories descend to a combed field theory?
\end{question}
\nid We might ask the yet stronger:
\begin{question}
Does every local 3-framed field theory with target tensor categories descend to an oriented field theory?
\end{question}

Instead of asking about the descent properties of an arbitrary field theory, we may equip our tensor category itself with extra structure and expect that structure to provide descent information.  Thus less speculatively we have the following trio of anomaly-vanishing conjectures:
\begin{conjecture}
Every spherical fusion category admits the structure of an $SO(3)$ homotopy fixed point, and therefore provides an oriented local field theory.
\end{conjecture}
\begin{conjecture}
Every pivotal fusion category admits the structure of an $SO(2)$ homotopy fixed point, and therefore provides a combed local field theory.
\end{conjecture}
\begin{conjecture}
Every fusion category is a $Spin(3)$ homotopy fixed point, and therefore provides a spin local field theory.
\end{conjecture}

One can formulate similar questions and conjectures for tensor categories that are merely finite rather than fusion.  For instance, we saw that any finite tensor category provides a partial (``non-compact") 3-dimensional 3-framed field theory.  If the finite tensor category is pivotal, we also observed that it is an $\Omega S^2$ fixed point.  We might expect it to be in fact an $SO(2)$ fixed point, and might anticipate that it moreover provides a non-compact 3-dimensional combed 3-framed field theory.  Even if the finite tensor category is not pivotal (there are such examples, see \cite[Rem. 2.11]{1204.5807}), we might still ask if it is necessarily $SO(2)$ fixed.

\appendix
\chapter{The cobordism hypothesis}

The cobordism hypothesis, conjectured by Baez and Dolan~\cite{MR1355899} and later proven by Hopkins and Lurie~\cite{lurie-ch}, provides a classification of local topological field theories.  In the tangentially framed case, this classification is particularly simple to state: the local tangentially framed $n$-dimensional topological field theories with target the symmetric monoidal $(\infty,n)$-category $\cC$ correspond to the fully dualizable objects of $\cC$.  

Here, a tangentially framed $n$-manifold is an $n$-manifold $M$ with a trivialization of its tangent bundle $\tau_M$, and a tangentially ($n$-)framed $k$-manifold, for $k < n$, is a $k$-manifold $N$ with a trivialization of the stabilized-up-to-dimension-$n$ tangent bundle $\tau_N \oplus \RR^{n-k}$; this structure is also referred to simply as an $n$-framing.  There is an $(\infty,n)$-category $\FrBord_n$ with objects being $n$-framed $0$-manifolds, morphisms being $n$-framed bordisms between those objects, 2-morphisms being $n$-framed bordisms between those 1-morphisms, and so on, up to the $n$-morphisms, which are the spaces of $n$-framed $n$-manifold bordisms between the $(n-1)$-morphisms.  A local tangentially framed $n$-dimensional topological field theory is by definition a symmetric monoidal functor from $\FrBord_n$ to a target symmetric monoidal $(\infty,n)$-category $\cC$.

A succinct and elementary definition of full dualizability can be stated in terms of the existence of adjoints in a collection of 2-categories extracted from the target category $\cC$.  For each $k$, with $-1 \leq k \leq n-2$, there is a functorial association
\[
\cC \mapsto h_2^{(k)} \cC
\]
taking a symmetric monoidal $(\infty,n)$-category $\cC$ to the 2-category (that is $(2,2)$-category) of $k$-morphisms of $\cC$.  More specifically, for $0 \leq k$, the 2-category $h_2^{(k)} \cC$ has objects the $k$-morphisms of $\cC$, 1-morphisms the $(k+1)$-morphisms of $\cC$, and 2-morphisms the equivalence classes of $(k+2)$-morphisms of $\cC$.  For $k=-1$, the 2-category $h_2^{(-1)} \cC$ has a single object, has 1-morphisms the objects of $\cC$ with composition given by the monoidal product of $\cC$, and has 2-morphisms the equivalence classes of 1-morphisms of $\cC$.

\begin{appexample}
	In the symmetric monoidal 3-category $\TC$ of tensor categories, the 2-categories $h^{(-1)}_2 \TC$, $h^{(0)}_2 \TC$, and $h^{(1)}_2 \TC$ are respectively (-1) the monoidal category of tensor categories and equivalence classes of bimodule categories, (0) the 2-category of tensor categories, bimodule categories, and natural isomorphism classes of bimodule functors, and (1) the 2-category of bimodule categories, bimodule functors, and natural transformations.
\end{appexample}

We can now recall the notion of full dualizability.  The notion of adjoint functors, given in Definition~\ref{def:Adjoints}, can be transported verbatim into any 2-category:

\begin{appdefinition} \label{def:adjoints_in_bicat}
		A 1-morphism $G: \cA \to \cB$ in a 2-category {\em admits a left adjoint} $F: \cB \to \cA$, or equivalently a 1-morphism $F: \cB \to \cA$ {\em admits a right adjoint} $G: \cA \to \cB$, if there are 2-morphisms, the {\em unit} $\eta: id_{\cB} \to G \circ F$ and the {\em counit} $\varepsilon: F \circ G \to id_{\cA}$, satisfying the following pair of equations:
		\begin{align*}
			(id_{G} \circledcirc \varepsilon  ) \circ (  \eta \circledcirc id_{G}) &= id_{G}, \\
			(\varepsilon \circledcirc id_{F}) \circ (id_{F} \circledcirc \eta) &= id_{F}.
		\end{align*}
	Here $\circledcirc$ denotes the horizontal composite of 2-morphisms.
	We say $F$ is the {\em left adjoint} of $G$ and $G$ is the {\em right adjoint} of $F$, and we denote this situation $F \dashv G$.
\end{appdefinition}

\begin{appdefinition}
	Let $\cC$ be a symmetric monoidal $(\infty,n)$-category. We say that $\cC$ {\em has adjoints for $k$-morphisms} if $h^{(k-1)}_2 \cC$ has both left and right adjoints for all 1-morphisms. 
\end{appdefinition}

\noindent Note that when $k=0$, this property is also referred to as ``having duals for objects", because the notion of having adjoints for 1-morphisms in the bicategory category $h^{(-1)}_2 \cC$ corresponds to the notion of having duals in the monoidal homotopy category of $\cC$.

\begin{appdefinition}
A symmetric monoidal $(\infty,n)$-category $\cC$ is \emph{$m$-dualizable} if it has adjoints for $k$-morphisms, for all $0 \leq k \leq m-1$.  An $n$-dualizable symmetric monoidal $(\infty,n)$-category is called \emph{fully dualizable}.
\end{appdefinition}

\noindent The maximal fully dualizable subcategory of the symmetric monoidal $(\infty,n)$-cate\-gory $\cC$ is denoted $d\cC$. The objects of $d\cC$ are called ``fully dualizable objects".\footnote{This definition differs in a slight technical respect from the definition given in Lurie~\cite{lurie-ch}.  Provided the category $\cC$ is fibrant (as is the case for the example we care about, $\TC$), the two definitions coincide.}  The `space' of fully dualizable objects, denoted $\widetilde{d\cC}$, is the maximal $(\infty,0)$-subcategory of $d\cC$---it is obtained from $d\cC$ simply by discarding all non-invertible $k$-morphisms, for all $k$. 

Because a symmetric monoidal functor takes adjoints to adjoints, and the bordism category $\FrBord_n$ is fully dualizable, a topological field theory with values in $\cC$ must take values in the fully dualizable subcategory $d\cC$.  In fact, field theories are precisely controlled by that subcategory:

\begin{apptheorem}[The Cobordism Hypothesis, framed version {\cite[Thm 2.4.6, Rmk 2.4.8]{lurie-ch}}]
	Let $\FrBord_n$ denote the symmetric monoidal $(\infty,n)$-category of tangentially framed bordisms, and let $\cC$ be a symmetric monoidal $(\infty,n)$-category.  There is an equivalence of $(\infty,n)$-categories
	\begin{equation*}
		\Fun^\otimes(\FrBord_n, \cC) \simeq \widetilde{d\cC}
	\end{equation*} 
from the $(\infty,n)$-category of tensor functors $\FrBord_n \ra \cC$, to the $(\infty,0)$-category $\widetilde{d\cC}$ of fully dualizable objects of $\cC$.  This equivalence takes a local field theory to its value on the standard $n$-framed point.	
\end{apptheorem}

\noindent In particular, there are no non-invertible transformations (or higher transformations) between local field theories, and a tangentially framed local field theory is completely determined, up to equivalence, by its value on a point.

Often, one is interested not only in tangentially framed field theories, but in field theories where the manifolds have a structure weaker than a framing, such as an orientation, or a spin structure, or a string structure.  Many topological structures of interest can be described in the following framework.  Given a map $\xi : X \ra BO(n)$, an $(X,\xi)$-structure on a $k$-manifold $M$, for $k \leq n$, is a map $M \ra X$ and a homotopy from the composite $M \ra X \ra BO(n)$ to the classifying map $M \xra{\tau_M \oplus \RR^{n-k}} BO(n)$---in other words, it is a homotopy lift along $\xi$ of the stabilized tangent map of $M$.  For example, if $X$ is $BSO(n)$, $BSpin(n)$, $BString(n)$, or a point, then we recover the usual notions of orientation, spin structure, string structure, or tangential framing, respectively.  There is an $(\infty,n)$-category of $(X,\xi)$-structured bordisms, denoted $\Bord^{(X,\xi)}_n$ and an $(X,\xi)$-structured local field theory is a symmetric monoidal functor from $\Bord^{(X,\xi)}_n$ to a target symmetric monoidal $(\infty,n)$-category $\cC$.

The framed bordism category $\FrBord_n$ has an $O(n)$-action by simultaneously rotating the $n$-framing at every point of all bordisms.  This provides an action, by precomposition, on the $(\infty,n)$-category of functors $\Fun^\otimes(\FrBord_n,\cC)$.  By the framed version of the cobordism hypothesis above, there is therefore an action (really a homotopy action) of $O(n)$ on the space $\widetilde{d\cC}$ for any $(\infty,n)$-category $\cC$.  Composition with the map $\Omega \xi: \Omega X \ra O(n)$ then gives a homotopy action of $\Omega X$ on $\widetilde{d\cC}$.  The structured version of the cobordism hypothesis classifies structured field theories in terms of this action: 
\begin{apptheorem}[The Cobordism Hypothesis, structured version {\cite[Thm 2.4.18, Thm 2.4.26]{lurie-ch}}] \label{thm:chstruc}
Given a map $\xi: X \ra BO(n)$, with $X$ connected, there is an equivalence of $(\infty,n)$-categories
\[
\Fun^\otimes(\Bord^{(X,\xi)}_n,\cC) \simeq \widetilde{d\cC}^{h\Omega X}
\]
from the $(\infty,n)$-category of $(X,\xi)$-structured field theories with target $\cC$ to the space of $\Omega X$ homotopy fixed points in the fully dualizable objects of $\cC$.
\end{apptheorem}

\backmatter

\bibliographystyle{amsalpha}
\bibliography{dtci-memo}
\end{document}